\documentclass[10pt]{article}

\usepackage{enumerate,xspace}
\usepackage{amsmath,amssymb,wasysym}
\usepackage[all]{xy}
\usepackage{proof}
\usepackage[svgnames]{xcolor}
\usepackage{tikz}
\usepackage[sort,nocompress]{cite}
\usepackage{stmaryrd} %need for special interpretation bracket
\usepackage{mathtools}
\usepackage{latexsym}
\usepackage{dsfont}
\usepackage{multicol}
\usepackage{cmll}
\usepackage{multirow}
\usepackage{longtable}

\usepackage{lscape}
\usepackage{array}

\usepackage{hyperref} %hyperlink package
\hypersetup{
    colorlinks,
    citecolor=red,
    filecolor=red,
    linkcolor=blue,
    urlcolor=red
}

\newtheorem{observation}{Remark}[section]
\newtheorem{lemma}[observation]{Lemma}  %%share counter with remark
\newtheorem{theorem}[observation]{Theorem}
\newtheorem{definition}[observation]{Definition}
\newtheorem{example}[observation]{Example}

\newtheorem{proposition}[observation]{Proposition} 
\newtheorem{corollary}[observation]{Corollary}

\newcommand\Item[1][]{%
      \ifx\relax#1\relax  \item \else \item[#1] \fi
      \abovedisplayskip=0pt\abovedisplayshortskip=0pt~\vspace*{-\baselineskip}
}

\makeatletter

\newdimen\w@dth

\def\setw@dth#1#2{\setbox\z@\hbox{\scriptsize $#1$}\w@dth=\wd\z@
\setbox\@ne\hbox{\scriptsize $#2$}\ifnum\w@dth<\wd\@ne \w@dth=\wd\@ne \fi
\advance\w@dth by 1.2em}

\def\t@^#1_#2{\allowbreak\def\n@one{#1}\def\n@two{#2}\mathrel
{\setw@dth{#1}{#2}
\mathop{\hbox to \w@dth{\rightarrowfill}}\limits
\ifx\n@one\empty\else ^{\box\z@}\fi
\ifx\n@two\empty\else _{\box\@ne}\fi}}
\def\t@@^#1{\@ifnextchar_ {\t@^{#1}}{\t@^{#1}_{}}}

\def\t@left^#1_#2{\def\n@one{#1}\def\n@two{#2}\mathrel{\setw@dth{#1}{#2}
\mathop{\hbox to \w@dth{\leftarrowfill}}\limits
\ifx\n@one\empty\else ^{\box\z@}\fi
\ifx\n@two\empty\else _{\box\@ne}\fi}}
\def\t@@left^#1{\@ifnextchar_ {\t@left^{#1}}{\t@left^{#1}_{}}}

\def\two@^#1_#2{\def\n@one{#1}\def\n@two{#2}\mathrel{\setw@dth{#1}{#2}
\mathop{\vcenter{\hbox to \w@dth{\rightarrowfill}\kern-1.7ex
                 \hbox to \w@dth{\rightarrowfill}}%
       }\limits
\ifx\n@one\empty\else ^{\box\z@}\fi
\ifx\n@two\empty\else _{\box\@ne}\fi}}
\def\tw@@^#1{\@ifnextchar_ {\two@^{#1}}{\two@^{#1}_{}}}

\def\tofr@^#1_#2{\def\n@one{#1}\def\n@two{#2}\mathrel{\setw@dth{#1}{#2}
\mathop{\vcenter{\hbox to \w@dth{\rightarrowfill}\kern-1.7ex
                 \hbox to \w@dth{\leftarrowfill}}%
       }\limits
\ifx\n@one\empty\else ^{\box\z@}\fi
\ifx\n@two\empty\else _{\box\@ne}\fi}}
\def\t@fr@^#1{\@ifnextchar_ {\tofr@^{#1}}{\tofr@^{#1}_{}}}

\newdimen\W@dth
\def\setW@dth#1#2{\setbox\z@\hbox{$#1$}\W@dth=\wd\z@
\setbox\@ne\hbox{$#2$}\ifnum\W@dth<\wd\@ne \W@dth=\wd\@ne \fi
\advance\W@dth by 1.2em}

\def\T@^#1_#2{\allowbreak\def\N@one{#1}\def\N@two{#2}\mathrel
{\setW@dth{#1}{#2}
\mathop{\hbox to \W@dth{\rightarrowfill}}\limits
\ifx\N@one\empty\else ^{\box\z@}\fi
\ifx\N@two\empty\else _{\box\@ne}\fi}}
\def\T@@^#1{\@ifnextchar_ {\T@^{#1}}{\T@^{#1}_{}}}

\def\T@left^#1_#2{\def\N@one{#1}\def\N@two{#2}\mathrel{\setW@dth{#1}{#2}
\mathop{\hbox to \W@dth{\leftarrowfill}}\limits
\ifx\N@one\empty\else ^{\box\z@}\fi
\ifx\N@two\empty\else _{\box\@ne}\fi}}
\def\T@@left^#1{\@ifnextchar_ {\T@left^{#1}}{\T@left^{#1}_{}}}

\def\Tofr@^#1_#2{\def\N@one{#1}\def\N@two{#2}\mathrel{\setW@dth{#1}{#2}
\mathop{\vcenter{\hbox to \W@dth{\rightarrowfill}\kern-1.7ex
                 \hbox to \W@dth{\leftarrowfill}}%
       }\limits
\ifx\N@one\empty\else ^{\box\z@}\fi
\ifx\N@two\empty\else _{\box\@ne}\fi}}
\def\T@fr@^#1{\@ifnextchar_ {\Tofr@^{#1}}{\Tofr@^{#1}_{}}}

\def\Two@^#1_#2{\def\N@one{#1}\def\N@two{#2}\mathrel{\setW@dth{#1}{#2}
\mathop{\vcenter{\hbox to \W@dth{\rightarrowfill}\kern-1.7ex
                 \hbox to \W@dth{\rightarrowfill}}%
       }\limits
\ifx\N@one\empty\else ^{\box\z@}\fi
\ifx\N@two\empty\else _{\box\@ne}\fi}}
\def\Tw@@^#1{\@ifnextchar_ {\Two@^{#1}}{\Two@^{#1}_{}}}

\def\to{\@ifnextchar^ {\t@@}{\t@@^{}}}
\def\from{\@ifnextchar^ {\t@@left}{\t@@left^{}}}
\def\tofro{\@ifnextchar^ {\t@fr@}{\t@fr@^{}}}
\def\To{\@ifnextchar^ {\T@@}{\T@@^{}}}
\def\From{\@ifnextchar^ {\T@@left}{\T@@left^{}}}
\def\Two{\@ifnextchar^ {\Tw@@}{\Tw@@^{}}}
\def\Tofro{\@ifnextchar^ {\T@fr@}{\T@fr@^{}}}

\def\p{\partial}

\def\s{\mathfrak{s}}

\def\t{\mathfrak t}

\makeatother

\title{Cartesian Differential Comonads and New Models of Cartesian Differential Categories}
\author{Sacha Ikonicoff and Jean-Simon Pacaud Lemay}
\begin{document}
\allowdisplaybreaks

\maketitle
%Les catégories différentielles cartésiennes sont équipées d'un combinateur différentiel qui formalise l'opération de dérivation du calcul différentiel à plusieurs variables, et fournissent aussi la sémantique du lambda-calcul différentiel. Une source importante d'exemple de catégories différentielles cartésienne provient des catégories coKleisli des comonades structurelles des categories différentielles, ce dernier concept fournissant la sémantique catégorique de la logique linéaire différentielle. Dans cet article, nous généralisons cette construction en introduisant la notion de comonade différentielle cartésienne, qui sont précisemment les comonades dont la catégorie de coKleisli est une catégorie différentielle cartésienne, ce qui offre une plus large gamme d'exemples. Nous construisons ainsi de nouveaux exemples de catégories différentielles cartésiennes provenant de comonades différentielles cartésiennes faisant intervenir les séries formelles, les algèbres à puissances divisées, et les algèbres de Zinbiel.
\begin{abstract}  Cartesian differential categories come equipped with a differential combinator that formalizes the derivative from multi-variable differential calculus, and also provide the categorical semantics of the differential $\lambda$-calculus. An important source of examples of Cartesian differential categories are the coKleisli categories of the comonads of differential categories, where the latter concept provides the categorical semantics of differential linear logic. In this paper, we generalize this construction by introducing Cartesian differential comonads, which are precisely the comonads whose coKleisli categories are Cartesian differential categories, and thus allows for a wider variety of examples of Cartesian differential categories. As such, we construct new examples of Cartesian differential categories from Cartesian differential comonads based on power series, divided power algebras, and Zinbiel algebras. 
\end{abstract}

\noindent \small \textbf{Acknowledgements.} The authors would like to thank Kristine Bauer and Robin Cockett for useful discussions. The authors would also like to thank Emily Riehl and the anonymous reviewers for editorial comments and suggestions. For this research, the first author was financially supported by a PIMS--CNRS Postdoctoral Fellowship, and the second author was financially supported by a NSERC Postdoctoral Fellowship - Award \#: 456414649. 

\tableofcontents

%%%%%%%%%%%%%%%%%%%%%%%%%%%%%%%%%%%%%%%%%%%%%%%%%%%%%%%%%%%%%%%%%%%%%
\section{Introduction}

Cartesian differential categories, introduced by Blute, Cockett, and Seely in \cite{blute2009cartesian}, formalize the theory of multivariable differential calculus by axiomatizing the (total) derivative, and also provide the categorical semantics of the differential $\lambda$-calculus, as introduced by Ehrhard and Regnier in \cite{ehrhard2003differential}. Briefly, a Cartesian differential category (Definition \ref{cartdiffdef}) is a category with finite products such that each homset is a commutative monoid, which allows for zero maps and sums of maps (Definition \ref{CLACdef}), and equipped with a differential combinator $\mathsf{D}$, which for every map ${f: A \to B}$ produces its derivatives $\mathsf{D}[f]: A \times A \to B$. The differential combinator satisfies seven axioms, known as \textbf{[CD.1]} to \textbf{[CD.7]}, which formalize the basic identities of the (total) derivative from multi-variable differential calculus such as the chain rule, linearity in the vector argument, symmetry of the partial derivatives, etc. Two main examples of Cartesian differential categories are the category of Euclidean spaces and real smooth functions between them (Example \ref{ex:smooth}), and the Lawvere Theory of polynomials over a commutative (semi)ring (Example \ref{ex:CDCPOLY}). An important class of examples of Cartesian differential categories, especially for this paper, are the coKleilsi categories of the comonads of differential categories \cite[Propostion 3.2.1]{blute2009cartesian}. 

Differential categories, introduced by Blute, Cockett, and Seely in \cite{blute2006differential}, provide the algebraic foundations of differentiation and the categorical semantics of differential linear logic \cite{ehrhard2017introduction}. Briefly, a differential category (Example \ref{ex:diffcat}) is a symmetric monoidal category with a comonad $\oc$, with comonad structure maps $\delta_A: \oc(A) \to \oc\oc(A)$ and ${\varepsilon_A: \oc(A) \to A}$, such that for each object $A$, $\oc(A)$ is a cocommutative comonoid with comultiplication ${\Delta_A: \oc(A) \to \oc(A) \otimes \oc(A)}$ and counit $e_A: \oc(A) \to I$, and equipped with a deriving transformation, which is a natural transformation ${\mathsf{d}_A: \oc(A) \otimes A \to \oc(A)}$. The deriving transformation satisfies five axioms, this time called \textbf{[d.1]} to \textbf{[d.5]}, which formalize basic identities of differentiation such as the chain rule and the product rule. In the opposite category of a differential category, called a codifferential category, the deriving transformation is a derivation in the classical algebra sense. Examples of differential categories include the opposite category of the category of vector spaces over a field where $\oc$ is induced by the free symmetric algebra \cite{blute2006differential,Blute2019}, as well as the opposite category of the category of real vector spaces where $\oc$ is instead induced by free $\mathcal{C}^\infty$-rings \cite{cruttwell2019integral}. 

In a differential category, a smooth map from $A$ to $B$ is a map of type $\oc(A) \to B$. In other words, the (infinitely) differentiable maps are precisely the coKleisli maps. The interpretation of coKleisli maps as smooth can be made precise when the differential category has finite (bi)products where one uses the deriving transformation to define a differential combinator on the coKleisli category. Briefly, for a coKleisli map $f: \oc(A) \to B$ (which is a map of type $A \to B$ in the coKleisli category), its derivative $\mathsf{D}[f]: \oc(A\times A) \to B$ (which is a map of type $A \times A \to B$ in the coKleisli category) is defined as the following composite (in the base category): 
\[ \mathsf{D}[f] :=  \xymatrixcolsep{3pc}\xymatrix{\oc(A \times A) \ar[r]^-{\Delta_{A \times A}} & \oc(A \times A) \otimes \oc(A \times A) \ar[r]^-{\oc(\pi_0) \otimes \oc(\pi_1)} &  \oc(A) \otimes \oc(A) \ar[r]^-{1_{\oc(A)} \otimes \varepsilon_A} & \oc(A) \otimes A \ar[r]^-{\mathsf{d}_A} & \oc(A) \ar[r]^-{ f } & B 
 } \]
where $\pi_i$ are the product projection maps. One then uses the five deriving transformations axioms \textbf{[d.1]} to \textbf{[d.5]} to prove that $\mathsf{D}$ satisfies the seven differential combinator axioms \textbf{[CD.1]} to \textbf{[CD.7]}. Thus, for a differential category with finite (bi)products, its coKleisli category is a Cartesian differential category. For the examples where $\oc$ is the free symmetric algebra or given by free $\mathcal{C}^\infty$-rings, the resulting coKleisli category can respectively be interpreted as the category of polynomials or real smooth functions over possibly infinite variables (but that will still only depend on a finite number of them), of which the Lawvere theory of polynomials or category of real smooth functions is a sub-Cartesian differential category. 

Taking another look at the construction of the differential combinator for the coKleisli category, if we define the natural transformation $\partial_A: \oc(A \times A) \to \oc(A)$ as the following composite: 
\[ \partial_A := \xymatrixcolsep{3.75pc}\xymatrix{\oc(A \times A) \ar[r]^-{\Delta_{A \times A}} & \oc(A \times A) \otimes \oc(A \times A) \ar[r]^-{\oc(\pi_0) \otimes \oc(\pi_1)} & \oc(A) \otimes \oc(A) \ar[r]^-{1_{\oc(A)} \otimes \varepsilon_A} & \oc(A) \otimes A \ar[r]^-{\mathsf{d}_A} & \oc(A)
 } \]
then the differential combinator is simply defined by precomposing a coKleisli map $f: \oc(A) \to B$ with $\partial$: 
  \[ \mathsf{D}[f] := \xymatrixcolsep{5pc}\xymatrix{\oc(A \times A) \ar[r]^-{\partial_A} & \oc(A) \ar[r]^-{f} & B} \]
It is important to stress that this is the composition in the base category and not the composition in the coKleisli category. Thus, the properties of the differential combinator $\mathsf{D}$ in the coKleisli category are fully captured by the properties of the natural transformation $\partial$ in the base category, which in turn are a result of the axioms of the deriving transformation $\mathsf{d}$. However, observe that the type of $\partial_A: \oc(A \times A) \to \oc(A)$ does not involve any monoidal structure. In fact, if one starts with a comonad whose coKleisli category is a Cartesian differential category, it is always possible to construct $\partial$, and to show that $\mathsf{D}[-] = - \circ \partial$, but it is not always possible to extract a monoidal structure on the base category. Thus, if one's goal is simply to build Cartesian differential categories from coKleisli categories, then a monoidal structure $\otimes$ or a deriving transformation $\mathsf{d}$, or even a comonoid structure $\Delta$ and $e$, are not always necessary. Therefore, the objective of this paper is to precisely characterize the comonads whose coKleisli categories are Cartesian differential categories. To this end, in this paper we introduce the novel notion of a Cartesian differential comonad.

Cartesian differential comonads are precisely the comonads whose coKleisli categories are Cartesian differential categories. Briefly, a Cartesian differential comonad is a comonad $\oc$ on a category with finite biproducts equipped with a differential combinator transformation, which is a natural transformation $\partial_A: \oc(A \times A) \to \oc(A)$ which satisfies six axioms called \textbf{[dc.1]} to \textbf{[dc.6]} (Definition \ref{def:cdcomonad}). The axioms of a differential combinator transformation are analogues of the axioms of a differential combinator. Thus, the coKleisli category of a Cartesian differential comonad is a Cartesian differential category where the differential combinator is defined by precomposition with the differential combinator transformation (Theorem \ref{thm1}). This is proven by reasonably straightforward calculations, but one must be careful when translating back and forth between the base category and the coKleisli category. Conversely, a comonad on a category with finite biproduct whose coKleisli category is a Cartesian differential category is in fact a Cartesian differential comonad, where the differential combinator transformation is the derivative of the identity map ${1_{\oc(A)}: \oc(A) \to \oc(A)}$ seen as a coKleisli map $A \to \oc(A)$ (Proposition \ref{prop1}). Using this, since we already know that the coKleisli category of a differential category is a Cartesian differential category, it immediately follows that the comonad of a differential category is a Cartesian differential comonad, where the differential combinator transformation is precisely the one defined above. Therefore, Cartesian differential comonads and differential combinator transformations are indeed generalizations of differential categories and deriving transformations. However, Cartesian differential comonads are a strict generalization since, as mentioned, they can be defined without the need of a monoidal structure. A very simple separating example is the identity comonad on any category with finite biproducts, where the differential combinator transformation is simply the second projection map (Example \ref{ex:identity}). While this example is trivial, it recaptures the fact that any category with finite biproducts is a Cartesian differential category and this example clearly works without any extra monoidal structure, and thus is not a differential category example. Therefore, Cartesian differential comonads allow for a wider variety of examples of Cartesian differential categories. As such, in this paper we present three new interesting examples of Cartesian differential comonads, which are not differential categories, and their induced Cartesian differential categories. These three examples are respectively based on formal power series, divided power algebras, and Zinbiel algebras. It is worth mentioning that these new examples arise more naturally as coCartesian differential monads (Example \ref{ex:CDM}), the dual notion of Cartesian differential comonads, and thus it is the opposite of the Kleisli category which is a Cartesian differential category. 

The first example (Section \ref{sec:PWex}) is based on reduced power series. Recall that a formal power series is said to be reduced if it has no constant/degree 0 term. While the composition of arbitrary multivariable formal power series is not always well defined, due to their constant terms, the composition of reduced multivariable power series is always well-defined \cite[Section 4.1]{brewer2014algebraic}, and so we may construct categories of reduced power series. Also, it is well known that power series are always and easily differentiable, similarly to polynomials, and that the derivative of a reduced multivariable power series is again reduced. Motivated by capturing power series differentiation, we show that the free reduced power series algebra monad \cite[Section 1.4.3]{Fresse98} is a coCartesian differential monad (Corollary \ref{cor:POW}) whose monad structure is based on reduced power series composition (Lemma \ref{lem:powmonad}) and whose differential combinator transformation is induced by standard power series differentiation (Proposition \ref{prop:powpartial}). Furthermore, the Lawvere theory of reduced power series (Example \ref{ex:CDCPOW}) is a sub-Cartesian differential category of the opposite category of the resulting Kleisli category. 

The second new example (Section \ref{secpuisdiv}) is based on divided power algebras. Divided power algebras, defined by Cartan \cite{cartan1954puissances}, are commutative non-unital associative algebras equipped with additional operations $(-)^{[n]}$ for all strictly positive integer $n$, satisfying some relations (Definition \ref{defipuisdiv}). In characteristic $0$, divided power algebras correspond precisely to commutative non-unital associative algebras. In positive characteristics, however, the two notions diverge. There exist free divided power algebras and we show that the free divided power algebra monad \cite[Section 10, Th{\'e}or{\`e}me 1 and 2]{roby65} is a coCartesian differential monad (Corollary \ref{cor:DIV}). Free divided power algebras correspond to the algebra of reduced divided power polynomials. Thus the differential combinator transformation of this example (Proposition \ref{Gammacomb}) captures differentiating divided power polynomials \cite{keigher2000}. In particular, the Lawvere theory of reduced divided power polynomials (Example \ref{ex:CDCdiv}) is a sub-Cartesian differential category of the opposite category of the Kleisli category of the free divided power algebra monad.  

The third new example (Section \ref{sec:ZAex}), and perhaps the most exotic example in this paper, is based on Zinbiel algebras. The notion of Zinbiel algebra was introduced by Loday \cite{loday95} and also further studied by Dokas \cite{dokas09}. A Zinbiel algebra is a vector space $A$ endowed with a non-associative and non-commutative bilinear operation $<$. Using the Zinbiel product, every Zinbiel algebra can be turned into a commutative non-unital associative algebra. The underlying vector space of free Zinbiel algebras is the same as the underlying vector space of the non-unital tensor algebra. Therefore, free Zinbiel algebras are spanned by (non-empty) associative words and equipped with a product $<$ (which is sometimes referred to as the semi-shuffle product). The resulting commutative associative algebra is then precisely the non-unital shuffle algebra over $V$. We show that the free Zinbiel algebra monad \cite[Proposition 1.8]{loday95} is a coCartesian differential monad (Corollary \ref{cor:ZIN}) whose differential combinator transformation (Proposition \ref{difcombzin}) corresponds to differentiating non-commutative polynomials with respect to the Zinbiel product. The resulting Cartesian differential category can be understood as the category of reduced non-commutative polynomials where the composition is defined using the Zinbiel product, which we simply call Zinbiel polynomials. As such, the Lawvere theory of Zinbiel polynomials is a new exotic example of a Cartesian differential category. It is worth mentioning that the shuffle algebra has been previously studied as an example of another generalization of differential categories in \cite{bagnol2016shuffle}, but not from the point of view of Zinbiel algebras. 

An important class of maps in a Cartesian differential category are the $\mathsf{D}$-linear maps (Definition \ref{linmapdef}), also often simply called linear maps \cite{blute2009cartesian}. A map $f: A \to B$ is $\mathsf{D}$-linear if its derivative $\mathsf{D}[f]: A \times A \to B$ is equal to $f$ evaluated in its second argument, that is, $\mathsf{D}[f] = f \circ \pi_1$ (where $\pi_1$ is the projection map of the \emph{second} argument). A $\mathsf{D}$-linear map should be thought of as being of degree 1, and thus does not have any higher-order derivative. Thus, in many examples, $\mathsf{D}$-linearity often coincides with the classical notion of linearity. For example, in the Cartesian differential category of real smooth functions, a smooth function is $\mathsf{D}$-linear if and only if it is $\mathbb{R}$-linear. For a Cartesian differential comonad, every map of the base category provides a $\mathsf{D}$-linear map in the coKleisli category. However, it is not necessarily the case that the base category is isomorphic to the subcategory of $\mathsf{D}$-linear maps of the coKleisli category. Indeed, a simple example of such a case is the trivial Cartesian differential comonad which maps every object to the zero object and thus every coKleisli map is a zero map. Clearly, if the base category is non-trivial it will not be equivalent to the subcategory of $\mathsf{D}$-linear maps. Instead, it is possible to provide necessary and sufficient conditions for the base category to be isomorphic to the subcategory of $\mathsf{D}$-linear maps of the coKleisli category. It turns out that this is precisely the case when the Cartesian differential comonad comes equipped with a $\mathsf{D}$-linear unit, which is a natural transformation ${\eta_A: A \to \oc(A)}$ satisfying two axioms \textbf{[du.1]} and \textbf{[du.2]} (Definition \ref{def:Dunit}). If it exists, a $\mathsf{D}$-linear unit is unique and it is equivalent to an isomorphism between the base category and the subcategory of $\mathsf{D}$-linear maps of the coKleisli category (Proposition \ref{etaFlem1}). In the context of differential categories, specifically in categorical models of differential linear logic, the $\mathsf{D}$-linear unit is precisely the codereliction \cite{blute2006differential,Blute2019,ehrhard2017introduction}. The Cartesian differential comonads based on power series, or divided power algebras, or Zinbiel algebras all come equipped with $\mathsf{D}$-linear units. 

In \cite{blute2015cartesian}, Blute, Cockett, and Seely give a characterization of the Cartesian differential categories which are the coKleisli categories of differential categories. Generalizing their approach, it is also possible to precisely characterize the Cartesian differential categories which are the coKleisli categories of Cartesian differential comonads (Section \ref{sec:abstract}). To this end, we must work with abstract coKleisli categories (Definition \ref{def:abstract}), which gives a description of coKleisli categories without starting from a comonad. Abstract coKleisli categories are the dual notion of F\"{u}hrmann's thunk-force-categories \cite{fuhrmann1999direct}, which instead do the same for Kleisli categories. Every abstract coKleisli category is canonically the coKleisli category of a comonad on a certain subcategory (Lemma \ref{lem:ep-com}), and conversely, the coKleisli category of any comonad is an abstract coKleisli category (Lemma \ref{cokleisliabstractlem}). In this paper, we introduce Cartesian differential abstract coKleisli categories (Definition \ref{def:abCDC}) which are abstract coKleisli categories that are also Cartesian differential categories such that the differential combinator and the abstract coKleisli structure are compatible. Every Cartesian differential abstract coKleisli category is canonically the coKleisli category of a Cartesian differential comonad over a certain subcategory of $\mathsf{D}$-linear maps (Proposition \ref{propab1}), and conversely, the coKleisli category of a Cartesian differential comonad is a Cartesian differential abstract category (Proposition \ref{propabcok}).  

In conclusion, Cartesian differential comonads give a minimum general construction to build coKleisli categories which are Cartesian differential categories. The theory of Cartesian differential comonads also highlights the interaction between the coKleisli structure and the differential combinator. While Cartesian differential comonads recapture some of the notions of differential categories, they are more general. Therefore, Cartesian differential comonads open the door to a variety of new, interesting, and exotic examples of Cartesian differential categories. New examples will be particularly important and of interest, especially since applications of Cartesian differential categories keep being developed, especially in the fields of machine learning and automatic differentiation.

\paragraph{Conventions:} In an arbitrary category, we use the classical notation for composition as opposed to diagrammatic order which was used in other papers on differential categories (such as in \cite{blute2009cartesian,lemay2018tangent} for example). The composite map ${g \circ f: A \to C}$ is the map that first does $f: A\to B$ then $g: B \to C$. We denote identity maps as ${1_A: A \to A}$. 

\section{Cartesian Differential Categories}\label{sec:CDC-background}

In this background section we quickly review Cartesian differential categories \cite{blute2009cartesian}. 

The underlying structure of a Cartesian differential category is that of a Cartesian left additive category, which in particular allows one to have zero maps and sums of maps, while also allowing for maps which do not preserve said sums or zeros. Maps which do preserve the additive structure are called \emph{additive} maps. Categorically speaking, a left additive category is a category which is \emph{skew}-enriched over the category of commutative monoids and monoid morphisms \cite{garner2020cartesian}. Then a Cartesian left additive category is a left additive category with finite products such that the product structure is compatible with the commutative monoid structure, that is, the projection maps are additive. Note that since we are working with commutative monoids, we do not assume that our Cartesian left additive categories necessarily come equipped with additive inverses, or in other words negatives. For a category with (chosen) finite products we denote the (chosen) terminal object as $\top$, the binary product of objects $A$ and $B$ by $A \times B$ with projection maps $\pi_0: A \times B \to A$ and $\pi_1: A \times B \to B$ and pairing operation $\langle -, - \rangle$, so that for maps $f: C \to A$ and $g: C \to B$, $\langle f,g \rangle: C \to A \times B$ is the unique map such that $\pi_0 \circ \langle f, g \rangle = f$ and $\pi_1 \circ \langle f, g \rangle = g$. As such, the product of maps $h: A \to B$ and $k: C \to D$ is the map $h \times k: A \times C \to B \times D$ defined as $h \times k = \langle h \circ \pi_0, k \circ \pi_1 \rangle$. 

\begin{definition}\label{CLACdef} A \textbf{left additive category} \cite[Definition 1.1.1]{blute2009cartesian} is a category $\mathbb{X}$ such that each hom-set $\mathbb{X}(A,B)$ is a commutative monoid, with binary addition $+: \mathbb{X}(A,B) \times \mathbb{X}(A,B) \to \mathbb{X}(A,B)$, $(f,g) \mapsto f +g$ and zero $0 \in \mathbb{X}(A,B)$, and such that pre-composition preserves the additive structure, that is, for any maps $f: A \to B$, $g: A \to B$, and $x: A^\prime \to A$, the following equality holds:
\begin{align*}
(f+g) \circ x = f \circ x + g \circ x && 0 \circ x = 0
\end{align*} 
A map $f: A \to B$ is said to be \textbf{additive} \cite[Definition 1.1.1]{blute2009cartesian} if post-composition by $f$ preserves the additive structure, that is, for any maps $x: A^\prime \to A$ and $y: A^\prime \to A$, the following equality holds:
\begin{align*}
f \circ (x + y) = f \circ x + f \circ y && f \circ 0 = 0
\end{align*}
 A \textbf{Cartesian left additive category} \cite[Definition 2.3]{lemay2018tangent} is a left additive category $\mathbb{X}$ which has finite products and such that all the projection maps $\pi_0: A \times B \to A$ and $\pi_1: A \times B \to B$ are additive. 
\end{definition}

We note that the definition of a Cartesian left additive category presented here is not precisely that given in \cite[Definition 1.2.1]{blute2009cartesian}, but was shown to be equivalent in \cite[Lemma 2.4]{lemay2018tangent}. Also note that in a Cartesian left additive category, the unique map to the terminal object $\top$ is the zero map ${0: A \to \top}$. Here are now some important maps for Cartesian differential categories that can be defined in any Cartesian left additive category: 

\begin{definition}\label{CLACmapsdef} In a Cartesian left additive category $\mathbb{X}$: 
\begin{enumerate}[{\em (i)}]
\item \label{injdef} For each pair of objects $A$ and $B$, define the \textbf{injection maps} $\iota_0: A \to A \times B$ and $\iota_1: B \to A \times B$ respectively as $\iota_0 := \langle 1_A, 0 \rangle$ and $\iota_1 := \langle 0, 1_B \rangle$
\item \label{nabladef} For each object $A$, define the \textbf{sum map} $\nabla_A: A \times A \to A$ as $\nabla_A := \pi_0 + \pi_1$. 
\item \label{elldef} For each object $A$, define the \textbf{lifting map} $\ell_A: A \times A \to (A \times A) \times (A \times A)$ as follows $\ell := \iota_0 \times \iota_1$. 
\item \label{cdef} For each object $A$ define the \textbf{interchange map} $c_A: (A \times A) \times (A \times A) \to (A \times A) \times (A \times A)$ as follows $c_A : = \left \langle \pi_0 \times \pi_0, \pi_1 \times \pi_1 \right \rangle$. 
\end{enumerate}
\end{definition}

It is important to note that while $c$ is natural in the expected sense, the injection maps $\iota_j$, the sum map $\nabla$, and the lifting map $\ell$ are not natural transformations. Instead, they are natural only with respect to additive maps. In particular, since the injection maps are not natural map for arbitrary maps, it follows that these injection maps do not make the product a coproduct, and therefore not a biproduct. However, the biproduct identities still hold in a Cartesian left additive category in the sense that the following equalities hold: 
\begin{align*}
\pi_0 \circ \iota_0 = 1_A && \pi_0 \circ \iota_1 = 0 && \pi_1 \circ \iota_0 = 0 && \pi_1 \circ \iota_1 = 1_B && \iota_0 \circ \pi_0 + \iota_1 \circ \pi_1 = 1_{A \times B}
\end{align*}
With all this said, it turns out that a category with finite biproducts is precisely a Cartesian left additive category where every map is additive \cite[Example 2.3.(ii)]{garner2020cartesian}. In that case, note the injection maps of Definition \ref{CLACmapsdef}.(\ref{injdef}) are precisely the injection maps of the coproduct, while the sum map of Definition \ref{CLACmapsdef}.(\ref{nabladef}) is the co-diagonal map of the coproduct. 

Cartesian differential categories are Cartesian left additive categories which come equipped with a differential combinator, which in turn is axiomatized by the basic properties of the directional derivative from multivariable differential calculus. There are various equivalent ways of expressing the axioms of a Cartesian differential category. Here we have chosen the one found in \cite[Definition 2.6]{lemay2018tangent} (using the notation for Cartesian left additive categories introduced above). It is important to notice that in the following definition, unlike in the original paper \cite{blute2009cartesian} and other early works on Cartesian differential categories, we use the convention used in the more recent works where the linear argument of $\mathsf{D}[f]$ is its second argument rather than its first argument.  

\begin{definition}\label{cartdiffdef} A \textbf{Cartesian differential category} \cite[Definition 2.1.1]{blute2009cartesian} is a Cartesian left additive category $\mathbb{X}$ equipped with a \textbf{differential combinator} $\mathsf{D}$, which is a family of operators:
\begin{align*} \mathsf{D}: \mathbb{X}(A,B) \to \mathbb{X}(A \times A,B) && \frac{f : A \to B}{\mathsf{D}[f]: A \times A \to B}
\end{align*}
and such that the following seven axioms hold:  
\begin{enumerate}[{\bf [CD.1]}]
\item \label{CDCax1} $\mathsf{D}[f+g] = \mathsf{D}[f] + \mathsf{D}[g]$ and $\mathsf{D}[0] = 0$ 
\item \label{CDCax2} $\mathsf{D}[f] \circ \left(1_A \times \nabla_A \right) = \mathsf{D}[f] \circ (1_A \times \pi_0) + \mathsf{D}[f] \circ (1_A \times \pi_1)$ and $\mathsf{D}[f] \circ \iota_0 = 0$
\item \label{CDCax3} $\mathsf{D}[1_A]=\pi_1$, $\mathsf{D}[\pi_0] = \pi_0 \circ \pi_1$, and $\mathsf{D}[\pi_1] = \pi_0 \circ \pi_1$
\item \label{CDCax4} $\mathsf{D}[\left\langle f,g \right \rangle] = \left \langle  \mathsf{D}[f], \mathsf{D}[g] \right \rangle$
\item \label{CDCax5} $\mathsf{D}[g \circ f] = \mathsf{D}[g] \circ \langle f \circ \pi_0, \mathsf{D}[f] \rangle$
\item \label{CDCax6} $\mathsf{D}\left[\mathsf{D}[f] \right] \circ \ell_A =  \mathsf{D}[f]$
\item \label{CDCax7} $\mathsf{D}\left[\mathsf{D}[f] \right] \circ c_A = \mathsf{D}\left[\mathsf{D}[f] \right]$
\end{enumerate}
For a map $f: A \to B$, $\mathsf{D}[f]: A \times A \to B$ is called the derivative of $f$. 
\end{definition}

A discussion on the intuition for the differential combinator axioms can be found in \cite[Remark 2.1.3]{blute2009cartesian}. It is also worth mentioning that there is a sound and complete term logic for Cartesian differential categories \cite[Section 4]{blute2009cartesian}.  

An important class of maps in a Cartesian differential category is the class of linear maps. In this paper, however, we borrow the terminology from \cite{garner2020cartesian} and will instead call them $\mathsf{D}$-linear maps. This terminology will help distinguish between the classical notion of linearity from commutative algebra and the Cartesian differential category notion of linearity.  

\begin{definition}\label{linmapdef} In a Cartesian differential category $\mathbb{X}$ with differential combinator $\mathsf{D}$, a map $f$ is said to be \textbf{$\mathsf{D}$-linear} \cite[Definition 2.2.1]{blute2009cartesian} if $\mathsf{D}[f]= f \circ \pi_1$. Define the subcategory of linear maps $\mathsf{D}\text{-}\mathsf{lin}[\mathbb{X}]$ to be the category whose objects are the same as $\mathbb{X}$ and whose maps are $\mathsf{D}$-linear in $\mathbb{X}$, and let $\mathsf{U}: \mathsf{D}\text{-}\mathsf{lin}[\mathbb{X}] \to \mathbb{X}$ be the obvious forgetful functor. 
\end{definition}

By \cite[Lemma 2.2.2]{blute2009cartesian}, every $\mathsf{D}$-linear is additive, and therefore it follows that $\mathsf{D}\text{-}\mathsf{lin}[\mathbb{X}]$ has finite biproducts, and is thus also a Cartesian left additive category (where every map is additive) such that the forgetful functor ${\mathsf{U}: \mathsf{D}\text{-}\mathsf{lin}[\mathbb{X}] \to \mathbb{X}}$ preserves the Cartesian left additive structure strictly. It is important to note that although additive and linear maps often coincide in many examples of Cartesian differential category, in an arbitrary Cartesian differential category, not every additive map is necessarily linear. Here are now some useful properties of linear maps: 

\begin{lemma}\label{linlem} \cite[Lemma 2.2.2, Corollary 2.2.3]{blute2009cartesian} In a Cartesian differential category $\mathbb{X}$ with differential combinator $\mathsf{D}$:
\begin{enumerate}[{\em (i)}]
\item \label{linlem.add} If $f: A \to B$ is $\mathsf{D}$-linear then $f$ is additive;
\item\label{linlemimportant1} For any map $f: A \to B$, define $\mathsf{L}[f]: A \to A$ (called the linearization of $f$ \cite[Definition 3.1]{cockett2020linearizing}) as the following composite: 
  \[ \mathsf{L}[f] := \xymatrixcolsep{5pc}\xymatrix{A \ar[r]^-{\iota_1} & A \times A \ar[r]^-{\mathsf{D}[f]} & B   
  } \]
Then $\mathsf{L}[f]$ is $\mathsf{D}$-linear. 
\item\label{linlemimportant2} A map $f: A \to B$ is $\mathsf{D}$-linear if and only if $f = \mathsf{L}[f]$.
\item  \label{linlem.pre} If $f: A \to B$ is $\mathsf{D}$-linear then for every map $g: B \to C$, $\mathsf{D}[g \circ f] = \mathsf{D}[g] \circ (f \times f)$;
\item  \label{linlem.post} If $g: B \to C$ is $\mathsf{D}$-linear then for every map $f: A \to B$, $\mathsf{D}[g \circ f] = g \circ \mathsf{D}[f]$. 
\end{enumerate}
\end{lemma}

We conclude this section with some examples of well-known Cartesian differential categories and their $\mathsf{D}$-linear maps. The first three examples are based on the standard notions of differentiating linear functions, polynomials, and smooth functions respectively. 

\begin{example}\label{ex:CDCbiproduct} \normalfont Any category $\mathbb{X}$ with finite biproduct is a Cartesian differential category where the differential combinator is defined by precomposing with the second projection map: $ \mathsf{D}[f] = f \circ \pi_1$. 
%  \[ \mathsf{D}[f] := \xymatrixcolsep{5pc}\xymatrix{A \times A \ar[r]^-{\pi_1} & A \ar[r]^-{f} & B    } \]
In this case, every map is $\mathsf{D}$-linear by definition and so $\mathsf{D}\text{-}\mathsf{lin}[\mathbb{X}] =  \mathbb{X}$. %In fact, a Cartesian differential category in which every map is $\mathsf{D}$-linear is a category with finite biproducts\footnote{That said, it is important to point out that there are examples of categories with finite biproducts with a differential combinator such that not every map is $\mathsf{D}$-linear.}. 
As a particular example, let $\mathbb{F}$ be a field and let $\mathbb{F}\text{-}\mathsf{VEC}$ be the category of $\mathbb{F}$-vector spaces and $\mathbb{F}$-linear maps between them. Then $\mathbb{F}\text{-}\mathsf{VEC}$ is a Cartesian differential category where for an $\mathbb{F}$-linear map ${f: V \to W}$, its derivative $\mathsf{D}[f]: V \times V \to W$ is defined as $\mathsf{D}[f](v,w) = f(w)$. 
\end{example}

\begin{example} \normalfont \label{ex:CDCPOLY} Let $\mathbb{F}$ be a field. Define the category $\mathbb{F}\text{-}\mathsf{POLY}$ whose object are $n \in \mathbb{N}$, where a map ${P: n \to m}$ is a $m$-tuple of polynomials in $n$ variables, that is, $P = \langle p_1(\vec x), \hdots, p_m(\vec x) \rangle$ with $p_i(\vec x) \in \mathbb{F}[x_1, \hdots, x_n]$. $\mathbb{F}\text{-}\mathsf{POLY}$ is a Cartesian differential category where the differential combinator is given by the standard differentiation of polynomials, that is, for a map ${P: n \to m}$, with $P = \langle p_1(\vec x), \hdots, p_m(\vec x) \rangle$, its derivative $\mathsf{D}[P]: n \times n \to m$ is defined as the tuple of the sum of the partial derivatives of the polynomials $p_i(\vec x)$:
\begin{align*}
\mathsf{D}[P](\vec x, \vec y) := \left( \sum \limits^n_{i=1} \frac{\partial p_1(\vec x)}{\partial x_i} y_i, \hdots, \sum \limits^n_{i=1} \frac{\partial p_n(\vec x)}{\partial x_i} y_i \right) && \sum \limits^n_{i=1} \frac{\partial p_j (\vec x)}{\partial x_i} y_i \in \mathbb{F}[x_1, \hdots, x_n, y_1, \hdots, y_n]
\end{align*} 
A map $P: n \to m$ is $\mathsf{D}$-linear if it of the form: 
\begin{align*}
P = \left \langle \sum \limits^{n}_{i=0} r_{i,1}x_{i}, \hdots, \sum \limits^{n}_{i=0} r_{i,m}x_{i} \right \rangle && r_{i,j} \in \mathbb{F} 
\end{align*}
In other words, $P =  \langle p_1(\vec x), \hdots, p_m(\vec x) \rangle$ is $\mathsf{D}$-linear if and only if each $p_i(\vec x)$ induces an $\mathbb{F}$-linear map ${\mathbb{F}^n \to \mathbb{F}}$. As such, $\mathsf{D}\text{-}\mathsf{lin}[\mathbb{F}\text{-}\mathsf{POLY}]$ is equivalent to the category $\mathbb{F}\text{-}\mathsf{LIN}$ whose objects are the finite powers $\mathbb{F}^n$ for each $n \in \mathbb{N}$ (including the singleton $\mathbb{F}^0 = \lbrace 0 \rbrace$) and whose maps are $\mathbb{F}$-linear maps ${\mathbb{F}^n \to \mathbb{F}^m}$. We note that this example can be generalized to the category of polynomials over an arbitrary commutative (semi)ring.
\end{example}

\begin{example}\label{ex:smooth} \normalfont  Let $\mathbb{R}$ be the set of real numbers. Define $\mathsf{SMOOTH}$ as the category whose objects are the Euclidean real vector spaces $\mathbb{R}^n$ and whose maps are the real smooth functions ${F: \mathbb{R}^n \to \mathbb{R}^m}$ between them. $\mathsf{SMOOTH}$ is a Cartesian differential category, arguably the canonical example, where the differential combinator is defined as the directional derivative of a smooth function. So for a smooth function $F: \mathbb{R}^n \to \mathbb{R}^m$, its derivative ${\mathsf{D}[F]: \mathbb{R}^n \times \mathbb{R}^n \to \mathbb{R}^m}$ is then defined as:
\[\mathsf{D}[F](\vec x, \vec y) := \left \langle \sum \limits^n_{i=1} \frac{\partial f_1}{\partial x_i}(\vec x) y_i, \hdots, \sum \limits^n_{i=1} \frac{\partial f_n}{\partial x_i}(\vec x) y_i \right \rangle\]
Note that $\mathbb{R}\text{-}\mathsf{POLY}$ is a sub-Cartesian differential category of $\mathsf{SMOOTH}$. A smooth function $F: \mathbb{R}^n \to \mathbb{R}^m$ is $\mathsf{D}$-linear if and only if it is $\mathbb{R}$-linear in the classical sense. Therefore, $\mathsf{D}\text{-}\mathsf{lin}[\mathsf{SMOOTH}]= \mathbb{R}\text{-}\mathsf{LIN}$. 
\end{example}

\begin{example} \normalfont An important source of examples of Cartesian differential categories, especially for this paper, are those which arise as the coKleisli category of a differential category \cite{blute2006differential,blute2015cartesian}. We will review this example in Example \ref{ex:diffcat}. 
\end{example}

There are many other interesting (and sometimes very exotic) examples of Cartesian differential categories in the literature. See \cite{garner2020cartesian,cockett2020linearizing} for lists of more examples of Cartesian differential categories. Interesting generalizations of Cartesian differential categories include $R$-linear Cartesian differential categories \cite{garner2020cartesian} (which adds the ability of scalar multiplication by a commutative ring $R$), generalized Cartesian differential categories \cite{cruttwell2017cartesian} (which generalizes the notion of differential calculus of smooth functions between open subsets), differential restriction categories \cite{cockett2011differential} (which generalizes the notion of differential calculus of partially defined smooth functions), and tangent categories \cite{cockett2014differential} (which generalizes the notion of differential calculus over smooth manifolds).

\section{Cartesian Differential Comonads}\label{sec:CDComonad}

In this section, we introduce the main novel concept of study in this paper: Cartesian differential comonads, which are precisely the comonads whose coKleisli category is a Cartesian differential category. This is a generalization of \cite[Proposition 3.2.1]{blute2009cartesian}, which states that the coKleisli category of the comonad of a differential category is a Cartesian differential category. The generalization comes from the fact that a Cartesian differential comonad can be defined without the need for a monoidal product or cocommutative comonoid structure on the comonad's coalgebras. As such, this allows for a wider variety of examples of Cartesian differential categories. Briefly, a Cartesian differential comonad is a comonad on a category with finite biproducts, which comes equipped with a differential combinator transformation, which generalizes the notion of a deriving transformation in a differential category \cite{blute2006differential,Blute2019}. The induced differential combinator is defined by precomposing a coKleisli map with the differential combinator transformation (with respect to composition in the base category). Conversely, a comonad whose coKleisli category is a Cartesian differential category is a Cartesian differential comonad, where the differential combinator transformation is defined using the coKleisli category's differential combinator. We point out that this statement, regarding comonads whose coKleisli categories are Cartesian differential categories, is a novel observation and shows us that even if one cannot extract a monoidal product on the base category from the coKleisli category, it is possible to obtain a natural transformation which captures differentiation. Lastly, we will also study the case where the $\mathsf{D}$-linear maps of the coKleisli category correspond to the maps of the base category. The situation arises precisely in the presence of what we call a $\mathsf{D}$-linear unit, which generalizes the notion of a codereliction from differential linear logic \cite{blute2006differential,Blute2019,fiore2007differential,ehrhard2017introduction}.

If only to introduce notation, recall that a comonad on a category $\mathbb{X}$ is a triple $(\oc, \delta, \varepsilon)$ consisting of a functor ${\oc: \mathbb{X} \to \mathbb{X}}$, and two natural transformations $\delta_A: \oc(A) \to \oc \oc (A)$, called the comonad comultiplication, and ${\varepsilon_A: \oc(A) \to A}$, called the comonad counit, and such that the following equalities hold: 
  \begin{equation}\label{comonadeq}\begin{gathered} 
\delta_{\oc(A)} \circ \delta_A = \oc(\delta_A) \circ \delta_A \quad \quad \quad \varepsilon_{\oc(A)} \circ \delta_A = 1_{\oc(A)} = \oc(\varepsilon_A) \circ \delta_A
  %\xymatrixcolsep{5pc}\xymatrix{ 
     %   \oc(A)  \ar[r]^-{\delta_A} \ar[d]_-{\delta_A} \ar@{=}[dr]^-{}& \oc \oc(A) \ar[d]^-{\varepsilon_{\oc(A)}}  & \oc(A)  \ar[r]^-{\delta_A} \ar[d]_-{\delta_A} & \oc \oc(A)  \ar[d]^-{\delta_{\oc(A)}}\\
       % \oc \oc(A) \ar[r]_-{\oc(\varepsilon_A)} & \oc(A)  & \oc \oc(A) \ar[r]_-{\oc(\delta_A)} & \oc \oc \oc(A)}
       \end{gathered}\end{equation}

\begin{definition}\label{def:cdcomonad} For a comonad $(\oc, \delta, \varepsilon)$ on a category $\mathbb{X}$ with finite biproducts, a \textbf{differential combinator transformation} on $(\oc, \delta, \varepsilon)$ is a natural transformation $\partial_A: \oc(A \times A) \to \oc(A)$ such that the following diagrams commute: 
\begin{enumerate}[{\bf [dc.1]}] 
\item Zero Rule: 
  \[  \xymatrixcolsep{5pc}\xymatrix{ \oc(A) \ar[dr]_-{0} \ar[r]^-{\oc(\iota_0)} & \oc(A \times A) \ar[d]^-{\partial_A} \\
  & \oc(A)  
  } \]
  where $\iota_0: A \to A \times A$ is defined as in Definition \ref{CLACmapsdef}.(\ref{injdef}). 
 \item Additive Rule: 
    \[  \xymatrixcolsep{7pc}\xymatrix{ \oc\!\left( A \times (A \times A) \right) \ar[d]_-{\oc(1_A \times \pi_0) + \oc(1_A \times \pi_1)} \ar[r]^-{\oc(1_A \times \nabla_A)} & \oc(A \times A) \ar[d]^-{\partial_A} \\
 \oc(A \times A) \ar[r]_-{\partial_A}  & \oc(A)  
  } \]
  where $\nabla_A: A \times A \to A$ is defined as in Definition \ref{CLACmapsdef}.(\ref{nabladef}). 
\item Linear Rule: 
  \[  \xymatrixcolsep{5pc}\xymatrix{ \oc(A \times A) \ar[r]^-{\partial_A} \ar[d]_-{\varepsilon_{A \times A}} & \oc(A) \ar[d]^-{\varepsilon_A} \\  
  A \times A \ar[r]_-{\pi_1} & A 
  } \]  
 \item Chain Rule: 
   \[  \xymatrixcolsep{5pc}\xymatrix{ \oc(A \times A) \ar[d]_-{\delta_{A \times A}} \ar[rr]^-{\partial_A} && \oc(A)  \ar[d]^-{\delta_A} \\
   \oc\oc(A \times A) \ar[r]_-{\oc\left(\langle \oc(\pi_0), \partial_{A} \rangle \right)} & \oc\!\left( \oc(A) \times \oc(A) \right) \ar[r]_-{\partial_{\oc(A)}} & \oc\oc(A)
  } \]
 \item Lift Rule: 
   \[  \xymatrixcolsep{5pc}\xymatrix{\oc\left( A \times A \right) \ar[ddr]_-{\partial_{A}} \ar[r]^-{\oc(\ell_A)} & \oc\!\left( (A \times A) \times (A \times A) \right) \ar[d]^-{\partial_{A \times A}} \\
   & \oc(A \times A) \ar[d]^-{\partial_A} \\
& \oc(A)  
  } \]
   where $\ell_A: A \times A \to (A \times A) \times (A \times A)$ is defined as in Definition \ref{CLACmapsdef}.(\ref{elldef}).  
 \item Symmetry Rule: 
   \[  \xymatrixcolsep{5pc}\xymatrix{\oc\left( (A \times A) \times (A \times A) \right) \ar[dd]_-{\partial_{A \times A}} \ar[r]^-{\oc(c_A)} & \oc\!\left( (A \times A) \times (A \times A) \right) \ar[d]^-{\partial_{A \times A}} \\
   & \oc(A \times A) \ar[d]^-{\partial_A} \\
 \oc(A \times A) \ar[r]_-{\partial_A}  & \oc(A)  
  } \]
     where $c_A: (A \times A) \times (A \times A) \to (A \times A) \times (A \times A)$ is defined as in Definition \ref{CLACmapsdef}.(\ref{cdef}).  
\end{enumerate}
A \textbf{Cartesian differential comonad} on a category $\mathbb{X}$ with finite biproducts is a quadruple $(\oc, \delta, \varepsilon, \partial)$ consisting of a comonad $(\oc, \delta, \varepsilon)$ and a differential combinator transformation $\partial$ on $(\oc, \delta, \varepsilon)$.  
\end{definition}

As the name suggests, the differential combinator transformations axioms correspond to some of the axioms a differential combinator. The zero rule \textbf{[dc.1]} and the additive rule \textbf{[dc.2]} correspond to \textbf{[CD.2]}, the linear rule \textbf{[dc.3]} corresponds to \textbf{[CD.3]}, the chain rule  \textbf{[dc.4]} corresponds to \textbf{[CD.5]}, the lift rule corresponds to \textbf{[CD.6]}, and lastly the symmetry rule \textbf{[dc.6]} corresponds to \textbf{[CD.7]}. 

Our goal is now to show that the coKleisli category of a Cartesian differential comonad is a Cartesian differential category. As we will be working with coKleisli categories, we will use the notation found in \cite{blute2015cartesian} and use interpretation brackets $\llbracket - \rrbracket$ to help distinguish between composition in the base category and coKleisli composition. So for a comonad $(\oc, \delta, \varepsilon)$ on a category $\mathbb{X}$, let $\mathbb{X}_\oc$ denote its coKleisli category, which is the category whose objects are the same as $\mathbb{X}$ and where a map $A \to B$ in the coKleisli category is map of type $\oc(A) \to B$ in the base category, that is, $\mathbb{X}_\oc(A,B) = \mathbb{X}(\oc(A), B)$. Composition of coKleisli maps ${\llbracket f \rrbracket: \oc(A) \to B}$ and $\llbracket g \rrbracket: \oc(B) \to C$ is defined as follows: 
\begin{align*}
\llbracket g \circ f \rrbracket :=   \xymatrixcolsep{3pc}\xymatrix{\oc (A) \ar[r]^-{\delta_A} & \oc\oc(A) \ar[r]^-{\oc \left( \llbracket f \rrbracket \right) } & \oc(B) \ar[r]^-{\llbracket g \rrbracket} & C } && \llbracket g \circ f \rrbracket = \llbracket g \rrbracket \circ \oc\left( \llbracket f \rrbracket \right) \circ \delta_A
\end{align*}
The identity maps in the coKleisli category is given by the comonad counit: 
\begin{align*}
\llbracket 1_A \rrbracket  := \xymatrixcolsep{5pc}\xymatrix{\oc (A) \ar[r]^-{\varepsilon_A} & A }
\end{align*}
Let $\mathsf{F}_\oc: \mathbb{X} \to \mathbb{X}_\oc$ be the standard inclusion functor which is defined on objects as $\mathsf{F}_\oc(A)=A$ and on maps ${f: A \to B}$ as follows: 
\begin{align*}
\llbracket \mathsf{F}_\oc(f) \rrbracket :=   \xymatrixcolsep{3pc}\xymatrix{\oc (A) \ar[r]^-{\varepsilon_A} & A \ar[r]^-{f} & B } && \llbracket \mathsf{F}_\oc(f) \rrbracket = f \circ \varepsilon_A
\end{align*}
A key map in this story is the coKleisli map whose interpretation is the identity map in the base category. So for every object $A$, define the map $\varphi_A: A \to \oc(A)$ in the coKleisli category as follows: 
  \begin{equation}\label{varphidef}\begin{gathered} 
\llbracket \varphi_A \rrbracket  := \xymatrixcolsep{5pc}\xymatrix{\oc (A) \ar[r]^-{1_{\oc(A)}} & \oc(A) }
\end{gathered}\end{equation}
Here are now some useful identities in the coKleisli category: 

\begin{lemma}\label{cokleislilem1}  Let $(\oc, \delta, \varepsilon)$ be a comonad on a category $\mathbb{X}$. Then the following equalities hold: 
\begin{enumerate}[{\em (i)}]
\item \label{cokleislilem1.right} $\llbracket g \circ \mathsf{F}_\oc(f) \rrbracket = \llbracket g \rrbracket \circ \oc(f)$ 
\item\label{cokleislilem1.left}  $\llbracket \mathsf{F}_{\oc}(g) \circ f \rrbracket = g \circ \llbracket f \rrbracket$
\item \label{cokleislilem1.varphi}  $\llbracket f \rrbracket = \llbracket \mathsf{F}_\oc\left( \llbracket f \rrbracket \right) \circ \varphi_A \rrbracket$ 
\item  \label{cokleislilem1.varphi2}  $\llbracket \varphi_B \circ \mathsf{F}_\oc(f) \rrbracket = \oc(f)$ 
\item \label{cokleislilem1.varphi3}  $\llbracket \varphi_{\oc(A)} \circ \varphi_A \rrbracket = \delta_A$ 
\end{enumerate}
\end{lemma}

It is a well-known result that if the base category has finite products, then so does the coKleisli category. 

\begin{lemma} \label{cokleisliproduct} \cite[Dual of Proposition 2.2]{szigeti1983limits} Let $(\oc, \delta, \varepsilon)$ be a comonad on a category $\mathbb{X}$ with finite products. Then the coKleisli category $\mathbb{X}_\oc$ has finite products where: 
\begin{enumerate}[{\em (i)}]
\item The product $\times$ on objects is defined as as in $\mathbb{X}$;
\item The projection maps $\llbracket \pi_0 \rrbracket: \oc(A \times B) \to A$ and $\llbracket \pi_1 \rrbracket: \oc(A \times B) \to B$ are defined respectively as follows: 
\begin{align*}
\llbracket \pi_0 \rrbracket:=   \xymatrixcolsep{3pc}\xymatrix{\oc (A \times B) \ar[r]^-{\varepsilon_{A \times B}} & A \times B \ar[r]^-{\pi_0} & A } &&\llbracket \pi_0 \rrbracket = \pi_0 \circ \varepsilon_{A \times B} \\
\llbracket \pi_1 \rrbracket:=   \xymatrixcolsep{3pc}\xymatrix{\oc (A \times B) \ar[r]^-{\varepsilon_{A \times B}} & A \times B \ar[r]^-{\pi_1} & B } &&\llbracket \pi_1 \rrbracket = \pi_1 \circ \varepsilon_{A \times B} 
\end{align*}
\item The pairing of coKleisli maps $\llbracket f \rrbracket: \oc(C) \to A$ and $\llbracket g \rrbracket: \oc(C) \to B$ is defined as in $\mathbb{X}$, that is:
\begin{align*}
 \llbracket \langle f, g \rangle \rrbracket := \xymatrixcolsep{5pc}\xymatrix{\oc (C) \ar[r]^-{\left \langle \llbracket f \rrbracket, \llbracket g \rrbracket \right \rangle} & A \times B }
\end{align*}
\item The terminal object $\top$ is the same as in $\mathbb{X}$. 
\item For coKleisli maps $\llbracket f \rrbracket: \oc(A) \to C$ and $\llbracket g \rrbracket: \oc(B) \to D$, their product is equal to the following composite:
 \begin{align*}
\llbracket f \times g \rrbracket:=   \xymatrixcolsep{3.5pc}\xymatrix{\oc (A \times B) \ar[r]^-{\langle \oc(\pi_0), \oc(\pi_1) \rangle} & \oc(A) \times \oc(B) \ar[r]^-{\llbracket f \rrbracket \times \llbracket g \rrbracket} & C \times D } &&\llbracket f \times g \rrbracket = \left( \llbracket f \rrbracket \times \llbracket g \rrbracket \right) \circ \langle \oc(\pi_0), \oc(\pi_1) \rangle
\end{align*}
\item \label{cokleisliproduct.F} $\mathsf{F}_\oc: \mathbb{X} \to \mathbb{X}_\oc$ preserves the finite product strictly, that is, the following equalities hold: 
\begin{align*}
\mathsf{F}_\oc(A \times B) &= A \times B && &\mathsf{F}_\oc(\top) &= \top \\
\llbracket \mathsf{F}_\oc(\pi_0) \rrbracket &= \llbracket \pi_0 \rrbracket &&& \llbracket \mathsf{F}_\oc(\pi_1) \rrbracket&= \llbracket \pi_1 \rrbracket \\
\llbracket \mathsf{F}_\oc\left(\langle f, g \rangle \right) \rrbracket &= \llbracket \langle \mathsf{F}_\oc(f), \mathsf{F}_\oc(g) \rangle \rrbracket &&& \llbracket \mathsf{F}_\oc\left( f \times g \right) \rrbracket &= \llbracket \mathsf{F}_\oc(f) \times \mathsf{F}_\oc(g) \rrbracket
\end{align*}
\end{enumerate}
\end{lemma}

If the base category is also Cartesian left additive, then so is the coKleisli category in a canonical way, that is, where the additive structure is simply that of the base category. 

\begin{lemma}\label{cokleisliCLAC} \cite[Proposition 1.3.3]{blute2009cartesian} Let $(\oc, \delta, \varepsilon)$ be a comonad on a Cartesian left additive category $\mathbb{X}$ with finite products. Then the coKleisli category $\mathbb{X}_\oc$ is a Cartesian left additive category where the finite product structure is given in Lemma \ref{cokleisliproduct} and where:
\begin{enumerate}[{\em (i)}]
\item The sum of coKleisli maps $\llbracket f \rrbracket: \oc(A) \to B$ and $\llbracket g \rrbracket: \oc(A) \to B$ is defined as in $\mathbb{X}$, that is:
\[\llbracket f+g \rrbracket = \llbracket f \rrbracket + \llbracket g \rrbracket\] 
\item The zero $\llbracket 0 \rrbracket: \oc(A) \to B$ is the same as in $\mathbb{X}$, that is: 
\[\llbracket 0 \rrbracket = 0\] 
\item \label{cokleisliCLAC.F1} $\mathsf{F}_\oc: \mathbb{X} \to \mathbb{X}_\oc$ preserves the additive structure strictly, that is, the following equalities hold: 
\begin{align*}
\llbracket \mathsf{F}_\oc(f + g) \rrbracket= \llbracket \mathsf{F}_\oc(f) + \mathsf{F}_\oc(g) \rrbracket && \llbracket \mathsf{F}_\oc(0) \rrbracket= 0
\end{align*}
\item \label{cokleisliCLAC.F2} The following equalities hold: 
\begin{align*}
\llbracket \iota_0 \rrbracket = \iota_0 \circ \varepsilon_{A} = \llbracket \mathsf{F}_\oc(\iota_0) \rrbracket &&  \llbracket \iota_1 \rrbracket = \iota_1 \circ \varepsilon_{B} =  \llbracket \mathsf{F}_\oc(\iota_1) \rrbracket\\
 \llbracket \nabla_A \rrbracket = \nabla_A \circ \varepsilon_{A \times A} =  \llbracket \mathsf{F}_\oc(\nabla_A) \rrbracket && \llbracket \ell_A \rrbracket = \ell_A \circ \varepsilon_{A \times A} =  \llbracket \mathsf{F}_\oc(\ell_A) \rrbracket 
\end{align*}
\[ \llbracket c_A \rrbracket = c_A \circ \varepsilon_{(A \times A) \times (A \times A)} = \llbracket \mathsf{F}_\oc(c_A) \rrbracket  \] 
where $\iota_j$, $\nabla$, $\ell$, and $c$ are defined as in Definition \ref{CLACmapsdef}.
\end{enumerate}
\end{lemma}

Now since every category $\mathbb{X}$ with finite biproducts is a Cartesian left additive category, it follows that for every comonad $(\oc, \delta, \varepsilon)$ on $\mathbb{X}$, the coKleisli category $\mathbb{X}_\oc$ is a Cartesian left additive category. It is important to point out that even if all maps in $\mathbb{X}$ are additive maps, the same is not true for $\mathbb{X}_\oc$. This is due to the fact that $\oc(f +g)$ and $\oc(0)$ do not necessarily equal $\oc(f) + \oc(g)$ and $0$ respectively.

We now provide the first main result of this paper: that the coKleisli category of a Cartesian differential comonad is a Cartesian differential category. 

\begin{theorem}\label{thm1} Let $(\oc, \delta, \varepsilon, \partial)$ be a Cartesian differential comonad on a category $\mathbb{X}$ with finite biproducts. Then the coKleisli category $\mathbb{X}_\oc$ is a Cartesian differential category where the Cartesian left additive structure is defined as in Lemma \ref{cokleisliCLAC} and the differential combinator $\mathsf{D}$ is defined as follows: for a map ${\llbracket f \rrbracket: \oc(A) \to B}$, its derivative $\llbracket \mathsf{D}[f] \rrbracket: \oc(A \times A) \to B$ is defined as the following composite:
  \[ \llbracket \mathsf{D}[f] \rrbracket := \xymatrixcolsep{5pc}\xymatrix{\oc(A \times A) \ar[r]^-{\partial_A} & \oc(A) \ar[r]^-{\llbracket f \rrbracket} & B  
  } \]
Furthermore: 
\begin{enumerate}[{\em (i)}]
 \item \label{thm1.varphi} For every object $A$ in $\mathbb{X}$, $\llbracket \mathsf{D}[\varphi_A] \rrbracket = \partial_A$. 
\item \label{thm1.lin}A coKleisli map $\llbracket f \rrbracket: \oc(A) \to B$ is $\mathsf{D}$-linear in $\mathbb{X}_\oc$ if and only if $\llbracket f \rrbracket \circ \partial_A \circ \oc(\iota_1) = \llbracket f \rrbracket$. 
\item For every map $f: A \to B$ in $\mathbb{X}$, $\llbracket \mathsf{F}_\oc(f) \rrbracket$ is $\mathsf{D}$-linear in $\mathbb{X}_\oc$. 
\item \label{Flindef} There is a functor $\mathsf{F}_{\mathsf{D}\text{-}\mathsf{lin}}: \mathbb{X} \to \mathsf{D}\text{-}\mathsf{lin}[\mathbb{X}_\oc]$ which is defined on objects as $\mathsf{F}_{\mathsf{D}\text{-}\mathsf{lin}}(A) = A$ and on maps $f: A \to B$ as $\llbracket \mathsf{F}_{\mathsf{D}\text{-}\mathsf{lin}}(f) \rrbracket = f \circ \varepsilon_A = \llbracket \mathsf{F}_{\oc}(f) \rrbracket$, and such that the following diagram commutes: 
  \[  \xymatrixcolsep{5pc}\xymatrix{ \mathbb{X} \ar[dr]_-{\mathsf{F}_{\mathsf{D}\text{-}\mathsf{lin}}} \ar[rr]^-{\mathsf{F}_\oc} && \mathbb{X}_\oc  \\
  & \mathsf{D}\text{-}\mathsf{lin}[\mathbb{X}_\oc] \ar[ur]_-{\mathsf{U}}
  } \]
\end{enumerate}
\end{theorem} 
\begin{proof} We prove the seven axioms of a differential combinator. We make heavy use of Lemma \ref{cokleislilem1}.(\ref{cokleislilem1.right}). 
\begin{enumerate}[{\bf [CD.1]}]
\item Here we use the additive enrichment of $\mathbb{X}$: 
\begin{align*}
\llbracket \mathsf{D}[f +g] \rrbracket &= \llbracket f + g \rrbracket \circ \partial_A  \\
&= \left( \llbracket f \rrbracket + \llbracket g \rrbracket \right) \circ \partial_A \\
&=  \llbracket f \rrbracket \circ \partial_A +  \llbracket g \rrbracket \circ \partial_A \\
&= \llbracket \mathsf{D}[f] \rrbracket + \llbracket \mathsf{D}[g] \rrbracket \\\\
\llbracket \mathsf{D}[0] \rrbracket &= \llbracket 0 \rrbracket \circ \partial_A \\
&= 0 \circ \partial_A  \\
&= 0
\end{align*}
\item Here we use the fact that every map in $\mathbb{X}$ is additive, and both the zero rule \textbf{[dc.1]} and additive rule \textbf{[dc.2]}:  
\begin{align*}
\llbracket  \mathsf{D}[f] \circ (1_A \times \nabla_A) \rrbracket &=~\llbracket  \mathsf{D}[f] \circ \mathsf{F}_{\oc}(1_A \times \nabla_A) \rrbracket \tag{Lem.\ref{cokleisliproduct}.(\ref{cokleisliproduct.F}) + Lem.\ref{cokleisliCLAC}.(\ref{cokleisliCLAC.F2})} \\
&=~\llbracket  \mathsf{D}[f]  \rrbracket \circ \oc(1_A \times \nabla_A)   \tag{Lem.\ref{cokleislilem1}.(\ref{cokleislilem1.right})} \\
&=~\llbracket f \rrbracket \circ \partial_A \circ \oc(1_A \times \nabla_A) \\
&=~\llbracket f \rrbracket \circ \partial_A \circ \left( \oc(1_A \times \pi_0) + \oc(1_A \times \pi_1) \right) \tag{\textbf{[dc.2]}}\\
&=~\llbracket f \rrbracket \circ \partial_A \circ \oc(1_A \times \pi_0) + \llbracket f \rrbracket \circ \partial_A \circ \oc(1_A \times \pi_1) \\
&=~\llbracket  \mathsf{D}[f]  \rrbracket \circ  \oc(1_A \times \pi_0) + \llbracket  \mathsf{D}[f]  \rrbracket \circ  \oc(1_A \times \pi_1) \\
&=~ \llbracket  \mathsf{D}[f] \circ \mathsf{F}_{\oc}(1_A \times \pi_0) \rrbracket + \llbracket  \mathsf{D}[f] \circ \mathsf{F}_{\oc}(1_A \times \pi_1) \rrbracket  \tag{Lem.\ref{cokleislilem1}.(\ref{cokleislilem1.right})} \\
&=~ \llbracket  \mathsf{D}[f] \circ (1_A \times \pi_0) \rrbracket + \llbracket  \mathsf{D}[f] \circ (1_A \times \pi_1) \rrbracket  \tag{Lem.\ref{cokleisliproduct}.(\ref{cokleisliproduct.F}) + Lem.\ref{cokleisliCLAC}.(\ref{cokleisliCLAC.F2})} \\
&=~\llbracket  \mathsf{D}[f] \circ (1_A \times \pi_0) + \mathsf{D}[f] \circ (1_A \times \pi_1) \rrbracket
\end{align*}
\begin{align*}
\llbracket  \mathsf{D}[f] \circ \iota_0 \rrbracket &=~\llbracket  \mathsf{D}[f] \circ \mathsf{F}_\oc(\iota_0) \rrbracket \tag{Lem.\ref{cokleisliCLAC}.(\ref{cokleisliCLAC.F2})} \\
&=~\llbracket  \mathsf{D}[f]  \rrbracket \circ  \oc(\iota_0) \tag{Lem.\ref{cokleislilem1}.(\ref{cokleislilem1.right})} \\
&=~\llbracket f \rrbracket \circ \partial_A \circ  \oc(\iota_0) \\
&=~\llbracket f \rrbracket \circ 0 \tag{\textbf{[dc.1]}}\\
&=~ 0
\end{align*}
\item Here we use the linear rule \textbf{[dc.3]} and Lemma \ref{cokleislilem1}.(\ref{cokleislilem1.left}):
\begin{align*}
\llbracket \mathsf{D}[1_A] \rrbracket &=~ \llbracket 1_A \rrbracket \circ \partial_A \\
&=~\varepsilon_A \circ \partial_A \\
&=~\pi_1 \circ \varepsilon_{A \times A} \tag{\textbf{[dc.3]}}\\
&=~\llbracket \pi_1 \rrbracket \\ \\
\llbracket \mathsf{D}[\pi_j] \rrbracket &=~ \llbracket \pi_0 \rrbracket \circ \partial_{A \times B} \\
&=~\pi_j \circ \varepsilon_{A \times B} \circ \partial_{A \times B} \\
&=~\pi_j \circ \pi_1 \circ \varepsilon_{(A \times B) \times (A \times B)} \tag{\textbf{[dc.3]}}\\
&=~ \pi_j \circ \llbracket \pi_1 \rrbracket \\ 
&=~\llbracket \mathsf{F}_\oc(\pi_j) \circ \pi_1 \rrbracket \tag{Lem.\ref{cokleislilem1}.(\ref{cokleislilem1.left})} \\
&=~ \llbracket \pi_j \circ \pi_1 \rrbracket \tag{Lem.\ref{cokleisliproduct}.(\ref{cokleisliproduct.F})} 
\end{align*}
\item This is mostly straightforward from the product structure: 
\begin{align*}
\llbracket \mathsf{D}\left[ \langle f, g \rangle \right] \rrbracket &=~\llbracket \langle f, g \rangle \rrbracket \circ \partial_A \\
&=~\left \langle \llbracket f \rrbracket, \llbracket g \rrbracket \right \rangle \circ \partial_A \\
&=~\left \langle \llbracket f \rrbracket \circ \partial_A, \llbracket g \rrbracket  \circ \partial_A \right\rangle \\
&=~ \left \langle \llbracket \mathsf{D}[f] \rrbracket, \llbracket \mathsf{D}[g] \rrbracket \right \rangle \\
&=~\llbracket \mathsf{D}[ \langle f,g \rangle] \rrbracket 
\end{align*}
\item Here we use the chain rule \textbf{[dc.4]} and the naturality of $\partial$: 
\begin{align*}
\llbracket \mathsf{D}[g \circ f] \rrbracket &=~\llbracket g \circ f \rrbracket \circ \partial_A \\
&=~  \llbracket g \rrbracket \circ \oc\left( \llbracket f \rrbracket \right) \circ \delta_A \circ \partial_A \\
&=~ \llbracket g \rrbracket \circ \oc\left( \llbracket f \rrbracket \right) \circ \partial_{\oc(A)} \circ \oc\left( \langle \oc(\pi_0), \partial_A \rangle \right) \circ \delta_{A \times A} \tag{\textbf{[dc.4]}} \\
&=~   \llbracket g \rrbracket \circ \partial_{B} \circ \oc\left( \llbracket f \rrbracket \times  \llbracket f \rrbracket \right) \circ \oc\left( \langle \oc(\pi_0), \partial_A \rangle \right) \circ \delta_{A \times A} \tag{Naturality of $\partial$} \\
&=~ \llbracket \mathsf{D}[g] \rrbracket \circ \oc\left( \llbracket f \rrbracket \times  \llbracket f \rrbracket \right) \circ \oc\left( \langle \oc(\pi_0), \partial_A \rangle \right) \circ \delta_{A \times A} \\ 
&=~ \llbracket \mathsf{D}[g] \rrbracket  \circ \oc\left( \left(\llbracket f \rrbracket \times  \llbracket f \rrbracket \right) \circ \langle \oc(\pi_0), \partial_A \rangle \right) \circ \delta_{A \times A} \tag{Functoriality of $\oc$} \\
&=~  \llbracket \mathsf{D}[g] \rrbracket  \circ \oc\left(  \left \langle \llbracket f \rrbracket \circ \oc(\pi_0), \llbracket f \rrbracket \circ \partial_A  \right \rangle \right) \circ \delta_{A \times A}  \\
&=~  \llbracket \mathsf{D}[g] \rrbracket  \circ \oc\left(  \left \langle \llbracket f \rrbracket \circ \oc(\pi_0), \llbracket \mathsf{D}[f] \rrbracket  \right \rangle \right) \circ \delta_{A \times A}  \\
&=~  \llbracket \mathsf{D}[g] \rrbracket \circ \oc\left(  \left \langle \llbracket f \circ \mathsf{F}_\oc(\pi_0) \rrbracket,  \llbracket \mathsf{D}[f] \rrbracket   \right \rangle \right) \circ \delta_{A \times A}   \tag{Lem.\ref{cokleislilem1}.(\ref{cokleislilem1.right})} \\
&=~ \llbracket \mathsf{D}[g] \rrbracket \circ \oc\left(  \left \langle \llbracket f \circ \pi_0 \rrbracket,  \llbracket \mathsf{D}[f] \rrbracket   \right \rangle \right) \circ \delta_{A \times A}  \tag{Lem.\ref{cokleisliproduct}.(\ref{cokleisliproduct.F})} \\ 
&=~ \llbracket \mathsf{D}[g] \rrbracket \circ \oc\left(  \left  \llbracket \left \langle f \circ \pi_0, \mathsf{D}[f]    \right \rangle \right \rrbracket \right) \circ \delta_{A \times A} \\
&=~ \left \llbracket  \mathsf{D}[g] \circ \left \langle f \circ \pi_0, \mathsf{D}[f]    \right \rangle \right \rrbracket
\end{align*}
\item Here we use the lifting rule \textbf{[dc.5]}: 
\begin{align*}
\llbracket \mathsf{D}\left[\mathsf{D}[f] \right] \circ \ell_A \rrbracket &=~\llbracket \mathsf{D}\left[\mathsf{D}[f] \right] \circ \mathsf{F}_\oc(\ell_A) \rrbracket  \tag{Lem.\ref{cokleisliCLAC}.(\ref{cokleisliCLAC.F2})} \\
&=~\llbracket \mathsf{D}\left[\mathsf{D}[f] \right] \rrbracket \circ \oc(\ell_A) \tag{Lem.\ref{cokleislilem1}.(\ref{cokleislilem1.right})} \\
&=~ \llbracket \mathsf{D}[f] \rrbracket \circ \partial_{A \times A} \circ \oc(\ell_A) \\
&=~\llbracket f \rrbracket  \circ \partial_A \circ \partial_{A \times A} \circ \oc(\ell_A) \\
&=~\llbracket f \rrbracket  \circ \partial_A \tag{\textbf{[dc.5]}} \\
&=~ \llbracket \mathsf{D}[f] \rrbracket 
\end{align*}
\item Here we use the symmetry rule \textbf{[dc.6]}: 
\begin{align*}
\llbracket \mathsf{D}\left[\mathsf{D}[f] \right] \circ c_A \rrbracket &=~\llbracket \mathsf{D}\left[\mathsf{D}[f] \right] \circ \mathsf{F}_\oc(c_A) \rrbracket  \tag{Lem.\ref{cokleisliCLAC}.(\ref{cokleisliCLAC.F2})} \\
&=~\llbracket \mathsf{D}\left[\mathsf{D}[f] \right] \rrbracket \circ \oc(c_A) \tag{Lem.\ref{cokleislilem1}.(\ref{cokleislilem1.right})} \\
&=~ \llbracket \mathsf{D}[f] \rrbracket \circ \partial_{A \times A} \circ \oc(c_A) \\
&=~\llbracket f \rrbracket  \circ \partial_A \circ \partial_{A \times A} \circ \oc(c_A) \\
&=~\llbracket f \rrbracket  \circ \partial_A \circ \partial_{A \times A} \tag{\textbf{[dc.6]}} \\
&=~ \llbracket \mathsf{D}[f] \rrbracket \circ \partial_{A \times A} \\
&=~\llbracket \mathsf{D}\left[\mathsf{D}[f] \right] \rrbracket 
\end{align*}
\end{enumerate}
So we conclude that $\mathsf{D}$ is a differential combinator, and therefore that the coKleisli category $\mathbb{X}_\oc$ is a Cartesian differential category. Next we prove the remaining claims. 
\begin{enumerate}[{\em (i)}]
\item This is automatic by definition since: 
\begin{align*}
\llbracket \mathsf{D}[\varphi_A] \rrbracket &=~ \llbracket \varphi_A \rrbracket \circ \partial_A \\
&=~ 1_{\oc(A)} \circ \partial_A \\
&=~\partial_A 
\end{align*}
\item By Lemma \ref{linlem}.(\ref{linlemimportant2}), $\llbracket f \rrbracket$ is $\mathsf{D}$-linear if and only if $\llbracket \mathsf{L}[f] \rrbracket = \llbracket f \rrbracket$. However, expanding the left hand side of the equality we have that: 
\begin{align*}
\llbracket \mathsf{L}[f]  \rrbracket &=~ \llbracket \mathsf{D}[f] \circ \iota_1 \rrbracket \\
&=~\llbracket \mathsf{D}[f] \circ \mathsf{F}_\oc(\iota_1) \rrbracket \tag{Lem.\ref{cokleisliCLAC}.(\ref{cokleisliCLAC.F2})} \\ 
&=~\llbracket \mathsf{D}[f]  \rrbracket \circ \oc(\iota_1) \tag{Lem.\ref{cokleislilem1}.(\ref{cokleislilem1.right})} \\
&=~\llbracket f \rrbracket \circ \partial_A \circ \oc(\iota_1)
\end{align*}
Therefore, $\llbracket f \rrbracket$ is $\mathsf{D}$-linear if and only if $\llbracket f \rrbracket \circ \partial_A \circ \oc(\iota_1) = \llbracket f \rrbracket$. 
\item For any map $f: A \to B$ in $\mathbb{X}$, using the linear rule \textbf{[dc.3]}, naturality of $\varepsilon$, and the biproduct identities, we compute: 
\begin{align*}
\llbracket \mathsf{F}_\oc(f) \rrbracket \circ \partial_A \circ \oc(\iota_1) &=~f \circ \varepsilon_A \circ \partial_A \circ \oc(\iota_1) \\
&=~f \circ \pi_1 \circ \varepsilon_{A \times A} \circ \oc(\iota_1) \\
&=~ f \circ \pi_1 \circ \iota_1 \circ \varepsilon_{A} \tag{Naturality of $\varepsilon$} \\
&=~f \circ \varepsilon_A \tag{Biproduct Identity} \\
&=~\llbracket \mathsf{F}_\oc(f) \rrbracket 
\end{align*}
Therefore, since $\llbracket \mathsf{F}_\oc(f) \rrbracket \circ \partial_A \circ \oc(\iota_1) = \llbracket \mathsf{F}_\oc(f) \rrbracket$, by the above, it follows that $\llbracket \mathsf{F}_\oc(f) \rrbracket $ is $\mathsf{D}$-linear. 
\item By the above, $\mathsf{F}_{\mathsf{D}\text{-}\mathsf{lin}}$ is well-defined and is indeed a functor since $\mathsf{F}_\oc$ is a functor. Furthermore, it is automatic by definition that $\mathsf{U} \circ \mathsf{F}_{\mathsf{D}\text{-}\mathsf{lin}} = \mathsf{F}_\oc$. \end{enumerate}
\end{proof} 

We will now prove the converse of Theorem \ref{thm1} by showing that a comonad whose coKleisli category is a Cartesian differential category is indeed a Cartesian differential comonad. 

\begin{proposition}\label{prop1} Let $\mathbb{X}$ be a category with finite biproducts and let $(\oc, \delta, \varepsilon)$ be a comonad on $\mathbb{X}$. Suppose that the coKleisli category $\mathbb{X}_\oc$ is a Cartesian differential category with differential combinator $\mathsf{D}$ such that: 
\begin{enumerate}[{\em (i)}]
\item The underlying Cartesian left additive structure of $\mathbb{X}_\oc$ is the one from Lemma \ref{cokleisliCLAC};
\item For every map $f: A \to B$ in $\mathbb{X}$, $\llbracket \mathsf{F}_\oc(f) \rrbracket$ is a $\mathsf{D}$-linear map in $\mathbb{X}_\oc$.
\end{enumerate}
Define the natural transformation $\partial_A: \oc(A \times A) \to \oc(A)$ as follows: 
\begin{equation}\label{partialdef}\begin{gathered}\partial_A := \xymatrixcolsep{5pc}\xymatrix{ \oc(A \times A) \ar[r]^-{\llbracket \mathsf{D}[\varphi_A] \rrbracket} & \oc(A) 
  } \end{gathered}\end{equation}
Then $(\oc, \delta, \varepsilon, \partial)$ is a Cartesian differential comonad and furthermore for every coKleisli map $\llbracket f \rrbracket: \oc(A) \to B$, the following diagram commutes (in $\mathbb{X}$): 
\begin{equation}\label{Dvarphi1}\begin{gathered} 
\xymatrixcolsep{5pc}\xymatrix{ \oc(A \times A) \ar[dr]_-{\llbracket \mathsf{D}[f] \rrbracket} \ar[r]^-{\partial_A} & \oc(A) \ar[d]^-{\llbracket f \rrbracket} \\
 & B }
 \end{gathered}\end{equation}
\end{proposition} 
\begin{proof} We begin by proving (\ref{Dvarphi1}) as it will be useful in other parts of the proof: 
\begin{align*}
\llbracket \mathsf{D}[f] \rrbracket &=~\llbracket \mathsf{D}\left[ \mathsf{F}_\oc\left( \llbracket f \rrbracket \right) \circ \varphi_A  \right]  \rrbracket \tag{Lem.\ref{cokleislilem1}.(\ref{cokleislilem1.varphi})} \\
&=~\llbracket \mathsf{F}_\oc\left( \llbracket f \rrbracket \right) \circ \mathsf{D}\left[  \varphi_A  \right]  \rrbracket \tag{$\mathsf{F}_\oc\left( \llbracket f \rrbracket \right)$ is $\mathsf{D}$-linear and Lem \ref{linlem}.(\ref{linlem.post})} \\
&=~ \llbracket f \rrbracket \circ \llbracket \mathsf{D}\left[  \varphi_A  \right] \rrbracket \tag{Lem.\ref{cokleislilem1}.(\ref{cokleislilem1.left})} \\
&=~ \llbracket f \rrbracket \circ \partial_A 
\end{align*}
So $\llbracket \mathsf{D}[f] \rrbracket = \llbracket f \rrbracket \circ \partial_A$. Next we prove that $\partial$ is natural: 
\begin{align*}
\partial_B \circ \oc(f \times f) &=~\llbracket \mathsf{D}[\varphi_B] \rrbracket \circ \oc(f \times f) \\ 
&=~ \llbracket \mathsf{D}[\varphi_B] \circ \mathsf{F}_\oc(f \times f)  \rrbracket  \tag{Lem.\ref{cokleislilem1}.(\ref{cokleislilem1.right})} \\
&=~\llbracket \mathsf{D}[\varphi_B] \circ \left( \mathsf{F}_\oc(f) \times  \mathsf{F}_\oc(f) \right)  \rrbracket  \tag{Lem.\ref{cokleisliproduct}.(\ref{cokleisliproduct.F})} \\ 
&=~\llbracket \mathsf{D}[\varphi_B \circ \mathsf{F}_\oc(f) ]  \rrbracket \tag{$\mathsf{F}_\oc(f)$ is $\mathsf{D}$-linear and Lem \ref{linlem}.(\ref{linlem.pre})} \\
&=~\llbracket \varphi_B \circ \mathsf{F}_\oc(f) \rrbracket \circ \partial_A \tag{\ref{Dvarphi1}} \\
&=~\oc(f) \circ \partial_A \tag{Lem.\ref{cokleislilem1}.(\ref{cokleislilem1.varphi2})} 
\end{align*}
So $\partial$ is a natural transformation. Next we show the six axioms of a differential combinator transformation: 
\begin{enumerate}[{\bf [dc.1]}] 
\item Here we use \textbf{[CD.2]}: 
\begin{align*}
\partial_A \circ \oc(\iota_0) &=~ \llbracket \mathsf{D}[\varphi_A] \rrbracket  \circ \oc(\iota_0) \\
&=~\llbracket \mathsf{D}[\varphi_A] \circ \mathsf{F}_\oc(\iota_0) \rrbracket  \tag{Lem.\ref{cokleislilem1}.(\ref{cokleislilem1.right})} \\
&=~\llbracket \mathsf{D}[\varphi_A] \circ \iota_0 \rrbracket \tag{Lem.\ref{cokleisliCLAC}.(\ref{cokleisliCLAC.F2})} \\
&=~\llbracket 0 \rrbracket \tag{\textbf{[CD.2]}} \\
&=~ 0 
\end{align*}
\item Here we use \textbf{[CD.2]}: 
\begin{align*}
\partial_A \circ \oc(1_A \times \nabla_A) &=~\llbracket \mathsf{D}[\varphi_A] \rrbracket  \circ \oc(1_A \times \nabla_A) \\
&=~\llbracket \mathsf{D}[\varphi_A] \circ \mathsf{F}_\oc(1_A \times \nabla_A) \rrbracket  \tag{Lem.\ref{cokleislilem1}.(\ref{cokleislilem1.right})} \\
&=~\llbracket \mathsf{D}[\varphi_A] \circ (1_A \times \nabla_A) \rrbracket \tag{Lem.\ref{cokleisliproduct}.(\ref{cokleisliproduct.F}) + Lem.\ref{cokleisliCLAC}.(\ref{cokleisliCLAC.F2})} \\
&=~\llbracket \mathsf{D}[\varphi_A] \circ (1_A \times\pi_0) + \mathsf{D}[\varphi_A] \circ (1_A \times\pi_1)  \rrbracket \tag{\textbf{[CD.2]}} \\
&=~ \llbracket \mathsf{D}[\varphi_A] \circ (1_A \times\pi_0) \rrbracket + \llbracket \mathsf{D}[\varphi_A] \circ (1_A \times\pi_1)  \rrbracket \\ 
&=~  \llbracket \mathsf{D}[\varphi_A] \circ \mathsf{F}_\oc(1_A \times\pi_0) \rrbracket + \llbracket \mathsf{D}[\varphi_A] \circ \mathsf{F}_\oc(1_A \times\pi_1)  \rrbracket \tag{Lem.\ref{cokleisliproduct}.(\ref{cokleisliproduct.F}) + Lem.\ref{cokleisliCLAC}.(\ref{cokleisliCLAC.F2})} \\
&=~\llbracket \mathsf{D}[\varphi_A] \rrbracket  \circ \oc(1_A \times \pi_0) +  \llbracket \mathsf{D}[\varphi_A] \rrbracket  \circ \oc(1_A \times \pi_1)  \tag{Lem.\ref{cokleislilem1}.(\ref{cokleislilem1.right})} \\
&=~\partial_A \circ \oc(1_A \times \pi_0) + \partial_A \circ \oc(1_A \times \pi_1) \\
&=~ \partial_A \circ \left( \oc(1_A \times \pi_0) + \oc(1_A \times \pi_0) \right) 
\end{align*}
\item Here we use \textbf{[CD.3]}: 
\begin{align*}
\varepsilon_A \circ \partial_A &=~\llbracket 1_A \rrbracket \circ \partial_A \\
&=~\llbracket \mathsf{D}[1_A] \rrbracket \tag{\ref{Dvarphi1}} \\
&=~\llbracket \pi_1 \rrbracket \tag{\textbf{[CD.3]}} \\
&=~\pi_1 \circ \varepsilon_{A \times A}
\end{align*}
\item Here we use \textbf{[CD.5]}: 
\begin{align*}
\delta_A \circ \partial_A &=~\llbracket \varphi_{\oc(A)} \circ \varphi_A \rrbracket \circ \partial_A \tag{Lem.\ref{cokleislilem1}.(\ref{cokleislilem1.varphi3})} \\
&=~\llbracket \mathsf{D}\left[ \varphi_{\oc(A)} \circ \varphi_A \right] \rrbracket  \tag{\ref{Dvarphi1}} \\
&=~ \llbracket \mathsf{D}\left[ \varphi_{\oc(A)} \right] \circ \langle \varphi_A \circ \pi_0, \mathsf{D}[\varphi_A] \rangle \rrbracket \tag{\textbf{[CD.5]}} \\
&=~\llbracket \mathsf{D}\left[ \varphi_{\oc(A)} \right] \rrbracket \circ \oc\left( \llbracket \langle \varphi_A \circ \pi_0, \mathsf{D}[\varphi_A] \rangle \rrbracket \right) \circ \delta_{A \times A} \\
&=~\partial_{\oc(A)} \circ \oc\left( \llbracket \langle \varphi_A \circ \pi_0, \mathsf{D}[\varphi_A] \rangle \rrbracket \right) \circ \delta_{A \times A} \\
&=~\partial_{\oc(A)} \circ \oc\left( \left \langle \llbracket \varphi_A \circ \pi_0 \rrbracket , \llbracket \mathsf{D}[\varphi_A] \rrbracket \right \rangle \right) \circ \delta_{A \times A} \\
&=~\partial_{\oc(A)} \circ \oc\left( \left \langle \llbracket \varphi_A \circ \pi_0 \rrbracket , \partial_A \right \rangle \right) \circ \delta_{A \times A} \\
&=~\partial_{\oc(A)} \circ \oc\left( \left \langle \llbracket \varphi_A \circ \mathsf{F}_\oc(\pi_0) \rrbracket , \partial_A \right \rangle \right) \circ \delta_{A \times A} \tag{Lem.\ref{cokleisliproduct}.(\ref{cokleisliproduct.F})} \\ 
&=~\partial_{\oc(A)} \circ \oc\left( \left \langle \oc(\pi_0) , \partial_A \right \rangle \right) \circ \delta_{A \times A} \tag{Lem.\ref{cokleislilem1}.(\ref{cokleislilem1.varphi2})} 
\end{align*}
\item Here we use \textbf{[CD.6]}: 
\begin{align*}
\partial_{A} \circ \partial_{A \times A} \circ \oc(\ell_A) &=~\llbracket \mathsf{D}[\varphi_{A}] \rrbracket \circ \partial_{A \times A} \circ \oc(\ell_A) \\
&=~ \llbracket \mathsf{D}\left[ \mathsf{D}[\varphi_{A}] \right] \rrbracket \circ \oc(\ell_A)  \tag{\ref{Dvarphi1}} \\
&=~ \llbracket \mathsf{D}\left[ \mathsf{D}[\varphi_{A}] \circ \mathsf{F}_\oc(\ell_A) \right] \rrbracket  \tag{Lem.\ref{cokleislilem1}.(\ref{cokleislilem1.right})} \\
&=~  \llbracket \mathsf{D}\left[ \mathsf{D}[\varphi_{A}] \circ \ell_A \right] \rrbracket \tag{Lem.\ref{cokleisliCLAC}.(\ref{cokleisliCLAC.F2})} \\
&=~\llbracket \mathsf{D}[\varphi_{A}] \rrbracket \tag{\textbf{[CD.6]}} \\
&=~ \partial_A 
\end{align*}
\item Here we use \textbf{[CD.7]}:
\begin{align*}
\partial_{A} \circ \partial_{A \times A} \circ \oc(c_A) &=~\llbracket \mathsf{D}[\varphi_{A}] \rrbracket \circ \partial_{A \times A} \circ \oc(c_A) \\
&=~ \llbracket \mathsf{D}\left[ \mathsf{D}[\varphi_{A}] \right] \rrbracket \circ \oc(c_A)  \tag{\ref{Dvarphi1}} \\
&=~ \llbracket \mathsf{D}\left[ \mathsf{D}[\varphi_{A}] \circ \mathsf{F}_\oc(c_A) \right] \rrbracket  \tag{Lem.\ref{cokleislilem1}.(\ref{cokleislilem1.right})} \\
&=~  \llbracket \mathsf{D}\left[ \mathsf{D}[\varphi_{A}] \circ c_A \right] \rrbracket \tag{Lem.\ref{cokleisliCLAC}.(\ref{cokleisliCLAC.F2})} \\
&=~\llbracket \mathsf{D}\left[ \mathsf{D}[\varphi_{A}] \right]  \rrbracket \tag{\textbf{[CD.7]}} \\
&=~\llbracket \mathsf{D}[\varphi_{A}] \rrbracket \circ \partial_{A \times A} \rrbracket \tag{\ref{Dvarphi1}} \\ 
&=~\partial_{A} \circ \partial_{A \times A} 
\end{align*}
\end{enumerate}
So we conclude that $\partial$ is a differential combinator transformation and therefore that $(\oc, \delta, \varepsilon, \partial)$ is a Cartesian differential comonad. 

\ \hfill  \end{proof} 

As a result, we obtain a bijective correspondence between differential combinator transformations and differential combinators. 

\begin{corollary} Let $\mathbb{X}$ be a category with finite biproducts and let $(\oc, \delta, \varepsilon)$ be a comonad on $\mathbb{X}$. Then the following are in bijective correspondence: 
\begin{enumerate}[{\em (i)}]
\item Differential combinator transformations $\partial$ on $(\oc, \delta, \varepsilon)$
\item Differential combinators $\mathsf{D}$ on the coKleisli category $\mathbb{X}_\oc$ with respect to the Cartesian left additive structure from Lemma \ref{cokleisliCLAC} and such that for every map $f: A \to B$ in $\mathbb{X}$, $\llbracket \mathsf{F}_\oc(f) \rrbracket: \oc(A) \to B$ is a $\mathsf{D}$-linear map in $\mathbb{X}_\oc$.
\end{enumerate}
via the constructions of Theorem \ref{thm1} and Proposition \ref{prop1}. 
\end{corollary} 
\begin{proof} This follows immediately from Theorem \ref{thm1}.(\ref{thm1.varphi}) and (\ref{Dvarphi1}). 
\end{proof} 

We now turn our attention back to the $\mathsf{D}$-linear maps in the coKleisli category of a Cartesian differential comonad. Specifically, we wish to provide necessary and sufficient conditions for when the subcategory of $\mathsf{D}$-linear maps is isomorphic to the base category. Explicitly, we wish to study when $\mathsf{F}_{\mathsf{D}\text{-}\mathsf{lin}}: \mathbb{X} \to \mathsf{D}\text{-}\mathsf{lin}[\mathbb{X}_\oc]$ as defined in Theorem \ref{thm1}.(\ref{Flindef}) is an isomorphism. The answer, as it turns out, is requiring that the comonad counit has a section. 

\begin{definition}\label{def:Dunit} Let $(\oc, \delta, \varepsilon, \partial)$ be a Cartesian differential comonad on a category $\mathbb{X}$ with finite biproducts. A \textbf{$\mathsf{D}$-linear unit} on $(\oc, \delta, \varepsilon, \partial)$ is a natural transformation $\eta_A: A \to \oc(A)$ such that the following diagrams commute: 
\begin{enumerate}[{\bf [du.1]}] 
\item Linear Rule:
  \[  \xymatrixcolsep{5pc}\xymatrix{ A \ar@{=}[dr]_-{} \ar[r]^-{\eta_A} & \oc(A) \ar[d]^-{\varepsilon_A} \\
  & A
  } \]
\item Linearization Rule: 
  \[  \xymatrixcolsep{5pc}\xymatrix{ \oc(A) \ar[r]^-{\varepsilon_A} \ar[d]_-{\oc(\iota_1)} & A \ar[d]^-{\eta_A} \\  
\oc(A \times A) \ar[r]_-{\partial_A} & \oc(A)
  } \]  
where $\iota_1: A \to A \times A$ is defined as in Definition \ref{CLACmapsdef}.(\ref{injdef}). 
\end{enumerate}
In other words, for every object $A$, $\partial_A \circ \oc(\iota_1)$ is a split idempotent via $\eta_A$ and $\varepsilon_A$. 
\end{definition}

Our first observation is that $\mathsf{D}$-linear units are unique. 

\begin{lemma} For a Cartesian differential comonad, if a $\mathsf{D}$-linear unit exists, then it is unique. 
\end{lemma}
\begin{proof}  Let $(\oc, \delta, \varepsilon, \partial)$ be a Cartesian differential comonad on a category $\mathbb{X}$ with finite biproducts. Suppose that $\eta$ and $\eta^\prime$ are two $\mathsf{D}$-linear units on $(\oc, \delta, \varepsilon, \partial)$. Combining the linear rule \textbf{[du.1]} and the linearization rule \textbf{[du.2]}, we compute: 
\begin{align*}
\eta^\prime_A &=~\eta^\prime_A \circ 1_A \\ 
&=~\eta^\prime_A \circ \varepsilon_A \circ \eta_A \tag{\textbf{[du.1]} for $\eta$} \\
&=~\partial_A \circ \oc(\iota_1) \circ \eta_A \tag{\textbf{[du.2]} for $\eta^\prime$} \\
&=~\eta_A \circ \varepsilon_A \circ \eta_A \tag{\textbf{[du.2]} for $\eta$} \\
&=~ \eta_A \circ 1_A \tag{\textbf{[du.1]} for $\eta$} \\
&=~\eta_A 
\end{align*}
So $\eta= \eta^\prime$. Therefore we conclude that if it exists, a $\mathsf{D}$-linear unit must be unique. \end{proof} 

We now prove that for a Cartesian differential comonad with a $\mathsf{D}$-linear unit, the $\mathsf{D}$-linear maps in the coKleisli category correspond precisely to the maps in the base category. To do so, we will first need the following useful identity: 

\begin{lemma} \label{Lvarphi}  Let $(\oc, \delta, \varepsilon, \partial)$ be a Cartesian differential comonad on a category $\mathbb{X}$ with finite biproducts. Then $\llbracket \mathsf{L}[\varphi_A] \rrbracket = \partial_A \circ \oc(\iota_1)$. 
%\begin{equation}\begin{gathered} 
%\xymatrixcolsep{5pc}\xymatrix{ \oc(A) \ar[dr]^-{\llbracket \mathsf{L}[\varphi_A] \rrbracket} \ar[d]_-{\oc(\iota_1)} \\  
%\oc(A \times A) \ar[r]_-{\partial_A} & \oc(A) } 
 %\end{gathered}\end{equation}
\end{lemma}
\begin{proof} We compute: 
 \begin{align*}
\llbracket \mathsf{L}[\varphi_A] \rrbracket &=~\llbracket \mathsf{D}[\varphi_A] \circ \iota_1 \rrbracket \\
&=~\llbracket \mathsf{D}[\varphi_A] \circ \mathsf{F}_{\mathsf{D}\text{-}\mathsf{lin}}\left( \oc(\iota_1) \right) \rrbracket   \tag{Lem.\ref{cokleisliCLAC}.(\ref{cokleisliCLAC.F2})} \\
&=~ \llbracket \mathsf{D}[\varphi_A]  \rrbracket \circ \oc(\iota_1)  \tag{Lem.\ref{cokleislilem1}.(\ref{cokleislilem1.right})} \\ 
&=~\partial_A \circ \oc(\iota_1) \tag{Theorem \ref{thm1}.(\ref{thm1.varphi})} 
\end{align*}
So the desired equality holds. 
\end{proof}

\begin{proposition}
    \label{etaFlem1} Let $(\oc, \delta, \varepsilon, \partial)$ be a Cartesian differential comonad on a category $\mathbb{X}$ with finite biproducts. Then the following are equivalent: 
\begin{enumerate}[{\em (i)}]
\item $\mathsf{F}_{\mathsf{D}\text{-}\mathsf{lin}}: \mathbb{X} \to \mathsf{D}\text{-}\mathsf{lin}[\mathbb{X}_\oc]$ is an isomorphism (where $\mathsf{F}_{\mathsf{D}\text{-}\mathsf{lin}}$ is defined as in Theorem \ref{thm1}.(\ref{Flindef}));
\item $(\oc, \delta, \varepsilon, \partial)$ has a $\mathsf{D}$-linear unit $\eta_A : A \to \oc (A)$. 
\end{enumerate}
\end{proposition}
\begin{proof} Suppose that $\mathsf{F}_{\mathsf{D}\text{-}\mathsf{lin}}: \mathbb{X} \to \mathsf{D}\text{-}\mathsf{lin}[\mathbb{X}_\oc]$ is an isomorphism. By Lemma \ref{linlem}.(\ref{linlemimportant1}), $\llbracket \mathsf{L}[\varphi_A] \rrbracket: \oc(A) \to \oc(A)$ is a $\mathsf{D}$-linear map from $A$ to $\oc(A)$ in the coKleisli category. Thus, we obtain a map of the desired type ${\mathsf{F}_{\mathsf{D}\text{-}\mathsf{lin}}^{-1}\left( \llbracket \mathsf{L}[\varphi_A] \rrbracket \right): A \to \oc(A)}$ in $\mathbb{X}$. So define ${\eta_A: A \to \oc(A)}$ as: 
\begin{equation}\label{etavarphi}\begin{gathered} 
\eta_A = \mathsf{F}_{\mathsf{D}\text{-}\mathsf{lin}}^{-1}\left( \llbracket \mathsf{L}[\varphi_A] \rrbracket \right)
 \end{gathered}\end{equation}
We will use Lemma \ref{Lvarphi} to show that $\eta$ is indeed a $\mathsf{D}$-linear unit. Starting with the naturality of $\eta$, for any map $f: A \to B$ in $\mathbb{X}$ we compute: 
\begin{align*}
\eta_B \circ f &=~\mathsf{F}_{\mathsf{D}\text{-}\mathsf{lin}}^{-1}\left( \llbracket \mathsf{L}[\varphi_B] \rrbracket \right) \circ f \\
&=~\mathsf{F}_{\mathsf{D}\text{-}\mathsf{lin}}^{-1}\left( \llbracket \mathsf{L}[\varphi_B] \rrbracket \right) \circ \mathsf{F}_{\mathsf{D}\text{-}\mathsf{lin}}^{-1}\left( \mathsf{F}_{\mathsf{D}\text{-}\mathsf{lin}}\left( f \right) \right) \tag{$\mathsf{F}_{\mathsf{D}\text{-}\mathsf{lin}}$ is an isomorphism} \\
&=~\mathsf{F}_{\mathsf{D}\text{-}\mathsf{lin}}^{-1}\left( \left \llbracket \mathsf{L}[\varphi_B] \circ  \mathsf{F}_{\mathsf{D}\text{-}\mathsf{lin}}\left( f \right) \right \rrbracket  \right) \tag{$\mathsf{F}_{\mathsf{D}\text{-}\mathsf{lin}}$ is a functor} \\
&=~\mathsf{F}_{\mathsf{D}\text{-}\mathsf{lin}}^{-1}\left( \left \llbracket \mathsf{L}[\varphi_B] \circ  \mathsf{F}_\oc \left( f \right) \right \rrbracket  \right) \\
&=~\mathsf{F}_{\mathsf{D}\text{-}\mathsf{lin}}^{-1}\left( \llbracket \mathsf{L}[\varphi_B] \rrbracket \circ \oc(f) \right)  \tag{Lem.\ref{cokleislilem1}.(\ref{cokleislilem1.right})} \\ 
&=~ \mathsf{F}_{\mathsf{D}\text{-}\mathsf{lin}}^{-1}\left( \partial_B \circ \oc(\iota_1) \circ \oc(f) \right)  \tag{Lemma \ref{Lvarphi}} \\
&=~\mathsf{F}_{\mathsf{D}\text{-}\mathsf{lin}}^{-1}\left(  \partial_B \circ \oc\left( \iota_1 \circ f \right) \right)  \tag{$\oc$ is a functor} \\
&=~\mathsf{F}_{\mathsf{D}\text{-}\mathsf{lin}}^{-1}\left(  \partial_B \circ \oc\left( (f \times f) \circ \iota_1 \right) \right)  \tag{Naturality of $\iota_1$} \\
&=~\mathsf{F}_{\mathsf{D}\text{-}\mathsf{lin}}^{-1}\left(  \partial_B \circ \oc(f \times f) \circ \oc(\iota_1) \right)   \tag{$\oc$ is a functor} \\
&=~\mathsf{F}_{\mathsf{D}\text{-}\mathsf{lin}}^{-1}\left(  \oc(f) \circ \partial_A \circ \oc(\iota_1) \right) \tag{Naturality of $\partial$} \\
&=~\mathsf{F}_{\mathsf{D}\text{-}\mathsf{lin}}^{-1}\left(  \oc(f) \circ \llbracket \mathsf{L}[\varphi_A] \rrbracket  \right)\tag{Lemma \ref{Lvarphi}} \\
&=~\mathsf{F}_{\mathsf{D}\text{-}\mathsf{lin}}^{-1}\left(   \llbracket \mathsf{F}_\oc\left( \oc(f) \right) \circ \mathsf{L}[\varphi_A] \rrbracket  \right)  \tag{Lem.\ref{cokleislilem1}.(\ref{cokleislilem1.left})} \\ 
&=~\mathsf{F}_{\mathsf{D}\text{-}\mathsf{lin}}^{-1}\left( \left \llbracket  \mathsf{F}_{\mathsf{D}\text{-}\mathsf{lin}}\left( \oc(f) \right) \circ \mathsf{L}[\varphi_A] \right \rrbracket  \right) \\
&=~\mathsf{F}_{\mathsf{D}\text{-}\mathsf{lin}}^{-1}\left( \mathsf{F}_{\mathsf{D}\text{-}\mathsf{lin}}\left( \oc(f) \right) \right) \circ \mathsf{F}_{\mathsf{D}\text{-}\mathsf{lin}}^{-1}\left( \llbracket \mathsf{L}[\varphi_A] \rrbracket \right) \tag{$\mathsf{F}^{-1}_{\mathsf{D}\text{-}\mathsf{lin}}$ is a functor} \\
&=~\oc(f) \circ \mathsf{F}_{\mathsf{D}\text{-}\mathsf{lin}}^{-1}\left( \llbracket \mathsf{L}[\varphi_A] \rrbracket \right) \tag{$\mathsf{F}_{\mathsf{D}\text{-}\mathsf{lin}}$ is an isomorphism} \\
&=~ \oc(f) \circ \eta_A 
\end{align*}
So $\eta$ is a natural transformation. Next we show the two axioms of a $\mathsf{D}$-linear unit: 
\begin{enumerate}[{\bf [du.1]}] 
\item We compute:  
\begin{align*}
\varepsilon_A \circ \eta_A &=~ \mathsf{F}_{\mathsf{D}\text{-}\mathsf{lin}}^{-1}\left( \llbracket \mathsf{F}_{\mathsf{D}\text{-}\mathsf{lin}}\left( \varepsilon_A \right) \rrbracket \right) \circ \eta_A \\
&=~\mathsf{F}_{\mathsf{D}\text{-}\mathsf{lin}}^{-1}\left( \llbracket \mathsf{F}_\oc \left( \varepsilon_A \right) \rrbracket \right) \circ \eta_A \\
&=~\mathsf{F}_{\mathsf{D}\text{-}\mathsf{lin}}^{-1}\left( \llbracket \mathsf{F}_\oc \left( \llbracket 1_A \rrbracket \right) \rrbracket \right) \circ \eta_A \\
&=~ \mathsf{F}_{\mathsf{D}\text{-}\mathsf{lin}}^{-1}\left( \llbracket \mathsf{F}_\oc \left( \llbracket 1_A \rrbracket \right) \rrbracket \right) \circ  \mathsf{F}_{\mathsf{D}\text{-}\mathsf{lin}}^{-1}\left( \llbracket \mathsf{L}[\varphi_A] \rrbracket \right) \\
&=~ \mathsf{F}_{\mathsf{D}\text{-}\mathsf{lin}}^{-1}\left( \llbracket \mathsf{F}_\oc \left( \llbracket 1_A \rrbracket \right) \circ  \mathsf{L}[\varphi_A] \rrbracket \right) \\
&=~  \mathsf{F}_{\mathsf{D}\text{-}\mathsf{lin}}^{-1}\left( \llbracket \mathsf{F}_\oc \left( \llbracket 1_A \rrbracket \right) \circ  \mathsf{D}[\varphi_A] \circ \iota_1 \rrbracket \right) \\
&=~ \mathsf{F}_{\mathsf{D}\text{-}\mathsf{lin}}^{-1}\left( \llbracket \mathsf{D}\left[  \mathsf{F}_\oc \left( \llbracket 1_A \rrbracket \right) \circ \varphi_A \right] \circ \iota_1 \rrbracket \right) \tag{$\mathsf{F}_\oc\left( \llbracket \varepsilon \rrbracket \right)$ is $\mathsf{D}$-linear and Lem \ref{linlem}.(\ref{linlem.post})} \\
&=~\mathsf{F}_{\mathsf{D}\text{-}\mathsf{lin}}^{-1}\left( \llbracket \mathsf{D}[1_A] \circ \iota_1 \rrbracket \right)  \tag{Lem.\ref{cokleislilem1}.(\ref{cokleislilem1.varphi})} \\ 
&=~\mathsf{F}_{\mathsf{D}\text{-}\mathsf{lin}}^{-1}\left( \llbracket \pi_1 \circ \iota_1 \rrbracket \right) \tag{\textbf{[CD.3]}} \\
&=~ \mathsf{F}_{\mathsf{D}\text{-}\mathsf{lin}}^{-1}\left( \llbracket 1_A \rrbracket \right) \tag{Biproduct identities} \\
&=~ 1_A \tag{$\mathsf{F}^{-1}_{\mathsf{D}\text{-}\mathsf{lin}}$ is a functor} 
\end{align*}
\item We compute: 
\begin{align*}
\eta_A \circ \varepsilon_A &=~\mathsf{F}_{\mathsf{D}\text{-}\mathsf{lin}}\left( \eta_A  \right) \\
&=~\mathsf{F}_{\mathsf{D}\text{-}\mathsf{lin}}\left(  \mathsf{F}_{\mathsf{D}\text{-}\mathsf{lin}}^{-1}\left( \llbracket \mathsf{L}[\varphi_A] \rrbracket \right)  \right) \\
&=~ \llbracket \mathsf{L}[\varphi_A] \rrbracket \tag{$\mathsf{F}_{\mathsf{D}\text{-}\mathsf{lin}}$ is an isomorphism} \\
&=~\partial_A \circ \oc(\iota_1)  \tag{Lemma \ref{Lvarphi}} 
\end{align*}
\end{enumerate}
So we conclude that $\eta$ is a $\mathsf{D}$-linear unit. Conversely, suppose that $(\oc, \delta, \varepsilon, \partial)$ comes equipped with a $\mathsf{D}$-linear unit $\eta_A: A \to \oc(A)$. Define the functor $\mathsf{F}_{\mathsf{D}\text{-}\mathsf{lin}}^{-1}: \mathsf{D}\text{-}\mathsf{lin}[\mathbb{X}_\oc] \to \mathbb{X}$ on objects as $\mathsf{F}_{\mathsf{D}\text{-}\mathsf{lin}}^{-1}(A) = A$ and on $\mathsf{D}$-linear coKleisli maps $\llbracket f \rrbracket: \oc(A) \to B$ as the following composite: 
\begin{align*}
\mathsf{F}_{\mathsf{D}\text{-}\mathsf{lin}}^{-1}(\llbracket f \rrbracket ) :=  \xymatrixcolsep{5pc}\xymatrix{A \ar[r]^-{\eta_A} & \oc (A) \ar[r]^-{\llbracket f \rrbracket}  & B
  }  && \mathsf{F}_{\mathsf{D}\text{-}\mathsf{lin}}^{-1}(\llbracket f \rrbracket ) = \llbracket f \rrbracket \circ \eta_A 
\end{align*}
We must show that $\mathsf{F}_{\mathsf{D}\text{-}\mathsf{lin}}^{-1}$ is indeed a functor. We first show that $\mathsf{F}_{\mathsf{D}\text{-}\mathsf{lin}}^{-1}$ preserves identities using the linear rule \textbf{[du.1]}:
\begin{align*}
\mathsf{F}_{\mathsf{D}\text{-}\mathsf{lin}}^{-1}\left( \llbracket 1_A \rrbracket \right) &=~ \llbracket 1_A \rrbracket \circ \eta_A \\
&=~ \varepsilon_A \circ \eta_A \\
&=~ 1_A \tag{\textbf{[du.1]}}
\end{align*}
Next we show that $\mathsf{F}_{\mathsf{D}\text{-}\mathsf{lin}}^{-1}$ also preserves composition using the linearization rule \textbf{[du.2]}:
\begin{align*}
\mathsf{F}_{\mathsf{D}\text{-}\mathsf{lin}}^{-1}\left( \llbracket g\circ f \rrbracket \right) &=~  \llbracket g\circ f \rrbracket \circ \eta_A \\
&=~  \llbracket g \rrbracket \circ \oc\left(  \llbracket  f \rrbracket  \right) \circ \delta_A \circ \eta_A \\
&=~ \llbracket g \rrbracket \circ \oc\left(  \llbracket  f \rrbracket  \circ \partial_A \circ \oc(\iota_1)  \right) \circ \delta_A \circ \eta_A \tag{$ \llbracket  f \rrbracket $ is $\mathsf{D}$-linear and Thm.\ref{thm1}.(\ref{thm1.lin})} \\
&=~  \llbracket g \rrbracket \circ \oc\left(  \llbracket  f \rrbracket \circ \eta_A \circ \varepsilon_A \right) \circ \delta_A \circ \eta_A \tag{\textbf{[du.2]}}  \\
&=~ \llbracket g \rrbracket \circ \oc\left(  \llbracket  f \rrbracket \right) \circ \oc(\eta_A) \circ \oc(\varepsilon_A) \circ \delta_A \circ \eta_A \tag{$\oc$ is a functor} \\
&=~  \llbracket g \rrbracket \circ \oc\left(  \llbracket  f \rrbracket \right) \circ \oc(\eta_A) \circ 1_{\oc(A)} \circ \eta_A \tag{Comonad Identity} \\
&=~ \llbracket g \rrbracket \circ \oc\left(  \llbracket  f \rrbracket \right) \circ \oc(\eta_A) \circ \eta_A \\
&=~\llbracket g \rrbracket \circ \oc\left(  \llbracket  f \rrbracket \right) \circ \eta_{\oc(A)} \circ \eta_A \tag{Naturality of $\eta$} \\
&=~\llbracket g \rrbracket \circ \eta_{B} \circ \llbracket f \rrbracket \circ \eta_A \tag{Naturality of $\eta$} \\
&=~\mathsf{F}_{\mathsf{D}\text{-}\mathsf{lin}}^{-1}\left( \llbracket g \rrbracket \right) \circ \mathsf{F}_{\mathsf{D}\text{-}\mathsf{lin}}^{-1}\left( \llbracket f \rrbracket \right)
\end{align*}
So $\mathsf{F}_{\mathsf{D}\text{-}\mathsf{lin}}^{-1}$ is indeed a functor. Next we show that $\mathsf{F}_{\mathsf{D}\text{-}\mathsf{lin}}$ and $\mathsf{F}_{\mathsf{D}\text{-}\mathsf{lin}}^{-1}$ are inverses. Starting with $\mathsf{F}_{\mathsf{D}\text{-}\mathsf{lin}}^{-1} \circ \mathsf{F}_{\mathsf{D}\text{-}\mathsf{lin}}$, clearly on objects we have that $\mathsf{F}_{\mathsf{D}\text{-}\mathsf{lin}}^{-1}\left( \mathsf{F}_{\mathsf{D}\text{-}\mathsf{lin}}(A) \right) = A$, while for maps we compute: 
\begin{align*}
\mathsf{F}_{\mathsf{D}\text{-}\mathsf{lin}}^{-1}\left( \llbracket \mathsf{F}_{\mathsf{D}\text{-}\mathsf{lin}}(f) \rrbracket \right) &=~\llbracket \mathsf{F}_{\mathsf{D}\text{-}\mathsf{lin}}(f) \rrbracket \circ \eta_A \\
&=~f \circ \varepsilon_A \circ \eta_A \\
&=~ f \circ 1_A \tag{\textbf{[du.1]}} \\
&=~ f
\end{align*}
Therefore, $\mathsf{F}_{\mathsf{D}\text{-}\mathsf{lin}}^{-1} \circ \mathsf{F}_{\mathsf{D}\text{-}\mathsf{lin}} = 1_\mathbb{X}$. Next for $\mathsf{F}_{\mathsf{D}\text{-}\mathsf{lin}} \circ \mathsf{F}^{-1}_{\mathsf{D}\text{-}\mathsf{lin}}$, clearly on objects we have that $\mathsf{F}_{\mathsf{D}\text{-}\mathsf{lin}}\left( \mathsf{F}^{-1}_{\mathsf{D}\text{-}\mathsf{lin}}(A) \right) = A$, while for maps we compute: 
\begin{align*}
\mathsf{F}_{\mathsf{D}\text{-}\mathsf{lin}}\left( \mathsf{F}^{-1}_{\mathsf{D}\text{-}\mathsf{lin}}\left( \llbracket f \rrbracket \right) \right) &=~ \mathsf{F}^{-1}_{\mathsf{D}\text{-}\mathsf{lin}}\left( \llbracket f \rrbracket \right) \circ \varepsilon_A  \\
&=~\llbracket f \rrbracket \circ \eta_A \circ \varepsilon_A \\
&=~ \llbracket f \rrbracket \circ \partial_A \circ \oc(\iota_1) \tag{\textbf{[du.2]}} \\
&=~ \llbracket f \rrbracket \tag{$ \llbracket  f \rrbracket $ is $\mathsf{D}$-linear and Thm.\ref{thm1}.(\ref{thm1.lin})} 
\end{align*}
Therefore, $\mathsf{F}_{\mathsf{D}\text{-}\mathsf{lin}} \circ \mathsf{F}^{-1}_{\mathsf{D}\text{-}\mathsf{lin}} = 1_{\mathsf{D}\text{-}\mathsf{lin}[\mathbb{X}_\oc]}$. So we conclude that $\mathsf{F}_{\mathsf{D}\text{-}\mathsf{lin}}$ is an isomorphism. 
\end{proof} 

As a result, in the presence of a $\mathsf{D}$-linear unit, we obtain the following characterizations of $\mathsf{D}$-linear maps. 

\begin{corollary}\label{etacor1} Let $(\oc, \delta, \varepsilon, \partial)$ be a Cartesian differential comonad on a category $\mathbb{X}$ with finite biproducts. If $(\oc, \delta, \varepsilon, \partial)$ has a $\mathsf{D}$-linear unit $\eta$, then the following are equivalent for a coKleisli map $\llbracket f \rrbracket: \oc(A) \to B$, 
\begin{enumerate}[{\em (i)}]
    \item $\llbracket f \rrbracket$ is $\mathsf{D}$-linear in $\mathbb{X}_\oc$
    \item There exists a (necessarily unique) map $g: A \to B$ in $\mathbb{X}$ such that $\llbracket f \rrbracket = g \circ \varepsilon_A = \llbracket \mathsf{F}_\oc(g ) \rrbracket$. 
    \item $\llbracket f \rrbracket \circ \eta_A \circ \varepsilon_A = \llbracket f \rrbracket$
\end{enumerate}
\end{corollary}
\begin{proof} For $(i) \Rightarrow (ii)$, suppose that $\llbracket f \rrbracket$ is $\mathsf{D}$-linear. Then set $g = \llbracket f \rrbracket \circ \eta_A = \mathsf{F}^{-1}_{\mathsf{D}\text{-}\mathsf{lin}} (\llbracket f \rrbracket)$. By Proposition \ref{etaFlem1}, we clearly have that $\llbracket f \rrbracket = g \circ \varepsilon_A = \llbracket \mathsf{F}_\oc(g ) \rrbracket$, and also that $g$ is unique. For $(ii) \Rightarrow (iii)$, suppose that $\llbracket f \rrbracket = g \circ \varepsilon_A$. Then by \textbf{[du.1]}, we have that: 
\begin{align*}
    \llbracket f \rrbracket \circ \eta_A \circ \varepsilon_A &=~ g \circ \varepsilon_A \circ \eta_A \circ \varepsilon_A \\
    &=~ g \circ 1_A \circ \varepsilon_A \tag{\textbf{[du.1]}} \\
    &=~ g \circ \varepsilon_A \\
    &=~  \llbracket f \rrbracket
\end{align*}
Lastly, for $(iii) \Rightarrow (i)$, suppose that $\llbracket f \rrbracket \circ \eta_A \circ \varepsilon_A = \llbracket f \rrbracket$. By \textbf{[du.2]}, this implies that $\llbracket f \rrbracket \circ \partial_A \circ \oc(\iota_1) = \llbracket f \rrbracket$. However by Theorem \ref{thm1}.(\ref{thm1.lin}), this implies that $\llbracket f \rrbracket$ is $\mathsf{D}$-linear. 
\end{proof}

We conclude with some examples of Cartesian differential comonads and, if it exists, their $\mathsf{D}$-linear units. 

\begin{example} \label{ex:diffcat} \normalfont The main example of a Cartesian differential comonad is the comonad of a differential category. Briefly, a differential category \cite[Definition 2.4]{blute2006differential} is an additive symmetric monoidal category $\mathbb{X}$ equipped with a comonad $(\oc, \delta, \varepsilon)$, two natural transformations $\Delta_A: \oc(A) \to \oc(A) \otimes \oc(A)$ and $e_A: \oc(A) \to I$ such that $\oc(A)$ is a cocommutative comonoid, and a natural transformation called a deriving transformation $\mathsf{d}_A: \oc(A) \otimes A \to \oc(A)$  satisfying certain coherences which capture the basic properties of differentiation \cite[Definition 7]{Blute2019}. By \cite[Proposition 3.2.1]{blute2009cartesian}, for a differential category $\mathbb{X}$ with finite products, its coKleisli category $\mathbb{X}_\oc$ is a Cartesian differential category where the differential combinator is defined using the deriving transformation. For a coKleisli map $\llbracket f \rrbracket: \oc A \to B$, its derivative $\llbracket \mathsf{D}[f] \rrbracket: \oc(A \times A) \to B$ is defined as: 
\[ \llbracket \mathsf{D}[f] \rrbracket :=  \xymatrixcolsep{2.75pc}\xymatrix{\oc(A \times A) \ar[r]^-{\Delta_{A \times A}} & \oc(A \times A) \otimes \oc(A \times A) \ar[r]^-{\oc(\pi_0) \otimes \oc(\pi_1)} &  \oc(A) \otimes \oc(A) \ar[r]^-{1_{\oc(A)} \otimes \varepsilon_A} & \oc(A) \otimes A \ar[r]^-{\mathsf{d}_A} & \oc(A) \ar[r]^-{ f } & B 
 } \]
Applying Proposition \ref{prop1}, we obtain a differential combinator transformation:
\[ \partial_A := \xymatrixcolsep{3.75pc}\xymatrix{\oc(A \times A) \ar[r]^-{\Delta_{A \times A}} & \oc(A \times A) \otimes \oc(A \times A) \ar[r]^-{\oc(\pi_0) \otimes \oc(\pi_1)} & \oc(A) \otimes \oc(A) \ar[r]^-{1_{\oc(A)} \otimes \varepsilon_A} & \oc(A) \otimes A \ar[r]^-{\mathsf{d}_A} & \oc(A)
 } \]
Furthermore, if there exists a natural transformation $u_A: I \to \oc(A)$ such that $e_A \circ u_A = 1_I$ and $u_A \circ e_A = \oc(0)$, then we obtain a $\mathsf{D}$-linear unit defined as follows: 
  \[  \eta_A:= \xymatrixcolsep{5pc}\xymatrix{ A \ar[r]^-{\lambda^{-1}_A} & I \otimes A \ar[r]^-{u_A \otimes 1_A} & \oc(A) \otimes A \ar[r]^-{\mathsf{d}_A} & \oc(A)
  } \]  
Readers familiar with differential linear logic will note that any differential \emph{storage} category \cite[Definition 4.10]{blute2006differential} has such a map $u$ and that in this case the $\mathsf{D}$-linear unit is precisely the codereliction \cite[Section 5]{Blute2019}. However, we stress that it is possible to have a $\mathsf{D}$-linear unit for differential categories that are not differential storage categories. We invite the reader to see \cite[Section 9]{Blute2019} and \cite[Example 4.7]{garner2020cartesian} for lists of examples of differential categories. 
\end{example}

\begin{example} \label{ex:CDM} \normalfont Our three main novel examples of Cartesian differential comonads that we introduce in Sections \ref{sec:PWex}, \ref{secpuisdiv}, and \ref{sec:ZAex} below, arise instead more naturally as the dual notion, which we simply call \textbf{coCartesian differential monads}. Following the convention in the differential category literature for the dual notion of differential categories, we have elected to keep the same terminology and notation for the dual notion of a differential combinator transformation. Briefly, a coCartesian differential monad on a category $\mathbb{X}$ with finite biproducts is a quadruple $(\mathsf{S}, \mu, \eta, \partial)$ consisting of a monad $(\mathsf{S}, \mu, \eta)$ (where  ${\mu_A: \mathsf{S}\mathsf{S}(A) \to \mathsf{S}(A)}$ and $\eta_A: A \to \mathsf{S}(A)$) and a natural transformation $\partial_A: \mathsf{S}(A) \to \mathsf{S}(A \times A)$, again called a differential combinator transformation, such that the dual diagrams of Definition \ref{def:cdcomonad} commute. By the dual statement of Proposition \ref{prop1}, the opposite category of the Kleisli category of a coCartesian differential monad is a Cartesian differential category.  The dual notion of a $\mathsf{D}$-linear unit is called a $\mathsf{D}$-linear counit, which would be a natural transformation $\varepsilon_A: \mathsf{S}(A) \to A$ such that the dual diagrams of Definition \ref{def:Dunit} commute. By the dual statement of Proposition \ref{etaFlem1}, the existence of a $\mathsf{D}$-linear counit implies that the opposite of the base category is isomorphic to the subcategory of the $\mathsf{D}$-linear of the opposite of the Kleisli category. 
\end{example}

The following are two ``trivial'' examples of Cartesian differential categories any category with finite biproducts. While both are ``trivial'' in their own way, they both provide simple separating examples. Indeed, the first is an example of a Cartesian differential comonad without a  $\mathsf{D}$-linear unit, while the second is a Cartesian differential comonad which is not induced by a differential category.

\begin{example} \normalfont Let $\mathbb{X}$ be a category with finite biproducts, and let $\top$ be the chosen zero object. Then the constant comonad $\mathsf{C}$ which sends every object to the zero object $\mathsf{C}(A) = \top$ and every map to zero maps $\mathsf{C}(f) = 0$ is a Cartesian differential comonad whose differential combinator transformation is simply $0$. This Cartesian differential comonad has a $\mathsf{D}$-linear unit if and only if every object of $\mathbb{X}$ is a zero object.  
\end{example}

\begin{example} \normalfont \label{ex:identity} Let $\mathbb{X}$ be a category with finite biproducts. Then the identity comonad $1_{\mathbb{X}}$ is a Cartesian differential comonad whose differential combinator transformation is the second projection $\pi_1: A \times A \to A$ and has a $\mathsf{D}$-linear unit given by the identity map $1_A: A \to A$. The resulting coKleisli category is simply the entire base category $\mathbb{X}$ and whose differential combinator the same as in Example \ref{ex:CDCbiproduct}. As such, this example recaptures Example \ref{ex:CDCbiproduct} that every category with finite biproducts is a Cartesian differential category where every map is $\mathsf{D}$-linear. \end{example}

\section{Cartesian Differential Abstract coKleisli Categories}\label{sec:abstract}

The goal of this section is to give a precise characterization of the Cartesian differential categories which are the coKleisli categories of Cartesian differential comonads. This is a generalization of the work done by Blute, Cockett, and Seely in \cite{blute2015cartesian}, where they characterize which Cartesian differential categories are the coKleisli categories of the comonads of differential categories. This was achieved using the concept of abstract coKleisli categories \cite[Section 2.4]{blute2015cartesian}, which is the dual notion of thunk-force-categories as introduced by F\"{u}hrmann in \cite{fuhrmann1999direct}. Abstract coKleisli categories provide a direct description of the structure of coKleisli categories in such a way that the coKleisli category of a comonad is an abstract coKleisli category and, conversely, every abstract coKleisli category is canonically the coKleisli category of a comonad on a certain subcategory. As such, here we introduced Cartesian differential abstract coKleisli categories which, as the name suggests, are abstract coKleisli categories that are also Cartesian differential categories such that the differential combinator and abstract coKleisli structure are compatible. We show that the coKleisli category of a Cartesian differential comonad is a Cartesian differential abstract coKleisli categories and that, conversely, every Cartesian differential abstract coKleisli category is canonically the coKleisli category of a Cartesian differential comonad on a certain subcategory. We will also study the $\mathsf{D}$-linear maps of Cartesian differential abstract coKleisli categories. 

We will start from the abstract coKleisli side of the story. 

\begin{definition}\label{def:abstract} An \textbf{abstract coKleisli structure} on a category $\mathbb{X}$ is a triple $(\oc, \varphi, \epsilon)$ consisting of an endofunctor $\oc: \mathbb{X} \to \mathbb{X}$, a natural transformation $\varphi_A: A \to \oc(A)$, and a family of maps $\epsilon_A: \oc(A) \to A$ (which are not necessarily natural), such that $\epsilon_{\oc(A)}: \oc\oc(A) \to \oc(A)$ is a natural transformation and the following diagrams commute: 
  \begin{equation}\label{abstracteq}\begin{gathered} \xymatrixcolsep{5pc}\xymatrix{ A \ar@{=}[dr]^-{} \ar[r]^-{\varphi_A} & \oc(A) \ar[d]^-{\epsilon_A} & \oc(A) \ar@{=}[dr]^-{} \ar[r]^-{\oc(\varphi_A)} & \oc\oc(A) \ar[d]^-{\epsilon_{\oc(A)}} & \oc\oc(A) \ar[r]^-{\epsilon_{\oc(A)}} \ar[d]_-{\oc(\epsilon_A)}& \oc(A) \ar[d]^-{\epsilon_A} \\
  & A & & \oc(A) & \oc(A) \ar[r]_-{\epsilon_A} & A 
}\end{gathered}\end{equation}
An \textbf{abstract coKleisli category} \cite[Definition 2.4.1]{blute2015cartesian} is a category $\mathbb{X}$ equipped with an abstract coKleisli structure $(\oc, \varphi, \epsilon)$. 
\end{definition}

Below in Lemma \ref{cokleisliabstractlem}, we will review how every coKleisli category is an abstract coKleisli category. In order to obtain the converse, we first need from an abstract coKleisli category to construct a category with comonad. In an abstract coKleisli category, there are an important class of maps called the $\epsilon$-natural maps (which are the dual of thunkable maps in thunk-force categories \cite[Definition 7]{fuhrmann1999direct}).  These $\epsilon$-natural maps form a subcategory which comes equipped with a comonad, and the coKleisli category of this comonad is the starting abstract coKleisli category. 

\begin{definition} In an abstract coKleisli category $\mathbb{X}$ with abstract coKleisli structure $(\oc, \varphi, \epsilon)$, a map ${f: A \to B}$ is said to \textbf{$\epsilon$-natural} if the following diagram commutes: 
\begin{equation}\label{}\begin{gathered} \xymatrixcolsep{5pc}\xymatrix{ \oc(A) \ar[d]_-{\epsilon_A} \ar[r]^-{\oc(f)} & \oc(B) \ar[d]^-{\epsilon_B} \\
A \ar[r]_-{f} & B
  } \end{gathered}\end{equation}
Define the subcategory of $\epsilon$-natural maps $\epsilon\text{-}\mathsf{nat}[\mathbb{X}]$ to be the category whose objects are the same as $\mathbb{X}$ and whose maps are $\epsilon$-natural in $\mathbb{X}$, and let $\mathsf{U}_{\epsilon}:\epsilon\text{-}\mathsf{nat}[\mathbb{X}] \to \mathbb{X}$ be the obvious forgetful functor.   
\end{definition}

As we will discuss in Lemma \ref{lemexact}, in the context of a coKleisli category of a comonad, these $\epsilon$-natural maps should be thought of as the maps in the base category. Here are now some basic properties of $\epsilon$-natural maps. 

\begin{lemma}\label{eplem} \cite[Section 2.4]{blute2015cartesian} Let $\mathbb{X}$ be an abstract coKleisli category with abstract coKleisli structure $(\oc, \varphi, \epsilon)$. Then: 
\begin{enumerate}[{\em (i)}]
\item Identity maps $1_A: A \to A$ are $\epsilon$-natural;
\item \label{eplem.comp} If $f: A \to B$ and $g: B \to C$ are $\epsilon$-natural, then their composite $g \circ f: A \to C$ is $\epsilon$-natural;
\item \label{eplem.ep} For every object $A$, $\epsilon_A: \oc(A) \to A$ is $\epsilon$-natural; 
\item \label{eplem.oc} For every map $f: A \to B$, $\oc(f): \oc(A) \to \oc(B)$ is $\epsilon$-natural. 
\end{enumerate}
\end{lemma}

We now review in detail how every abstract coKleisli category is isomorphic to the coKleisli category of a canonical comonad on the subcategory of $\epsilon$-natural maps. 

\begin{lemma} \label{lem:ep-com} \cite[Dual of Theorem 4]{fuhrmann1999direct} Let $\mathbb{X}$ be an abstract coKleisli category with abstract coKleisli structure $(\oc, \varphi, \epsilon)$. Define the natural transformation $\beta_A: \oc(A) \to \oc\oc(A)$ as follows: 
  \[ \beta_A := \xymatrixcolsep{5pc}\xymatrix{\oc(A) \ar[r]^-{\oc(\varphi_A)} & \oc\oc(A)  
  } \]
Then $(\oc, \beta, \epsilon)$ is a comonad on $\epsilon\text{-}\mathsf{nat}[\mathbb{X}]$ such that the functor $\mathsf{G}_\epsilon: \mathbb{X} \to \epsilon\text{-}\mathsf{nat}[\mathbb{X}]_\oc$ defined on objects as $\mathsf{G}_\epsilon(A)=A$ and on a map ${f: A \to B}$ as the following composite: 
\begin{align*}
\llbracket \mathsf{G}_\epsilon(f) \rrbracket :=   \xymatrixcolsep{3pc}\xymatrix{\oc (A) \ar[r]^-{\oc(f)} & \oc(B) \ar[r]^-{\epsilon_B} & B } && \llbracket \mathsf{G}_\epsilon(f) \rrbracket = \epsilon_B \circ \oc(f)
\end{align*}
is an isomorphism with inverse $\mathsf{G}^{-1}_\epsilon: \epsilon\text{-}\mathsf{nat}[\mathbb{X}]_\oc \to \mathbb{X}$ defined on objects as $\mathsf{G}_\epsilon(A)=A$ and on a coKleisli map ${\llbracket f \rrbracket: \oc(A) \to B}$ as the following composite: 
\begin{align*}
\mathsf{G}^{-1}_\epsilon\left( \llbracket f \rrbracket \right) :=   \xymatrixcolsep{3pc}\xymatrix{A \ar[r]^-{\varphi_A} & \oc(A) \ar[r]^-{\llbracket f \rrbracket} & B } && \mathsf{G}^{-1}_\epsilon\left( \llbracket f \rrbracket \right) = \llbracket f \rrbracket \circ \varphi_A 
\end{align*}
\end{lemma}

We now wish to equip abstract coKleisli categories with Cartesian differential structure. To do so, we must first discuss Cartesian left additive structure for abstract coKleisli categories. We start with the finite product structure:

\begin{definition} A \textbf{Cartesian abstract coKleisli category} \cite[Definition 2.4.1]{blute2015cartesian} is an abstract coKleisli category $\mathbb{X}$ with abstract coKleisli structure $(\oc, \varphi, \epsilon)$ such that $\mathbb{X}$ has finite products and all the projection maps $\pi_0: A \times B \to A$ and $\pi_1: A \times B \to B$ are $\epsilon$-natural. 
\end{definition}

For a Cartesian abstract coKleisli category, it follows that $\epsilon$-natural maps are closed under the finite product structure. 

\begin{lemma} \cite[Section 2.4]{blute2015cartesian} Let $\mathbb{X}$ be a Cartesian abstract coKleisli category with abstract coKleisli structure $(\oc, \varphi, \epsilon)$. Then:
\begin{enumerate}[{\em (i)}]
\item If $f: C \to A$ and $g: C \to B$ are $\epsilon$-natural, then their pairing $\langle f, g \rangle: C \to A \times B$ is $\epsilon$-natural;
\item If $h: A \to C$ and $k: B \to D$ are $\epsilon$-natural, then their product $h \times k: A \times B \to C \times D$ is $\epsilon$-natural.
\end{enumerate}
Therefore, $\epsilon\text{-}\mathsf{nat}[\mathbb{X}]$ has finite products (which is defined as in $\mathbb{X}$). 
\end{lemma}

Next we discuss Cartesian left additive structure for abstract coKleisli categories, where we require that $\epsilon$-natural maps are closed under the additive structure.  
\begin{definition} A \textbf{Cartesian left additive abstract coKleisli category} is a Cartesian abstract coKleisli category $\mathbb{X}$ with abstract coKleisli structure $(\oc, \varphi, \epsilon)$ such that $\mathbb{X}$ is also a Cartesian left additive category, zero maps $0: A \to B$ are $\epsilon$-natural, and if ${f: A \to B}$ and $g: A \to B$ are $\epsilon$-natural, then their sum ${f+g : A \to B}$ is $\epsilon$-natural.
 \end{definition}
 
For a Cartesian left additive abstract coKleisli category, the subcategory of $\epsilon$-natural maps also form a Cartesian left additive category. It is important to stress however that $\epsilon$-natural maps are not assumed to be additive, and therefore the subcategory of $\epsilon$-natural maps does not necessarily have biproducts. 

\begin{lemma}\label{abstractCLAClem} Let $\mathbb{X}$ be a Cartesian left additive abstract coKleisli category with abstract coKleisli structure $(\oc, \varphi, \epsilon)$. Then $\epsilon\text{-}\mathsf{nat}[\mathbb{X}]$ is a Cartesian left additive category (where the necessary structure is defined as in $\mathbb{X}$). Furthermore,
\begin{enumerate}[{\em (i)}]
\item \label{ep.zero} $\epsilon_A \circ \oc(0) = 0$ 
\item \label{ep.sum} If $f: A \to B$ and $g: A \to B$ are $\epsilon$-natural, then $\varepsilon_B \circ \oc(f + g) = \epsilon_B \circ \oc(f) + \epsilon_B \circ \oc(g)$
\end{enumerate}
\end{lemma}
\begin{proof} It is clear that $\epsilon\text{-}\mathsf{nat}[\mathbb{X}]$ is a Cartesian left additive category. For (\ref{ep.zero}) we use the fact that $0$ is $\epsilon$-natural: 
\begin{align*}
    \epsilon_A \circ \oc(0) &=~ 0 \circ \epsilon_A \tag{$0$ is $\epsilon$-natual} \\
    &=~ 0
\end{align*}
For (\ref{ep.sum}), we use the fact that the sum of $\epsilon$-natural maps is $\epsilon$-natural: 
\begin{align*}
   \epsilon_B \circ \oc(f + g)&=~ (f+g) \circ \epsilon_A \tag{$f+g$ is $\epsilon$-natual} \\
    &=~ f \circ \epsilon_A + g \circ \epsilon_A \\
    &=~ \epsilon_B \circ \oc(f) + \epsilon_B \circ \oc(g) \tag{$f$ and $g$ are $\epsilon$-natural}
\end{align*}
\end{proof}

We are now in a position to define Cartesian differential abstract coKleisli categories. 

\begin{definition}\label{def:abCDC} A \textbf{Cartesian differential abstract coKleisli category} is a Cartesian differential category $\mathbb{X}$, with differential combinator $\mathsf{D}$, such that $\mathbb{X}$ is also a Cartesian left additive abstract coKleisli category with abstract coKleisli structure $(\oc, \varphi, \epsilon)$ and every $\epsilon$-natural map is $\mathsf{D}$-linear. 
\end{definition}

We will now show that for a Cartesian differential abstract coKleisli category, the canonical comonad on the subcategory of $\epsilon$-natural maps is a Cartesian differential comonad and that the coKleisli category is isomorphic to the starting Cartesian differential abstract coKleisli category.

\begin{proposition}\label{propab1} Let $\mathbb{X}$ be a Cartesian differential abstract coKleisli category with differential combinator $\mathsf{D}$ and abstract coKleisli structure $(\oc, \varphi, \epsilon)$. Then $\epsilon\text{-}\mathsf{nat}[\mathbb{X}]$ is a category with finite biproducts and $(\oc, \beta, \epsilon, \partial)$ (where $(\oc, \beta, \epsilon)$ is defined as in Lemma \ref{lem:ep-com}) is a Cartesian differential comonad on $\epsilon\text{-}\mathsf{nat}[\mathbb{X}]$ where the differential combinator transformation $\partial_A: \oc(A) \to \oc(A \times A)$ is defined as follows: 
\begin{equation}\label{partialdef2}\begin{gathered}\partial_A := \xymatrixcolsep{5pc}\xymatrix{ \oc(A \times A) \ar[r]^-{\oc\left( \mathsf{D}[\varphi_A] \right)} & \oc\oc(A) \ar[r]^-{\epsilon_{\oc(A)}} & \oc(A) 
  } \end{gathered}\end{equation}
Furthermore, $\mathsf{G}_\epsilon: \mathbb{X} \to \epsilon\text{-}\mathsf{nat}[\mathbb{X}]_\oc$ is a Cartesian differential isomorphism, so in particular, the following equalities hold: 
\begin{equation}\label{}\begin{gathered} 
\llbracket \mathsf{G}_\epsilon(\mathsf{D}[f] ) \rrbracket = \llbracket \mathsf{D}\left[ \mathsf{G}_\epsilon(f) \right] \rrbracket \quad \quad \quad \mathsf{G}^{-1}_\epsilon\left( \llbracket \mathsf{D}[f] \rrbracket \right) = \mathsf{D}\left[ \mathsf{G}^{-1}_\epsilon\left( \llbracket f \rrbracket \right) \right] 
 \end{gathered}\end{equation}
 where the differential combinator on the coKleisli category $\epsilon\text{-}\mathsf{nat}[\mathbb{X}]_\oc$ is defined as in Theorem \ref{thm1}. 
\end{proposition} 
\begin{proof} We must first explain why $\epsilon\text{-}\mathsf{nat}[\mathbb{X}]$ has finite biproducts. By Lemma \ref{abstractCLAClem}, $\epsilon\text{-}\mathsf{nat}[\mathbb{X}]$ is a Cartesian left additive category. However, by assumption, every $\epsilon$-natural map is $\mathsf{D}$-linear, so by Lemma \ref{linlem}.(\ref{linlem.add}), every $\epsilon$-natural map is also additive. Therefore, $\epsilon\text{-}\mathsf{nat}[\mathbb{X}]$ is a Cartesian left additive category such that every map is additive, and so we conclude that $\epsilon\text{-}\mathsf{nat}[\mathbb{X}]$ has finite biproducts. 

Next we must explain why the proposed differential combinator transformation is well defined, that is, we must show that $\partial_A$ is indeed $\epsilon$-natural. However, by Lemma \ref{eplem}.(\ref{eplem.ep}) and (\ref{eplem.oc}), $\epsilon_{\oc A}$ and $\oc\left( \mathsf{D}[\varphi_A] \right)$ are both $\epsilon$-natural. Then by Lemma \ref{eplem}.(\ref{eplem.comp}), their composite $\partial_A = \epsilon_{\oc A} \circ \oc\left( \mathsf{D}[\varphi_A] \right) $ is $\epsilon$-natural. Next we show the naturality of $\partial$. So for an $\epsilon$-natural map $f: A\to B$, we compute: 
\begin{align*}
\partial_B \circ \oc(f \times f) &=~ \epsilon_{\oc(B)} \circ \oc\left( \mathsf{D}[\varphi_B] \right) \circ \oc(f \times f) \\
&=~ \epsilon_{\oc(B)} \circ \oc\left(  \mathsf{D}[\varphi_B] \circ (f \times f) \right) \tag{$\oc$ is a functor} \\
&=~\epsilon_{\oc(B)} \circ  \oc\left(  \mathsf{D}[\varphi_B \circ f] \right) \tag{$f$ is $\mathsf{D}$-linear and Lem \ref{linlem}.(\ref{linlem.pre})} \\
&=~\epsilon_{\oc(B)} \circ  \oc\left(  \mathsf{D}[\oc(f) \circ \varphi_A] \right) \tag{Naturality of $f$} \\
&=~\epsilon_{\oc(B)} \circ  \oc\left(  \oc(f) \circ \mathsf{D}[\varphi_A] \right) \tag{$\oc(f)$ is $\mathsf{D}$-linear and Lem \ref{linlem}.(\ref{linlem.post})} \\
&=~\epsilon_{\oc(B)} \circ \oc\oc(f) \circ \oc\left( \mathsf{D}[\varphi_A] \right) \tag{$\oc$ is a functor} \\
&=~ \oc(f) \circ \epsilon_{\oc(A)} \circ \oc\left( \mathsf{D}[\varphi_A] \right) \tag{Naturality of $\epsilon_{\oc(-)}$} \\
&=~\oc(f) \circ \partial_A 
\end{align*}
So $\partial$ is a natural transformation. Now we must show that $\partial$ satisfies the six axioms of a differential combinator transformation. Note that the calculations below are similar to the ones in the proof of Proposition \ref{prop1}.
\begin{enumerate}[{\bf [dc.1]}] 
\item Here we use \textbf{[CD.2]}: 
\begin{align*}
\partial_A \circ \oc(\iota_0) &=~\epsilon_{\oc(A)} \circ \oc\left( \mathsf{D}[\varphi_A] \right) \circ \oc(\iota_0) \\
&=~ \epsilon_{\oc(A)} \circ \oc\left( \mathsf{D}[\varphi_A] \circ \iota_0 \right)   \tag{$\oc$ is a functor} \\
&=~ \epsilon_{\oc(A)} \circ \oc(0) \tag{\textbf{[CD.2]}} \\
&=~ 0 \tag{Lem.\ref{abstractCLAClem}.(\ref{ep.zero})}
\end{align*}
\item Here we use \textbf{[CD.2]}: 
\begin{align*}
\partial_A \circ \oc(1_A \times \nabla_A) &=~\epsilon_{\oc(A)} \circ \oc\left( \mathsf{D}[\varphi_A] \right) \circ \oc(1_A \times \nabla_A) \\
&=~\epsilon_{\oc(A)} \circ \oc\left( \mathsf{D}[\varphi_A] \circ (1_A \times \nabla_A) \right)   \tag{$\oc$ is a functor} \\ 
&=~\epsilon_{\oc(A)} \circ \oc\left( \mathsf{D}[\varphi_A] \circ (1_A \times\pi_0)  +  \mathsf{D}[\varphi_A] \circ (1_A \times\pi_1)  \right) \tag{\textbf{[CD.2]}} \\
&=~\epsilon_{\oc(A)} \circ \epsilon_{\oc(A)} \circ \oc\left( \mathsf{D}[\varphi_A] \circ (1_A \times\pi_0) \right) + \epsilon_{\oc(A)} \circ \oc\left( \mathsf{D}[\varphi_A] \circ (1_A \times\pi_1)  \right)  \tag{Lem.\ref{abstractCLAClem}.(\ref{ep.sum})} \\
&=~\epsilon_{\oc(A)} \circ \epsilon_{\oc(A)} \circ \oc\left( \mathsf{D}[\varphi_A]\right) \circ \oc\left (1_A \times\pi_0 \right) + \epsilon_{\oc(A)} \circ \oc\left( \mathsf{D}[\varphi_A] \right) \circ \oc\left(1_A \times\pi_1  \right) \tag{$\oc$ is a functor} \\ 
&=~\partial_A \circ \oc(1_A \times \pi_0) + \partial_A \circ \oc(1_A \times \pi_1) \\
&=~ \partial_A \circ \left( \oc(1_A \times \pi_0) + \oc(1_A \times \pi_0) \right) 
\end{align*}
\item Here we use \textbf{[CD.3]} and the fact that since $\epsilon_A$ is $\epsilon$-natural, that it is also $\mathsf{D}$-linear: 
\begin{align*}
\epsilon_A \circ \partial_A &=~\epsilon_A \circ \epsilon_{\oc(A)} \circ \oc\left( \mathsf{D}[\varphi_A] \right) \\
&=~ \epsilon_A \circ \oc(\epsilon_A) \circ \oc\left( \mathsf{D}[\varphi_A] \right) \tag{Abstract coKleisli structure identity} \\
&=~ \epsilon_A \circ \oc\left(\epsilon_A \circ \mathsf{D}[\varphi_A] \right)  \tag{$\oc$ is a functor} \\
&=~\epsilon_A \circ \oc\left( \mathsf{D}\left[ \epsilon_A \circ \varphi_A \right] \right) \tag{$\epsilon$ is $\mathsf{D}$-linear and Lem \ref{linlem}.(\ref{linlem.post})} \\
&=~\epsilon_A \circ \oc\left( \mathsf{D}\left[ 1_A \right] \right) \tag{Abstract coKleisli structure identity} \\
&=~\epsilon_A \circ \oc\left( \pi_1 \right) \tag{\textbf{[CD.3]}} \\
&=~\pi_1 \circ \epsilon_{A \times A} \tag{$\pi_1$ is $\epsilon$-natural}
\end{align*}
\item Here we use \textbf{[CD.5]}: 
\begin{align*}
\beta_A \circ \partial_A &=~\oc(\varphi_A) \circ \partial_A \\
&=~\oc(\varphi_A) \circ \epsilon_{\oc(A)} \circ \oc\left( \mathsf{D}[\varphi_A] \right) \\
&=~ \epsilon_{\oc\oc(A)} \circ \oc\oc(\varphi_A) \circ \oc\left( \mathsf{D}[\varphi_A] \right)  \tag{Naturality of $\epsilon_{\oc(-)}$} \\
&=~\epsilon_{\oc\oc(A)} \circ \oc\left( \oc(\varphi_A) \circ \mathsf{D}[\varphi_A] \right)   \tag{$\oc$ is a functor} \\
&=~\epsilon_{\oc\oc(A)} \circ \oc\left( \mathsf{D}\left[ \oc(\varphi_A) \circ \varphi_A \right] \right) \tag{$\oc(\varphi_A)$ is $\mathsf{D}$-linear and Lem \ref{linlem}.(\ref{linlem.post})} \\
&=~\epsilon_{\oc\oc(A)} \circ \oc\left( \mathsf{D}\left[  \varphi_{\oc(A)} \circ \varphi_A \right] \right) \tag{Naturality of $\varphi$} \\ 
&=~\epsilon_{\oc\oc(A)} \circ \oc\left( \mathsf{D}\left[  \varphi_{\oc(A)} \right] \circ  \left\langle \varphi_A \circ \pi_0, \mathsf{D}[\varphi_A] \right \rangle  \right) \tag{\textbf{[CD.5]}} \\ 
&=~\epsilon_{\oc\oc(A)} \circ \oc\left( \mathsf{D}\left[  \varphi_{\oc(A)} \right] \right) \circ  \oc\left( \left\langle \varphi_A \circ \pi_0, \mathsf{D}[\varphi_A] \right \rangle  \right)\tag{$\oc$ is a functor} \\
&=~\partial_{\oc(A)} \circ  \oc\left( \left\langle \varphi_A \circ \pi_0, \mathsf{D}[\varphi_A] \right \rangle  \right) \\
&=~\partial_{\oc(A)} \circ \oc\left( \left\langle \oc(\pi_0) \circ \varphi_{A \times A}, \mathsf{D}[\varphi_A] \right \rangle  \right)  \tag{Naturality of $\varphi$} \\ 
&=~ \partial_{\oc(A)} \circ \oc\left( \left\langle \oc(\pi_0) \circ \varphi_{A \times A}, 1_{\oc(A)} \mathsf{D}[\varphi_A] \right \rangle  \right) \\
&=~ \partial_{\oc(A)} \circ \oc\left( \left\langle \oc(\pi_0) \circ \varphi_{A \times A}, \epsilon_{\oc(A)} \circ \varphi_{\oc(A)} \mathsf{D}[\varphi_A] \right \rangle  \right) \tag{Abstract coKleisli structure identity} \\
&=~ \partial_{\oc(A)} \circ \oc\left( \left\langle \oc(\pi_0) \circ \varphi_{A \times A}, \epsilon_{\oc(A)} \circ \oc\left( \mathsf{D}[\varphi_A] \right) \circ \varphi_{A \times A} \right \rangle  \right)\tag{Naturality of $\varphi$} \\
&=~\partial_{\oc(A)} \circ \oc\left( \left\langle \oc(\pi_0), \epsilon_{\oc(A)} \circ \oc\left( \mathsf{D}[\varphi_A] \right) \right \rangle  \circ \varphi_{A \times A}  \right) \\
&=~\partial_{\oc(A)} \circ \oc\left( \left\langle \oc(\pi_0), \partial_A \right \rangle  \circ \varphi_{A \times A}  \right) \\
&=~\partial_{\oc(A)} \circ \oc\left( \left\langle \oc(\pi_0), \partial_A \right \rangle \right)  \circ \oc\left( \varphi_{A \times A}  \right) \tag{$\oc$ is a functor} \\
&=~\partial_{\oc(A)} \circ \oc\left( \left \langle \oc(\pi_0) , \partial_A \right \rangle \right) \circ \beta_{A \times A}
\end{align*}
\end{enumerate}
For the remaining two axioms, it'll be useful to compute $\partial_{A} \circ \partial_{A \times A}$: 
\begin{align*}
\partial_{A} \circ \partial_{A \times A}  &=~\epsilon_{\oc(A)} \circ \oc\left( \mathsf{D}[\varphi_A] \right) \circ \partial_{A \times A} \\
&=~\epsilon_{\oc(A)} \circ \oc\left( \mathsf{D}[\varphi_A] \right) \circ \epsilon_{\oc(A \times A)} \circ \oc\left( \mathsf{D}[\varphi_{A \times A}] \right) \\
&=~\epsilon_{\oc(A)} \circ  \epsilon_{\oc\oc(A)} \circ \oc\oc\left( \mathsf{D}[\varphi_A] \right) \circ \oc\left( \mathsf{D}[\varphi_{A \times A}] \right)   \tag{Naturality of $\epsilon_{\oc(-)}$} \\
&=~\epsilon_{\oc(A)} \circ  \epsilon_{\oc\oc(A)} \circ \oc\left( \oc\left( \mathsf{D}[\varphi_A] \right) \circ \mathsf{D}[\varphi_{A \times A}] \right) \tag{$\oc$ is a functor} \\
&=~\epsilon_{\oc(A)} \circ  \epsilon_{\oc\oc(A)} \circ \oc\left(  \mathsf{D}\left[ \oc\left( \mathsf{D}[\varphi_A] \right) \circ \varphi_{A \times A} \right] \right)  \tag{$\oc\left( \mathsf{D}[\varphi_A] \right)$ is $\mathsf{D}$-linear and Lem \ref{linlem}.(\ref{linlem.post})} \\
&=~\epsilon_{\oc(A)} \circ  \epsilon_{\oc\oc(A)} \circ \oc\left(  \mathsf{D}\left[ \varphi_{\oc(A)} \circ \mathsf{D}[\varphi_A] \right] \right) \tag{Naturality of $\varphi$} \\
&=~\epsilon_{\oc(A)} \circ  \epsilon_{\oc\oc(A)} \circ \oc\left(  \mathsf{D}\left[ \varphi_{\oc(A)} \right] \circ \left \langle \mathsf{D}[\varphi_A] \circ\pi_0, \mathsf{D}\left[ \mathsf{D}[\varphi_A] \right] \right \rangle  \right) \tag{\textbf{[CD.5]}} \\
&=~\epsilon_{\oc(A)} \circ  \epsilon_{\oc\oc(A)} \circ \oc\left(  \mathsf{D}\left[ \varphi_{\oc(A)} \right] \right) \circ \oc\left(\left \langle \mathsf{D}[\varphi_A] \circ\pi_0, \mathsf{D}\left[ \mathsf{D}[\varphi_A] \right] \right \rangle  \right) \tag{$\oc$ is a functor} \\
&=~\epsilon_{\oc(A)} \circ \oc(\epsilon_{\oc(A)}) \circ \oc\left(  \mathsf{D}\left[ \varphi_{\oc(A)} \right] \right) \circ \oc\left(\left \langle \mathsf{D}[\varphi_A] \circ\pi_0, \mathsf{D}\left[ \mathsf{D}[\varphi_A] \right] \right \rangle  \right)  \tag{Abstract coKleisli structure identity} \\
&=~\epsilon_{\oc(A)} \circ \oc\left(\epsilon_{\oc(A)} \circ  \mathsf{D}\left[ \varphi_{\oc(A)} \right] \right) \circ \oc\left(\left \langle \mathsf{D}[\varphi_A] \circ\pi_0, \mathsf{D}\left[ \mathsf{D}[\varphi_A] \right] \right \rangle  \right) \tag{$\oc$ is a functor} \\
&=~\epsilon_{\oc(A)} \circ \oc\left(  \mathsf{D}\left[ \epsilon_{\oc(A)} \circ \varphi_{\oc(A)} \right] \right) \circ \oc\left(\left \langle \mathsf{D}[\varphi_A] \circ\pi_0, \mathsf{D}\left[ \mathsf{D}[\varphi_A] \right] \right \rangle  \right) \tag{$\epsilon$ is $\mathsf{D}$-linear and Lem \ref{linlem}.(\ref{linlem.post})} \\
&=~\epsilon_{\oc(A)} \circ \oc\left(  \mathsf{D}\left[ 1_{\oc(A)} \right] \right) \circ \oc\left(\left \langle \mathsf{D}[\varphi_A] \circ\pi_0, \mathsf{D}\left[ \mathsf{D}[\varphi_A] \right] \right \rangle  \right) \tag{Abstract coKleisli structure identity} \\
&=~\epsilon_{\oc(A)} \circ \oc\left( \pi_1 \right) \circ \oc\left(\left \langle \mathsf{D}[\varphi_A] \circ\pi_0, \mathsf{D}\left[ \mathsf{D}[\varphi_A] \right] \right \rangle  \right) \tag{\textbf{[CD.3]}} \\
&=~\epsilon_{\oc(A)} \circ \oc\left( \pi_1 \circ \left \langle \mathsf{D}[\varphi_A] \circ\pi_0, \mathsf{D}\left[ \mathsf{D}[\varphi_A] \right] \right \rangle  \right) \tag{$\oc$ is a functor} \\
&=~\epsilon_{\oc(A)} \circ \oc\left( \mathsf{D}\left[ \mathsf{D}[\varphi_A] \right]  \right)
\end{align*}
So we have that: 
\begin{equation}\label{partial2}\begin{gathered} 
\partial_{A} \circ \partial_{A \times A} = \epsilon_{\oc(A)} \circ \oc\left( \mathsf{D}\left[ \mathsf{D}[\varphi_A] \right]  \right)
 \end{gathered}\end{equation}
 So now we can easily show \textbf{[dc.5]} and \textbf{[dc.6]}, 
\begin{enumerate}[{\bf [dc.1]}] 
\setcounter{enumi}{4}
\item Here we use \textbf{[CD.6]}, \textbf{[CD.2]}, and the above identity: 
\begin{align*}
\partial_{A} \circ \partial_{A \times A} \circ \oc(\ell_A) &=~ \epsilon_{\oc(A)} \circ \oc\left( \mathsf{D}\left[ \mathsf{D}[\varphi_A] \right]  \right) \circ \oc(\ell_A) \tag{partial2} \\
&=~ \epsilon_{\oc(A)} \circ \oc\left( \mathsf{D}\left[ \mathsf{D}[\varphi_A] \right]  \circ \ell_A  \right)  \tag{$\oc$ is a functor} \\
&=~ \epsilon_{\oc(A)} \circ \oc\left( \mathsf{D}[\varphi_A]   \right) \tag{\textbf{[CD.6]}} \\
&=~ \partial_A 
\end{align*}
\item Here we use \textbf{[CD.7]}:
\begin{align*}
\partial_{A} \circ \partial_{A \times A} \circ \oc(c_A) &=~ \epsilon_{\oc(A)} \circ \oc\left( \mathsf{D}\left[ \mathsf{D}[\varphi_A] \right]  \right) \circ \oc(c_A) \tag{partial2} \\
&=~ \epsilon_{\oc(A)} \circ \oc\left( \mathsf{D}\left[ \mathsf{D}[\varphi_A] \right]  \circ c_A  \right)  \tag{$\oc$ is a functor} \\
&=~ \epsilon_{\oc(A)} \circ \oc\left( \mathsf{D}\left[ \mathsf{D}[\varphi_A] \right]  \right)  \tag{\textbf{[CD.7]}} \\
&=~\partial_{A} \circ \partial_{A \times A}  \tag{partial2}
\end{align*}
\end{enumerate}
So we conclude that $\partial$ is a differential combinator transformation, and therefore $(\oc, \beta, \epsilon, \partial)$ is a Cartesian differential comonad. It remains to show that $\mathsf{G}_\epsilon$ and $\mathsf{G}^{-1}_\epsilon$ commute with the differential combinator. First, for any map ${f: A \to B}$ in $\mathbb{X}$, we compute: 
\begin{align*}
 \llbracket \mathsf{D}\left[ \mathsf{G}_\epsilon(f) \right] \rrbracket &=~\llbracket \mathsf{G}_\epsilon(f)  \rrbracket \circ \partial_A \\
 &=~\epsilon_{B} \circ \oc(f) \circ \partial_A \\
 &=~\epsilon_{B} \circ \oc(f) \circ \epsilon_{\oc(A)} \circ \oc\left( \mathsf{D}[\varphi_A]   \right) \\
  &=~\epsilon_{B} \circ \epsilon_{\oc(B)} \circ \oc\oc(f) \circ \oc\left(\mathsf{D}[\varphi_A]  \right) \tag{Naturality of $\epsilon_{\oc(-)}$} \\
 &=~\epsilon_{B} \circ \epsilon_{\oc(B)} \circ \oc\left( \oc(f) \circ\mathsf{D}[\varphi_A]  \right)  \tag{$\oc$ is a functor} \\
 &=~\epsilon_{B} \circ \epsilon_{\oc(B)} \circ \oc\left( \mathsf{D}\left[  \oc(f) \circ \varphi_A \right]  \right)  \tag{$\oc(f)$ is $\mathsf{D}$-linear and Lem \ref{linlem}.(\ref{linlem.post})} \\ 
 &=~\epsilon_{B} \circ \epsilon_{\oc(B)} \circ \oc\left( \mathsf{D}\left[  \varphi_B \circ f \right]  \right) \tag{Naturality of $\varphi$} \\
 &=~\epsilon_{B} \circ \epsilon_{\oc(B)} \circ \oc\left( \mathsf{D}\left[  \varphi_B \right] \circ \langle f \circ \pi_0, \mathsf{D}[f] \rangle  \right) \tag{\textbf{[CD.5]}} \\
  &=~\epsilon_{B} \circ \oc(\epsilon_B) \circ \oc\left( \mathsf{D}\left[  \varphi_B \right] \circ \langle f \circ \pi_0, \mathsf{D}[f] \rangle  \right)  \tag{Abstract coKleisli structure identity} \\
 &=~\epsilon_{B} \circ \oc\left( \epsilon_B \circ \mathsf{D}\left[  \varphi_B \right] \circ \langle f \circ \pi_0, \mathsf{D}[f] \rangle \right)   \tag{$\oc$ is a functor} \\
 &=~\epsilon_{B}  \circ \oc\left(  \mathsf{D}\left[ \epsilon_B \circ  \varphi_B \right] \circ \langle f \circ \pi_0, \mathsf{D}[f] \rangle  \right)  \tag{$\epsilon$ is $\mathsf{D}$-linear and Lem \ref{linlem}.(\ref{linlem.post})} \\ 
 &=~\epsilon_{B}  \circ \oc\left(  \mathsf{D}\left[ 1_B \right] \circ \langle f \circ \pi_0, \mathsf{D}[f] \rangle  \right) \tag{Abstract coKleisli structure identity} \\
 &=~\epsilon_{B}  \circ \oc\left(  \pi_1 \circ \langle f \circ \pi_0, \mathsf{D}[f] \rangle \right)\tag{\textbf{[CD.3]}} \\
 &=~\epsilon_B \circ \oc\left( \mathsf{D}[f] \right) \\
 &=~\llbracket \mathsf{G}_\epsilon(\mathsf{D}[f] ) \rrbracket
\end{align*}
Next for any $\epsilon$-natural coKleisli map $\llbracket f \rrbracket: \oc(A) \to B$, note that $\llbracket f \rrbracket$ is then $\mathsf{D}$-linear in $\mathbb{X}$, so we compute: 
\begin{align*}
 \mathsf{D}\left[ \mathsf{G}^{-1}_\epsilon\left( \llbracket f \rrbracket \right) \right] &=~ \mathsf{D}\left[ \llbracket f \rrbracket  \circ \varphi_A \right] \\
 &=~ \llbracket f \rrbracket \circ \mathsf{D}[\varphi_A] \tag{$\llbracket f \rrbracket$ is $\mathsf{D}$-linear and Lem \ref{linlem}.(\ref{linlem.post})} \\
 &=~\llbracket f \rrbracket \circ 1_{\oc(A)} \circ  \mathsf{D}[\varphi_A] \\
 &=~\llbracket f \rrbracket \circ \epsilon_{\oc(A)} \circ \varphi_{\oc(A)} \circ  \mathsf{D}[\varphi_A]  \tag{Abstract coKleisli structure identity} \\
 &=~\llbracket f \rrbracket \circ  \epsilon_{\oc(A)} \circ \oc\left( \mathsf{D}[\varphi_A]   \right) \circ \varphi_{A \times A} \tag{Naturality of $\varphi$} \\
 &=~\llbracket f \rrbracket \circ \partial_A \circ \varphi_{A \times A} \\
 &=~ \llbracket \mathsf{D}[f] \rrbracket \circ \varphi_{A \times A} \\
 &=~\mathsf{G}^{-1}_\epsilon\left( \llbracket \mathsf{D}[f] \rrbracket \right) 
\end{align*}
So we conclude that $\mathbb{X}$ and $\epsilon\text{-}\mathsf{nat}[\mathbb{X}]_\oc$ are isomorphic as Cartesian differential categories. 
\end{proof} 

It is important to note that while $\epsilon$-natural maps are assumed to be $\mathsf{D}$-linear, the converse is not necessarily true. It turns out that all $\mathsf{D}$-linear maps are $\epsilon$-natural precisely when the Cartesian differential comonad has a $\mathsf{D}$-linear unit. 

\begin{lemma} Let $\mathbb{X}$ be a Cartesian differential abstract coKleisli category with differential combinator $\mathsf{D}$ and abstract coKleisli structure $(\oc, \varphi, \epsilon)$. Define the natural transformation $\eta_A: A \to \oc(A)$ as follows: 
\begin{equation}\label{etadef2}\begin{gathered}\eta_A := \xymatrixcolsep{5pc}\xymatrix{ A \ar[r]^-{\mathsf{L}[\varphi_A]} & \oc(A) 
  } \end{gathered}\end{equation}
where $\mathsf{L}$ is defined as in Lemma \ref{linlem}.(\ref{linlemimportant1}). Then the following are equivalent:   
\begin{enumerate}[{\em (i)}]
\item $\epsilon\text{-}\mathsf{nat}[\mathbb{X}] = \mathsf{D}\text{-}\mathsf{lin}[\mathbb{X}]$, that is, every $\mathsf{D}$-linear map is $\epsilon$-natural; 
\item For every object $A$, $\eta_A$ is $\epsilon$-natural; 
\item $\eta$ is a $\mathsf{D}$-linear unit for $(\oc, \beta, \epsilon, \partial)$. 
\end{enumerate}
\end{lemma}
\begin{proof} We first show that $\eta$ is indeed a natural transformation. So for a $\epsilon$-natural map $f: A \to B$, we compute: 
\begin{align*}
    \oc(f) \circ \eta_A &=~ \oc(f) \circ \mathsf{L}[\varphi_A] \\
    &=~ \oc(f) \circ \mathsf{D}[\varphi_A] \circ \iota_1 \\
    &=~ \mathsf{D}[\oc(f) \circ \varphi_A] \circ \iota_1 \tag{$\oc(f)$ is $\mathsf{D}$-linear and Lem \ref{linlem}.(\ref{linlem.post})} \\
    &=~ \mathsf{D}[\varphi_B \circ f] \circ \iota_1 \tag{Naturality of $\varphi$} \\
    &=~ \mathsf{D}[\varphi_B] \circ (f \times f) \circ \iota_1 \tag{$f$ is $\mathsf{D}$-linear and Lem \ref{linlem}.(\ref{linlem.pre})} \\
    &=~  \mathsf{D}[\varphi_B]\circ \iota_1 \circ f \\
    &=~ \eta_B \circ f
\end{align*}
So $\eta$ is indeed a natural transformation. 

For $(i) \Rightarrow (ii)$, first note that by definition, every $\epsilon$-natural map is $\mathsf{D}$-linear. So also assuming the converse that every $\mathsf{D}$-linear map is $\epsilon$-natural does indeed imply that $\epsilon\text{-}\mathsf{nat}[\mathbb{X}] = \mathsf{D}\text{-}\mathsf{lin}[\mathbb{X}]$. Now suppose that this is the case. By Lemma \ref{linlem}.(\ref{linlemimportant1}), $\mathsf{L}[\varphi_A]$ is $\mathsf{D}$-linear, and so by assumption, $\mathsf{L}[\varphi_A]$ is also $\epsilon$-natural. Next, for $(ii) \Rightarrow (iii)$, we must show that $\eta$ satisfies both $\mathsf{D}$-linear unit axioms:
\begin{enumerate}[{\bf [du.1]}] 
\item Here we use \textbf{[CD.3]}: 
\begin{align*}
    \epsilon_A \circ \eta_A &=~ \epsilon_A \circ \mathsf{L}[\varphi_A] \\
    &=~ \epsilon_A \circ \mathsf{D}[\varphi_A] \circ \iota_1 \\
    &=~  \mathsf{D}[\epsilon_A \circ \varphi_A] \circ \iota_1\tag{$\varepsilon$ is $\mathsf{D}$-linear and Lem \ref{linlem}.(\ref{linlem.post})} \\
    &=~ \mathsf{D}[1_A] \circ \iota_1 \\
    &=~ \pi_1 \circ \iota_1 \\
    &=~ 1_A
\end{align*}
\item Here we use that $\eta_A$ is assumed to be $\epsilon$-natural: 
\begin{align*}
     \eta_A \circ \epsilon_A &=~ \epsilon_{\oc(A)} \circ \oc(\eta_A) \tag{$\eta_A$ is $\epsilon$-natural} \\
     &=~ \epsilon_{\oc(A)} \circ \oc\left( \mathsf{L}\left[ \varphi_A \right] \right) \\
     &=~ \epsilon_{\oc(A)} \circ \oc\left( \mathsf{D}\left[ \varphi_A \right] \circ \iota_1 \right) \\
     &=~ \epsilon_{\oc(A)} \circ \oc\left( \mathsf{D}\left[ \varphi_A \right] \right) \circ \oc(\iota_1) \tag{$\oc$ is a functor} \\
     &=~ \partial_A \circ \oc(\iota_1) 
\end{align*}
\end{enumerate}
So we conclude that $\eta$ is a $\mathsf{D}$-linear unit. Lastly, for $(iii) \Rightarrow (i)$, by Proposition \ref{etaFlem1}, we have that $\epsilon\text{-}\mathsf{nat}[\mathbb{X}] \cong \mathsf{D}\text{-}\mathsf{lin}\left[\epsilon\text{-}\mathsf{nat}[\mathbb{X}]_\oc \right]$. By Proposition \ref{propab1}, $\mathbb{X}$ and $\epsilon\text{-}\mathsf{nat}[\mathbb{X}]_\oc$ are isomorphic as Cartesian differential categories, which implies that we also have that $\mathsf{D}\text{-}\mathsf{lin}[\mathbb{X}] \cong \mathsf{D}\text{-}\mathsf{lin}\left[\epsilon\text{-}\mathsf{nat}[\mathbb{X}]_\oc \right]$. Therefore, $\epsilon\text{-}\mathsf{nat}[\mathbb{X}] \cong \mathsf{D}\text{-}\mathsf{lin}[\mathbb{X}]$. However, it is straightforward to work out that this isomorphism is in fact an equality and so $\epsilon\text{-}\mathsf{nat}[\mathbb{X}] = \mathsf{D}\text{-}\mathsf{lin}[\mathbb{X}]$. 
\end{proof}

We turn our attention to the converse of Proposition \ref{propab1}. We will now explain how every coKleisli category of a Cartesian differential comonad is a Cartesian differential abstract coKleisli category. To do so, let us first quickly review how every coKleisli category is an abstract coKleisli category. 

\begin{lemma}\label{cokleisliabstractlem} \cite[Proposition 2.6.3]{blute2015cartesian} Let $(\oc, \delta, \varepsilon)$ be a comonad on a category $\mathbb{X}$. Then define: 
\begin{enumerate}[{\em (i)}]
\item The endofunctor $\oc_\oc: \mathbb{X}_\oc \to \mathbb{X}_\oc$ on objects as $\oc_\oc(A) = \oc(A)$ and on a coKleisli map $\llbracket f \rrbracket: \oc(A) \to B$ as the following composite: 
\begin{align*}
\llbracket \oc_\oc(f) \rrbracket :=   \xymatrixcolsep{3pc}\xymatrix{\oc\oc (A) \ar[r]^-{\varepsilon_{\oc(A)}} & \oc(A) \ar[r]^-{\delta_A} & \oc\oc(A) \ar[r]^-{\oc\left( \llbracket f \rrbracket \right)} & \oc(B) } && \llbracket \oc_\oc(f) \rrbracket = \oc\left( \llbracket f \rrbracket \right) \circ \delta_A \circ \varepsilon_{\oc(A)}
\end{align*}
\item The family of coKleisli maps $\llbracket \epsilon_A \rrbracket: \oc\oc(A) \to A$ as the following composite: 
\begin{align*}
\llbracket \epsilon_A \rrbracket :=   \xymatrixcolsep{3pc}\xymatrix{\oc\oc (A) \ar[r]^-{\varepsilon_{\oc(A)}} & \oc(A) \ar[r]^-{\varepsilon_A} & A } && \llbracket \epsilon_A \rrbracket = \varepsilon_A \circ \varepsilon_{\oc(A)}
\end{align*}
\end{enumerate}
Then the coKleisli category $\mathbb{X}_\oc$ is an abstract coKleisli category with abstract coKleisli structure $(\oc_\oc, \varphi, \epsilon)$, where $\varphi$ is defined as in (\ref{varphidef}). Furthermore, 
\begin{enumerate}[{\em (i)}]
\item \label{lemepnatcok} A coKleisli map $\llbracket f \rrbracket: \oc(A) \to B$ is $\epsilon$-natural if and only if $\llbracket f \rrbracket \circ \varepsilon_{\oc(A)} = \llbracket f \rrbracket \circ \oc(\varepsilon_A)$;
\item \label{lemepvarep} For every map $f: A \to B$ in $\mathbb{X}$, $\llbracket \mathsf{F}_\oc(f) \rrbracket: \oc(A) \to B$ is $\epsilon$-natural;
\item There is a functor $\mathsf{F}_\epsilon: \mathbb{X} \to \epsilon\text{-}\mathsf{nat}[\mathbb{X}_\oc]$ which is defined on objects as $\mathsf{F}_\epsilon(A) = A$ and on maps ${f: A \to B}$ as $\llbracket \mathsf{F}_\epsilon(f) \rrbracket = f \circ \varepsilon_A = \llbracket \mathsf{F}_{\oc}(f) \rrbracket$, and such that the following diagram commutes: 
  \[  \xymatrixcolsep{5pc}\xymatrix{ \mathbb{X} \ar[dr]_-{\mathsf{F}_\epsilon} \ar[rr]^-{\mathsf{F}_\oc} && \mathbb{X}_\oc  \\
  & \epsilon\text{-}\mathsf{nat}[\mathbb{X}_\oc] \ar[ur]_-{\mathsf{U}}
  } \]
\end{enumerate}
\end{lemma}

A natural question to ask is when the subcategory of $\epsilon$-natural maps of a coKleisli category is isomorphic to the base category. The answer is when the comonad is exact (for monads, this is called the equalizer requirement \cite[Definition 8]{fuhrmann1999direct}). 

\begin{lemma}\label{lemexact}  \cite[Dual of Theorem 9]{fuhrmann1999direct} Let $(\oc, \delta, \varepsilon)$ be a comonad on a category $\mathbb{X}$. Then the following are equivalent: 
\begin{enumerate}[{\em (i)}]
\item $\mathsf{F}_\epsilon: \mathbb{X} \to \epsilon\text{-}\mathsf{nat}[\mathbb{X}_\oc]$ is an isomorphism; 
\item The comonad $(\oc, \delta, \varepsilon)$ is exact \cite[Secion 2.6]{blute2015cartesian}, that is, for every object $A$, the following is a coequalizer diagram: 
  \[  \xymatrixcolsep{5pc}\xymatrix{\oc\oc(A) \ar@<1ex>[r]^{\varepsilon_{\oc(A)}} \ar@<-1ex>[r]_{\oc(\varepsilon_A)} & \oc(A) \ar[r]^-{\varepsilon_A} & A } \]
\end{enumerate}
\end{lemma}

In the case of an exact comonad, the base category can be recovered from the coKleisli category using the subcategory of $\epsilon$-natural maps. For abstract coKleisli categories, note that the comonad from Lemma \ref{lem:ep-com} is always exact. 

For a comonad on the category with finite products, the coKleisli category is a Cartesian abstract coKleisli category. 

\begin{lemma} \cite[Section 2.6]{blute2015cartesian} Let $(\oc, \delta, \varepsilon)$ be a comonad on a category $\mathbb{X}$ with finite products. Then the coKleisli category $\mathbb{X}_\oc$ is a Cartesian abstract coKleisli category with abstract coKleisli structure as defined in Lemma \ref{cokleisliabstractlem}. 
\end{lemma}

For a comonad on a Cartesian left additive category, the coKleisli category is a Cartesian left additive abstract coKleisli category. 

\begin{lemma}  Let $(\oc, \delta, \varepsilon)$ be a comonad on a Cartesian left additive category $\mathbb{X}$. Then the coKleisli category $\mathbb{X}_\oc$ is a Cartesian left additive abstract coKleisli category with abstract coKleisli structure as defined in Lemma \ref{cokleisliabstractlem} and Cartesian left additive structure as defined in Lemma \ref{cokleisliCLAC}. 
\end{lemma}
\begin{proof} We must show that zero maps are $\epsilon$-natural, and that the sum of $\epsilon$-natural maps is again $\epsilon$-natural. We will make use of Lemma \ref{cokleisliabstractlem}.(\ref{lemepnatcok}). Starting with zero maps, we compute: 
\begin{align*}
\llbracket 0 \rrbracket \circ \varepsilon_{\oc(A)} &=~ 0 \circ \varepsilon_{\oc(A)} \\
&=~ 0 \\
&=~ 0 \circ \oc(\varepsilon_A) \\
&=~ \llbracket 0 \rrbracket \circ \oc(\varepsilon_A) 
\end{align*}
So $\llbracket 0 \rrbracket \circ \varepsilon_{\oc(A)} = \llbracket 0 \rrbracket \circ \oc(\varepsilon_A) $, so by Lemma \ref{cokleisliabstractlem}.(\ref{lemepnatcok}), we conclude that $\llbracket 0 \rrbracket$ is $\epsilon$-natural. Next, suppose that ${\llbracket f \rrbracket: \oc(A) \to B}$ and $\llbracket g \rrbracket: \oc(A) \to B$ are both $\epsilon$-natural, then we compute 
\begin{align*}
\llbracket f+g \rrbracket \circ \varepsilon_{\oc(A)} &=~ \left( \llbracket f \rrbracket + \llbracket g \rrbracket \right) \circ \varepsilon_{\oc(A)} \\
&=~ \llbracket f \rrbracket \circ \varepsilon_{\oc(A)} +  \llbracket g \rrbracket \circ \varepsilon_{\oc(A)} \\
&=~  \llbracket f \rrbracket \circ \oc(\varepsilon_A)  +  \llbracket g \rrbracket \circ \oc(\varepsilon_A) \tag{$\llbracket f \rrbracket$ and $\llbracket g \rrbracket$ are $\epsilon$-nat. + Lem.\ref{cokleisliabstractlem}.(\ref{lemepnatcok}) } \\
&=~ \left( \llbracket f \rrbracket + \llbracket g \rrbracket \right) \circ \oc(\varepsilon_A) \\
&=~ \llbracket f+g \rrbracket \circ  \oc(\varepsilon_A)
\end{align*}
So $\llbracket f+g \rrbracket \circ \varepsilon_{\oc(A)} = \llbracket f+g \rrbracket \circ  \oc(\varepsilon_A)$, and so by Lemma \ref{cokleisliabstractlem}.(\ref{lemepnatcok}) it follows that $\llbracket f+g \rrbracket$ is $\epsilon$-natural. Therefore, we conclude that $\mathbb{X}_\oc$ is a Cartesian left additive abstract coKleisli category. 
\end{proof}

We will now show that for a Cartesian differential comonad, its coKleisli category is a Cartesian differential abstract coKleisli category. 

\begin{proposition}\label{propabcok} Let $(\oc, \delta, \varepsilon)$ be a Cartesian differential comonad on a category $\mathbb{X}$ with finite biproducts. Then $\mathbb{X}_\oc$ is a Cartesian differential abstract coKleisli category with Cartesian differential structure defined in Theorem \ref{thm1} and abstract coKleisli structure $(\oc_\oc, \varphi, \epsilon)$ as defined in Lemma \ref{cokleisliabstractlem}. 
\end{proposition} 	
\begin{proof} We must show that every $\epsilon$-natural map is $\mathsf{D}$-linear. To do so, we make use of both Theorem \ref{thm1}.(\ref{thm1.lin}) and Lemma \ref{cokleisliabstractlem}.(\ref{lemepnatcok}). So suppose that $\llbracket f \rrbracket: \oc(A) \to B$ is $\epsilon$-natural, then we compute: 
\begin{align*}
\llbracket f \rrbracket \circ \partial_A \circ \oc(\iota_1) &=~\llbracket f \rrbracket \circ \partial_A \circ \oc(\iota_1) \circ 1_{\oc(A)} \\
&=~ \llbracket f \rrbracket \circ \partial_A \circ \oc(\iota_1) \circ \oc(\varepsilon_A) \circ \delta_A \tag{Comonad Identity} \\
&=~ \llbracket f \rrbracket \circ \partial_A \circ \oc(\iota_1 \circ \varepsilon_A) \circ \delta_A \tag{$\oc$ is a functor} \\
&=~ \llbracket f \rrbracket \circ \partial_A \circ \oc\left( (\varepsilon_A \times \varepsilon_A) \circ \iota_1 \right) \circ \delta_A \tag{Naturality of $\iota_1$} \\
&=~ \llbracket f \rrbracket \circ \partial_A \circ \oc(\varepsilon_A \times \varepsilon_A) \circ \oc(\iota_1) \circ \delta_A \tag{$\oc$ is a functor} \\
&=~ \llbracket f \rrbracket \circ \oc(\varepsilon_A) \circ \partial_{\oc(A)} \circ \oc(\iota_1) \circ \delta_A \tag{Naturality of $\partial$} \\
&=~ \llbracket f \rrbracket \circ \varepsilon_{\oc(A)} \circ \partial_{\oc(A)} \circ \oc(\iota_1) \circ \delta_A \tag{$\llbracket f \rrbracket$ is $\epsilon$-nat. + Lem.\ref{cokleisliabstractlem}.(\ref{lemepnatcok}) } \\
&=~ \llbracket f \rrbracket \circ \pi_1 \circ \varepsilon_{\oc(A) \times \oc(A)} \circ \oc(\iota_1) \circ \delta_A \tag{\textbf{[dc.3]}} \\
&=~ \llbracket f \rrbracket \circ \pi_1 \circ \iota_1 \circ \varepsilon_{\oc(A)} \circ \delta_A \tag{Naturality of $\varepsilon$} \\
&=~ \llbracket f \rrbracket \circ 1_{\oc(A)} \circ 1_{\oc(A)} \tag{Biproduct Idenity + Comonad Identity} \\
&=~  \llbracket f \rrbracket 
\end{align*}
So $\llbracket f \rrbracket \circ \partial_A \circ \oc(\iota_1) = \llbracket f \rrbracket$, and so by Theorem \ref{thm1}.(\ref{thm1.lin}), it follows that $\llbracket f \rrbracket$ is $\mathsf{D}$-linear. Therefore, we conclude that $\mathbb{X}_\oc$ is a Cartesian differential abstract coKleisli category. 
\end{proof}

We conclude this section by showing that for a Cartesian differential comonad with a $\mathsf{D}$-linear unit, the underlying comonad is exact and that a coKleisli map is $\mathsf{D}$-linear if and only if it $\epsilon$-natural. 

\begin{lemma}\label{etaexact} Let $(\oc, \delta, \varepsilon, \partial)$ be a Cartesian differential comonad on a category $\mathbb{X}$ with finite biproducts. Then the following are equivalent: 
\begin{enumerate}[{\em (i)}]
\item $(\oc, \delta, \varepsilon, \partial)$ has a $\mathsf{D}$-linear unit $\eta_A : A \to \oc (A)$;
\item The comonad $(\oc, \delta, \varepsilon)$ is exact and for each object $A$, the $\mathsf{D}$-linear map $\llbracket \mathsf{L}[\varphi_A] \rrbracket: \oc(A) \to \oc(A)$ is $\epsilon$-natural. 
\end{enumerate}
\end{lemma}
\begin{proof} For $(i) \Rightarrow (ii)$, suppose that $(\oc, \delta, \varepsilon, \partial)$ has a $\mathsf{D}$-linear unit $\eta_A : A \to \oc (A)$. By the comonad identities, we have that $\delta_A$ is a common section of $\varepsilon_{\oc(A)}$ and $\oc(\varepsilon_A)$, that is, $\varepsilon_{\oc(A)} \circ \delta_A =1_{\oc(A)} = \oc(\varepsilon_A) \circ \delta_A$. By \textbf{[du.1]}, we also have that $\eta_A$ is a section of $\varepsilon_A$, that is, $\varepsilon_A \circ \eta_A$. As such, this implies that the following diagram is a split coequalizer: 
  \[  \xymatrixcolsep{5pc}\xymatrix{\oc\oc(A) \ar@<1ex>[r]^{\varepsilon_{\oc(A)}} \ar@<-1ex>[r]_{\oc(\varepsilon_A)} & \oc(A) \ar[r]^-{\varepsilon_A} & A } \]
  However, every split coequalizer diagram is also an equalizer diagram. So we conclude that $(\oc, \delta, \varepsilon)$ is exact. Next we must show that $\llbracket \mathsf{L}[\varphi_A] \rrbracket$ is $\epsilon$-natural. Note that as was shown in the proof of Proposition \ref{etaFlem1}, $\llbracket \mathsf{L}[\varphi_A] \rrbracket = \eta_A \circ \varepsilon_A = \llbracket \mathsf{F}_\oc(\eta_A) \rrbracket$. Then by Lemma \ref{cokleisliabstractlem}.(\ref{lemepvarep}), it follows that $\llbracket \mathsf{L}[\varphi_A] \rrbracket$ is $\epsilon$-natural. 
% We first show that we have a coequalizer. So suppose that $\llbracket f \rrbracket: \oc(A) \to B$ coequalizes $\varepsilon_{\oc(A)}$ and $\oc(\varepsilon_A)$, that is, $\llbracket f \rrbracket \circ \varepsilon_{\oc(A)} = \llbracket f \rrbracket \circ \oc(\varepsilon_A)$. Note that by Lemma \ref{cokleisliabstractlem}.(\ref{lemepnatcok}), this implies that $\llbracket f \rrbracket$ is $\epsilon$-natural, and so by Proposition \ref{propabcok}, $\llbracket f \rrbracket$ is also $\mathsf{D}$-linear. Now consider the composite map $\llbracket f \rrbracket \circ \eta_A : A \to B$. Then by Corollary \ref{etacor1}, since $\llbracket f \rrbracket$ is $\mathsf{D}$-linear, $\llbracket f \rrbracket \circ \eta_A \circ \varepsilon_A = \llbracket f \rrbracket$. For uniqueness, suppose there was another map $h: A \to B$ such that $h \circ \varepsilon_A = \llbracket f \rrbracket$. Precomposing both sides by $\eta_A$, by \textbf{[du.1]}, we have that $h = \llbracket f \rrbracket \circ \eta_A$. So we conclude that we have a coequalizer and therefore that the comonad is exact. 

For $(ii) \Rightarrow (i)$, suppose that $(\oc, \delta, \varepsilon)$ is exact and $\llbracket \mathsf{L}[\varphi_A] \rrbracket$ is $\epsilon$-natural. By Lemma \ref{cokleisliabstractlem}.(\ref{lemepnatcok}), the latter implies that $\llbracket \mathsf{L}[\varphi_A] \rrbracket \circ \varepsilon_{\oc(A)} = \llbracket \mathsf{L}[\varphi_A] \rrbracket \circ \oc(\varepsilon_A)$. Therefore, by the couniversal property of the coequalizer, there exists a unique map $\eta_A: A \to \oc(A)$ such that the following diagram commutes: 
  \[  \xymatrixcolsep{5pc}\xymatrix{\oc\oc(A) \ar@<1ex>[r]^{\varepsilon_{\oc(A)}} \ar@<-1ex>[r]_{\oc(\varepsilon_A)} & \oc(A) \ar[dr]_-{\llbracket \mathsf{L}[\varphi_A] \rrbracket} \ar[r]^-{\varepsilon_A} & A \ar@{-->}[d]^-{\exists \oc ~ \eta_A}  \\
  && \oc(A)} \]
To show that $\eta$ is a $\mathsf{D}$-linear unit, we will make use of Lemma \ref{Lvarphi} and of the fact that $\varepsilon_A$ is epic (which is a consequence of the coequalizer assumption). Starting with naturality, for any map $f: A \to B$, we compute: 
\begin{align*}
\oc(f) \circ \eta_A \circ \varepsilon_A &=~ \oc(f) \circ \llbracket \mathsf{L}[\varphi_A] \rrbracket \\
&=~ \oc(f) \circ \partial_A \circ \oc(\iota_1) \tag{Lemma \ref{Lvarphi}} \\
&=~ \partial_B \circ \oc(f \times f) \circ \oc(\iota_1) \tag{Naturality of $\partial$} \\
&=~ \partial_B \circ \oc\left( (f \times f) \circ \iota_1 \right) \tag{$\oc$ is a functor} \\
&=~ \partial_B \circ \oc( \iota_1 \circ f) \tag{Naturality of $\iota_1$} \\
&=~ \partial_B \circ \oc(\iota_1) \circ \oc(f) \tag{$\oc$ is a functor} \\
&=~ \llbracket \mathsf{L}[\varphi_B] \rrbracket \circ \oc(f) \\
&=~ \eta_B \circ \varepsilon_B \circ \oc(f) \\ 
&=~ \eta_B \circ f \circ \varepsilon_A \tag{Naturality of $\varepsilon$} 
\end{align*}
So $\oc(f) \circ \eta_A \circ \varepsilon_A = \eta_B \circ f \circ \varepsilon_A$. Since $\varepsilon_A$ is epic, it follows that $\oc(f) \circ \eta_A = \eta_B \circ f$. Therefore, $\eta$ is a natural transformation. Next we show that $\eta$ satisfies the two axioms of a $\mathsf{D}$-linear unit: 
\begin{enumerate}[{\bf [du.1]}] 
\item Here we use \textbf{[dc.3]}: 
\begin{align*}
    \varepsilon_A \circ \eta_A \circ \varepsilon_A &=~ \varepsilon_A \circ \llbracket \mathsf{L}[\varphi_A] \rrbracket \\
    &=~ \varepsilon_A \circ \partial_A \circ \oc(\iota_1) \tag{Lemma \ref{Lvarphi}} \\
    &=~ \pi_1 \circ \varepsilon_{A \times A} \circ \oc(\iota_1) \tag{\textbf{[dc.3]}} \\
    &=~ \pi_1 \circ \iota_1  \circ \varepsilon_A  \tag{Naturality of $\varepsilon$} \\
    &=~ 1_A \circ \varepsilon_A \tag{Biproduct Identity}\\ 
    &=~ \varepsilon_A
\end{align*}
So $\varepsilon_A \circ \eta_A \circ \varepsilon_A = \varepsilon_A$. Since $\varepsilon_A$ is epic, it follows that $\varepsilon_A \circ \eta_A = 1_A$. 
\item Automatic by Lemma \ref{Lvarphi} since: 
\begin{align*}
     \eta_A \circ \varepsilon_A &=~\llbracket \mathsf{L}[\varphi_A] \rrbracket \\
     &=~\partial_A \circ \oc(\iota_1) \tag{Lemma \ref{Lvarphi}} \\
\end{align*}
\end{enumerate}
So we conclude that $\eta$ is a $\mathsf{D}$-linear unit. 
\end{proof}

\begin{corollary}\label{etacor2} Let $(\oc, \delta, \varepsilon, \partial)$ be a Cartesian differential comonad with a $\mathsf{D}$-linear unit $\eta$ on a category $\mathbb{X}$ with finite biproducts. Then for a coKleisli map $\llbracket f \rrbracket: \oc(A) \to B$, $\llbracket f \rrbracket$ is $\mathsf{D}$-linear in $\mathbb{X}_\oc$ if and only if $\llbracket f \rrbracket$ is $\epsilon$-natural in $\mathbb{X}_\oc$. As such, we have the following chain of isomorphisms: $\mathbb{X} \cong \epsilon\text{-}\mathsf{nat}[\mathbb{X}_\oc] \cong  \mathsf{D}\text{-}\mathsf{lin}[\mathbb{X}_\oc]$
\end{corollary}
\begin{proof} By Proposition \ref{propabcok}, we already have that every $\epsilon$-natural map is $\mathsf{D}$-linear. So suppose that $\llbracket f \rrbracket$ is $\mathsf{D}$-linear. By Corollary \ref{etacor1}, since $\llbracket f \rrbracket$ is $\mathsf{D}$-linear, we have that $\llbracket f \rrbracket \circ \eta_A \circ \varepsilon_A = \llbracket f \rrbracket$. However, this also implies that $\llbracket f \rrbracket = \llbracket \mathsf{F}_\oc\left( \llbracket f \rrbracket \circ \eta_A \right) \rrbracket$. Then by Lemma \ref{cokleisliabstractlem}.(\ref{lemepvarep}), it follows that $\llbracket f \rrbracket$ is $\epsilon$-natural. By Proposition \ref{etaFlem1}, we have that $\mathbb{X} \cong \mathsf{D}\text{-}\mathsf{lin}[\mathbb{X}_\oc]$, while by Lemma \ref{etaexact} and Lemma \ref{lemexact}, we have that $\mathbb{X} \cong \epsilon\text{-}\mathsf{nat}[\mathbb{X}_\oc]$. \hfill \end{proof}

\section{Example: Reduced Power Series}\label{sec:PWex}

In this section we construct a Cartesian differential comonad (in the opposite category) based on \emph{reduced} formal power series, which therefore induces a Cartesian differential category of \emph{reduced} formal power series. To the extent of the authors' knowledge, this is a new observation.  This is an interesting and important non-trivial example of a Cartesian differential comonad which does not arise from a differential category. Unsurprisingly, the differential combinator will reflect the standard differentiation of arbitrary multivariable power series. However, the problem with arbitrary power series lies with composition. Indeed, famously, power series with degree 0 coefficients, also called constant terms, cannot be composed, since in general this results in an infinite non-converging sum in the base field. Thus, multivariable formal power series do not form a category, since their composition may be undefined. \emph{Reduced} formal power series are power series with no constant term. These can be composed \cite[Section 4.1]{brewer2014algebraic} and thus, we obtain a Lawvere theory of reduced power series. The total derivative of a reduced power series is again reduced, and therefore, we obtain a Cartesian differential category of reduced power series. Futhermore, this Cartesian differential category of reduced power series is in fact a subcategory of the opposite category of the Kleisli category of the coCartesian differential monad $\mathsf{P}$, the free reduced power series algebra monad, which can be seen as the free complete algebra functor induced by the operad of commutative algebras \cite[Section 1.4.4]{Fresse98}. Lastly, it is worth mentioning that, while in this section we will work with vector spaces over a field, we note that all the constructions easily generalize to the category of modules over a commutative (semi)ring.

Let $\mathbb{F}$ be a field. For an $\mathbb{F}$-vector space $V$, define $\mathsf{P}(V)$ as follows: 
\[ \mathsf{P}(V)= \prod^{\infty}_{n=1} \left(V^{\otimes n}\right)_{S(n)}, \]
where $(V^{\otimes n})_{S(n)}$ denotes the vector space of symmetrized $n$-tensors, that is, classes of tensors of length $n$ under the action of the symmetric group which permutes the factors in $V^{\otimes n}$. An arbitrary element $\mathfrak{t} \in \mathsf{P}(V)$ is then an infinite ordered list $\mathfrak{t} = \left( \mathfrak{t}(n) \right)^\infty_{n=1}$ where $\mathfrak{t}(n) \in (V^{\otimes n})_{S(n)}$. Therefore, an arbitrary element of $ \mathsf{P}(V)$ can be written in the following form: 
\[ \mathfrak{t} = \left( \mathfrak{t}(n) \right)^\infty_{n=1} = \left( \sum \limits^m_{i=1} v_{(n,i,1)}  \hdots  v_{(n,i,n)} \right)^\infty_{n=1}  , \]
where $v_{(n,k,1)}  \hdots  v_{(n,k,n)}$ denotes the class of $v_{(n,k,1)} \otimes \hdots \otimes v_{(n,k,n)}\in V^{\otimes n}$ under the action of the symmetric group. If $X$ is basis of $V$, then $\mathsf{P}(V) \cong \mathbb{F}\llbracket X \rrbracket_+$ \cite[Section 1.4.4]{Fresse98}, where $\mathbb{F}\llbracket X \rrbracket_+$ is the non-unital associative ring of reduced power series over $X$, that is, power series over $X$ with no constant/degree $0$ term. Therefore, $\mathsf{P}(V)$ is a non-unital associative $\mathbb{F}$-algebra. The algebra structure is induced by concatenation of classes of tensors:
\[*:v_{1} \dots v_{n}\otimes w_1\dots w_{k}\mapsto v_{1} \dots v_{n}w_1\dots w_{k}\]
 which provides a commutative, associative multiplication $*:(V^{\otimes n})_{S(n)}\otimes (V^{\otimes k})_{S(k)}\to(V^{\otimes n+k})_{S(n+k)}$. We will sometimes abbreviate a  juxtaposition of a finite family $(v_i)_{i =1}^n$ of elements of $V$ (resp. a concatenation of a finite family of symmetrized tensors $(\underline v_i)_{i=1}^n$ with $\underline v_i\in (V^{\otimes n_i})_{S(n_i)}$) by:
 \begin{align*}
     \prod_{i\in I}v_i=v_1\hdots v_i, &&\left(\mbox{resp. }\prod_{i\in I}\underline v_i=\underline v_1*\hdots*\underline v_{n}\right).
 \end{align*}

It is worth pointing out that $\mathsf{P}(V)$ does not have a unit element. More specifically, $\mathsf{P}(V)$ will not come equipped with a natural map of type $\mathbb{F} \to \mathsf{P}(V)$. So $\mathsf{P}(V)$ will not induce an algebra modality, and therefore will not induce a differential category structure on $\mathbb{F}\text{-}\mathsf{VEC}^{op}$.

The monad structure of $\mathsf{P}$ corresponds to the composition of power series, which is closely related to the composition of polynomials. Define the functor $\mathsf{P}: \mathbb{F}\text{-}\mathsf{VEC} \to \mathbb{F}\text{-}\mathsf{VEC}$ as mapping an $\mathbb{F}$-vector space $V$ to $\mathsf{P}(V)$, as defined above, and mapping an $\mathbb{F}$-linear map $f: V \to W$ to the $\mathbb{F}$-linear map ${\mathsf{P}(f): \mathsf{P}(V) \to \mathsf{P}(W)}$ defined on elements $\mathfrak t$ as above by: 
\begin{align*}
 \mathsf{P}(f)(\mathfrak{t}) = \left( \sum \limits^m_{i=1} f(v_{(n,i,1)}) \hdots f(v_{(n,i,n)}) \right)^\infty_{n=1}.
\end{align*}
%where, abusing notation slightly, $\mathfrak{p}(n) \in \mathsf{Sym}_n(V)$ is understood as an element of $\mathsf{Sym}(V)$, and therefore we have that ${\mathsf{Sym}(f)\left( \mathfrak{p}(n) \right) \in \mathsf{Sym}_n(W)}$ is well defined. 
Define the monad unit $\eta_V: V \to \mathsf{P}(V)$ by: 
\begin{align*}
   \eta_V(v) = (v, 0, 0, \hdots).
\end{align*}
From a power series point of view, if $X$ is a basis of $V$, $\eta_V$ maps a basis element $x \in X$ to its associated monomial of degree $1$. For the monad multiplication, let us first consider an element $\mathfrak{s} \in \mathsf{P}\mathsf{P}(V)$, which is a list of symmetrized tensor products of lists of symmetrized tensor products, $\mathfrak{s} = \left( \mathfrak{s}(n) \right)^\infty_{n=1}$, $\mathfrak{s}(n) \in \left((\mathsf{P}(V))^{\otimes n}\right)_{S(n)}$ and thus, $\mathfrak{s}(n)$ is of the form: 
\[\mathfrak{s}(n) = \sum \limits^m_{i=1} \mathfrak{s}(n)_{(i,1)}\hdots \mathfrak{s}(n)_{(i,n)} \]
for some $\mathfrak{s}(n)_{(i,j)} \in \mathsf{P}(V)$. Now for every partition of $n$ not involving $0$, that is, for every $n_1+ \hdots + n_k = n$ with $n_j \geq 1$, define $\mathfrak{s}(n_1, \hdots, n_k) \in (V^{\otimes n})_{S(n)}$ as follows:
\[\mathfrak{s}(n_1, \hdots, n_k) = \sum \limits^m_{i=1}\mathfrak{s}(k)_{(i,1)}(n_1)*\hdots* \mathfrak{s}(k)_{(i,k)}(n_k),  \ \]
where $*$ is the concatenation multiplication defined above. Define the monad multiplication $\mu_V: \mathsf{P}\mathsf{P}(V) \to \mathsf{P}(V)$ as follows: 
\begin{align*}
 \mu_V(\mathfrak{s}) =  \left(\sum\limits_{k=1}^n \sum \limits_{n_1 + \hdots + n_k=n} \mathfrak{s}(n_1, \hdots, n_k) \right)^\infty_{n=1} &&
 \mu_V(\mathfrak{s})(n) = \sum\limits_{k=1}^n \sum \limits_{n_1 + \hdots + n_k=n} \mathfrak{s}(n_1, \hdots, n_k)
\end{align*}
This monad multiplication corresponds to the composition of multivariable reduced power series, as defined explicitly in \cite[Section 4.1]{brewer2014algebraic}. 
%It is interesting to point out that if the indexing of $\mathsf{P}(V)$ starts at $n=0$, and thus would include $\mathbb{F}$ representing the constant terms of power series, it would not be possible to define a well-defined monad multiplication, for the same reason one cannot always compose formal power series with degree 0 coefficients. 

\begin{lemma}\label{lem:powmonad}\cite[Section 1.4.3]{Fresse98}
$(\mathsf{P}, \mu, \eta)$ is a monad on $\mathbb{F}\text{-}\mathsf{VEC}$.   
\end{lemma}
%\begin{proof}
%We leave it as an exercise for the motivated reader to compute and prove the monad identities directly. The monad identity $\mu \circ \mathsf{P}(\eta)= 1 = \mu \circ \eta$ corresponds to the identity rules for composition, while $\mu \circ \mu = \mu \circ \mathsf{P}(\mu)$ corresponds to associativity of composition.
%\end{proof}
We now introduce the differential combinator transformation for $\mathsf{P}$, which will correspond to differentiating power series. 
%Since differentiating power series is the same as differentiating polynomials at each coefficient, we may use the differential combinator transformation for $\mathsf{Sym}$ (as defined in Example \ref{ex:VEC!}), which we will denote here as $\partial^{\mathsf{Sym}}_V: \mathsf{Sym}(V) \to \mathsf{Sym}(V \times V)$. 
Define $\partial_V: \mathsf{P}(V) \to \mathsf{P}(V \times V)$ by setting: 
\begin{align*}
\partial_V(\mathfrak{t})(n) =  \sum \limits^m_{i=1}\sum_{j=1}^n \left((v_{(n,i,1)},0)  \hdots \widehat{(v_{(n,i,j)},0)}\hdots (v_{(n,i,n)},0)\right) (0,v_{n,i,j}), 
\end{align*}
where $\mathfrak t$ is an arbitrary element of $\mathsf P(V)$ as above and $\widehat{(v_{(n,i,j)},0)}$ indicates the omission of the factor $(v_{(n,i,j)},0)$ in the product. With our conventions, we can abbreviate:
\begin{align*}
\partial_V(\mathfrak{t})(n) =  \sum \limits^m_{i=1}\sum_{j=1}^n \left(\prod_{k\neq j}(v_{(n,i,k)},0) \right) (0,v_{n,i,j}).
\end{align*}
If $X$ is a basis of $V$, the differential combinator transformation can described as a map ${\partial_V: \mathbb{F}\llbracket X \rrbracket_+ \to \mathbb{F}\llbracket X \sqcup X \rrbracket_+}$ which maps a reduced power series $\mathfrak{t}(\vec x)$ to its sum of its partial derivatives: 
\[ \partial_V(\mathfrak t(\vec x)) = \sum \limits_{x_i \in \vec x} \frac{\partial \mathfrak{t}(\vec x)}{\partial x_i} x^\ast_i \]
where $x^\ast_i$ denotes the element $x_i$ in the second copy of $X$ in the disjoint union $X \sqcup X$. Note that even if $\mathfrak{t}(\vec x)$ depends on an infinite list of variables, $\partial_V(\mathfrak{t}(\vec x))$ is well-defined as a formal power series. 

It is worth insisting on the fact that $\partial$ cannot be induced by a deriving transformation in the sense of Example \ref{ex:diffcat}. Indeed, as a map, $\partial$ does not factor through a map $\mathsf P(V)\to\mathsf P(V)\otimes V$. Note that a power series could have infinite partial derivatives and, since infinite sums and $\otimes$ are generally incompatible, the derivative of a power series could not be described as an element of $\mathsf{P}(V) \otimes V$. Moreover, we already noted the lack of unit: a differential operator of type $\mathsf{P}(V) \to \mathsf{P}(V) \otimes V$ would not be able to properly derive degree 1 monomials without a unit argument to put in the $\mathsf{P}(V)$ component.

%To prove that $\partial$ is indeed a differential combinator transformation, we will make use of the fact that we already know from Example \ref{ex:VEC!} that $\partial^{\mathsf{Sym}}$ is a differential combinator. Indeed, since the necessary structure maps of $\mathsf{P}$ are defined using those of $\mathsf{Sym}$, it will then be mostly straightforward to prove the differential combinator transformation axioms. 

\begin{proposition}\label{prop:powpartial} $\partial$ is a differential combinator transformation on $(\mathsf{P}, \mu, \eta)$.
\end{proposition}
\begin{proof}
First, it is straightforward to see that $\partial$ is indeed a natural transformation. 
%We must first check that $\partial$ is indeed a natural transformation. To do so, we use the naturality of $\partial^{\mathsf{Sym}}$, which is that $\partial^{\mathsf{Sym}}_W \circ \mathsf{Sym}(f) = \mathsf{Sym}(f \times f) \circ \partial^{\mathsf{Sym}}_V$. Then for an arbitrary $\mathbb{F}$-linear map $f: V \to W$, we compute: 
%\begin{align*}
%\partial_W\left(  \mathsf{P}(f)\left( \mathfrak{p} \right) \right) &=~ \left(\partial^{\mathsf{Sym}}_W \left( \mathsf{P}(f)\left( \mathfrak{p} \right)(n) \right)\right)^\infty_{n=1} \\
%&=~ \left(\partial^{\mathsf{Sym}}_W \left( \mathsf{Sym}(f)\left( \mathfrak{p}(n) \right) \right)\right)^\infty_{n=1} \\
%&=~ \left(\mathsf{Sym}(f \times f) \left( \partial^{\mathsf{Sym}}_V\left( \mathfrak{p}(n) \right) \right)\right)^\infty_{n=1} \tag{Naturality of $\partial^{\mathsf{Sym}}$} \\
%&=~ \left(\mathsf{Sym}(f \times f) \left( \partial_V\left( \mathfrak{p} \right)(n) \right)\right)^\infty_{n=1} \\
%&=~ \mathsf{P}(f \times f) \left( \partial_V\left( \mathfrak{p} \right) \right)
%\end{align*}
%So we conclude that $\partial$ is natural. 
We must show that $\partial$ satisfies the dual of the six axioms \textbf{[dc.1]} to \textbf{[dc.6]} from Definition \ref{def:cdcomonad}. Throughout this proof, it is sufficient to do the calculations on an arbitrary element $\mathfrak t \in \mathsf P(V)$. We can further assume that for all $n\in\mathbb{N}$, $\mathfrak t(n)=v_{(n,1)}*\hdots *v_{(n,n)}$ (the integer $m$ as above is equal to 1), and then extend by linearity.
\begin{enumerate}[{\bf [dc.1]}] 
\item One has, for all $n>0$:
			\begin{align*}
				\mathsf P(\pi_0)\circ \p_V(\t)(n)&= \mathsf P(\pi_0)\left(\sum_{j=1}^n \left(\prod_{k\neq j}(v_{(n,k)},0) \right) (0,v_{(n,j)})\right) =\sum_{j=1}^n\left( \prod_{k\neq j}v_{(n,k)}\right)0=0.
			\end{align*}
So $\mathsf{P}(\pi_0) \circ \partial_V = 0$.			
			\item Let $\Delta: V \to V \times V$ be the diagonal map. One computes that:
			\begin{align*}
			 	\mathsf P(V\times \Delta)\circ\p_V(\t)(n)&=\mathsf P(V\times \Delta)\left(\sum_{j=1}^n \left(\prod_{k\neq j}(v_{(n,k)},0) \right) (0,v_{(n,j)})\right),\\
			 	&=\sum_{j=1}^n\left(\prod_{k\neq j}(v_{n,k},0,0)\right)(0,v_{(n,j)},v_{(n,j)}),\\
			 		&=\left(\sum_{j=1}^n\left(\prod_{k\neq j}(v_{n,k},0,0)\right)(0,v_{(n,j)},0)\right)+\left(\sum_{j=1}^n\left(\prod_{k\neq j}(v_{n,k},0,0)\right)(0,0,v_{(n,j)})\right),\\
			 &=\left(\mathsf P\mathsf P(V\times\iota_0)+\mathsf P(V\times\iota_1)\right)\left(\sum_{j=1}^n \left(\prod_{k\neq j}(v_{(n,k)},0) \right) (0,v_{(n,j)})\right),\\
			 	&=\left(\mathsf P(V\times\iota_0)+\mathsf P(V\times\iota_1)\right)\circ\p_V(\t)(n).
			 \end{align*}
So $\mathsf{P}(1_V \times \Delta_V) \circ \partial_V = \left(\mathsf{P}(1_V \times \iota_0) + \mathsf{P}(1_V \times \iota_1) \right) \circ  \partial_V$.			 
			 \item This is a straightforward verification. So $\partial_V \circ \eta_V = \eta_{V \times V} \circ \iota_1$.
%			 On one hand,
%			 \begin{equation*}
%			 	\p_V\circ \eta_V(a)=\p_V(a)=(0,a).
%			 \end{equation*}
%			 On the other hand,
%			 \begin{equation*}
%			 	\eta_{V\times V}\circ\iota_1(a)=\eta_{V\times V}(0,a)=(0,a).
%			 \end{equation*}
			 \item Let $\s$ be an element in $\mathsf P(\mathsf P(V))$ such that
			 \[
			 	\s(k)=\s(k)_{1}\dots\s(k)_{k},
			 \]
			 with, for all $j\in\{1,\dots,k\}, n>0$,
			 \[
			 	\s(k)_{j}(n)=v_{(k,j,n,1)}\dots v_{(k,j,n,n)},
			 \]
			 where $v_{(k,j,n,l)}\in V$ for all $k,j,n,l$. Note that elements of this type span $\mathsf P(\mathsf P(V))$. We then can make the computation on $\mathfrak s$ and extend by linearity.

			 On one hand, for all $n>0$,
			 \begin{align*}
			     \mu_{\mathsf P}(\s)(n)&=\sum\limits_{k=1}^n \sum \limits_{n_1 + \hdots + n_k=n} \mathfrak{s}(n_1, \hdots, n_k),\\
			     &=\sum\limits_{k=1}^n \sum \limits_{n_1 + \hdots + n_k=n}\mathfrak s(k)_1(n_1)*\hdots* \mathfrak s(k)_k(n_k),\\
			     &=\sum\limits_{k=1}^n \sum \limits_{n_1 + \hdots + n_k=n}v_{(k,1,n_1,1)}\hdots v_{(k,1,n_1,n_1)}\hdots v_{(k,k,n_k,1)}\hdots v_{(k,k,n_k,n_k)},
			 \end{align*}
%			 which we can abbreviate:
%			 \[\mu_{\mathsf P}(\s)(n)=\sum\limits_{k=1}^n \sum \limits_{n_1 + \hdots + n_k=n}\prod_{j=1}^k\prod_{l=1}^{n_j}v_{(k,j,n_j,l)},\]
			 and so,
			 \begin{multline*}
			 	\p_V\circ\mu_{\mathsf P}(\s)(n)=\\\sum\limits_{k=1}^n \sum \limits_{n_1 + \hdots + n_k=n}\sum_{j=1}^{k}\sum_{l=1}^{n_j}(v_{(k,1,n_1,1)},0)\hdots (v_{(k,1,n_1,n_1)},0)\hdots \widehat{(v_{(k,j,n_j,l)},0)}\hdots (v_{(k,k,n_k,n_k)},0)\ (0,v_{(k,j,n_j,l)})
			 \end{multline*}
			 On the other hand, for all $k>0$,
			 \[
			 	\p_{\mathsf P(V)}(\s)(k)=\sum_{j=1}^k\mathfrak (s(k)_1,0)\hdots \widehat{(s(k)_j,0)}\hdots (s(k)_k,0)\ (0,s(k)_j),
			 \]
			So if $[-,-]$ is the pairing operator for the coproduct, we have that: 
			 \[
			 	\mathsf P\left(\left[\mathsf P(\iota_0),\p_V\right]\right)\circ\p_{\mathsf P(V)}(\s)(k)=\sum_{j=1}^k\mathsf P(\iota_0)(\s(k)_{1})\hdots \widehat{\mathsf P(\iota_0)(\s(k)_{j})}\hdots \mathsf P(\iota_0)(\s(k)_{k})\ \partial_{V}(\s(k)_j),
			 \]
			 and so, for all $n>0$,
            \begin{align*}
			 &	\mu_{\mathsf P(V)}\circ\mathsf P\left(\left[\mathsf P(\iota_0),\p_V\right]\right)\circ\p_{\mathsf P(V)}(\s)(n)=\sum_{k=1}^n\sum_{n_1+\dots+n_k=n}\mathsf P\left(\left[\mathsf P(\iota_0),\p_V\right]\right)\circ\p_{\mathsf P(V)}(\s)(n_1,\dots,n_k)\\
			 	&=\sum_{k=1}^n\sum_{n_1+\dots+n_k=n}\sum_{j=1}^k\mathsf P(\iota_0)(\s(k)_{1})*\hdots* \widehat{\mathsf P(\iota_0)(\s(k)_{j})}*\hdots* \mathsf P(\iota_0)(\s(k)_{k})\ *\partial_{V}(\s(k)_j),\\
			 	&=\sum\limits_{k=1}^n \sum \limits_{n_1 + \hdots + n_k=n}\sum_{j=1}^{k}\sum_{l=1}^{n_j}(v_{(k,1,n_1,1)},0)\hdots(v_{(k,1,n_1,n_1)},0)\hdots \widehat{(v_{(k,j,n_j,l)},0)}\hdots (v_{(k,k,n_k,n_k)},0)\ (0,v_{(k,j,n_j,l)})
			 \end{align*}
So $\partial_V \circ \mu_V = \mu_{V \times V} \circ \mathsf{P}\left( [\mathsf{P}(\iota_0), \partial_V ] \right) \circ \partial_{\mathsf{P}(V)}$. 	
\end{enumerate}
In order to prove the two last identities, let us compute $\partial_{V\times V}(\partial_V(\t))$. For all $n>0$,
\begin{align*}
&\p_{V\times V}(\p_V(\t))(n)=\\
			&\sum_{j_0\neq j_1\in\{1,\dots,n\}}\left(\prod_{k\neq j_0,j_1}(v_{(n,i,k)},0,0,0) \right) (0,v_{n,i,j_0},0,0)(0,0,v_{n,i,j_1},0)+\sum_{j=1}^n\left(\prod_{k\neq j}(v_{(n,i,k)},0,0,0) \right) (0,0,0,v_{n,i,j})
\end{align*}
\begin{enumerate}[{\bf [dc.1]}] 
	\setcounter{enumi}{4}
\item Using our computation for $\partial_{V\times V}(\partial_V(\t))$, we get, for all $n>0$:
			\begin{multline*}
				\mathsf P(\pi_0\times\pi_1)\circ\p_{V\times V}(\p_V(\t))(n)=\\\sum_{j_0\neq j_1\in\{1,\dots,n\}}\left(\prod_{k\neq j_0,j_1}(v_{(n,i,k)},0) \right) (0,0)(0,0)+\sum_{j=1}^n\left(\prod_{k\neq j}(v_{(n,i,k)},0) \right) (0,v_{n,i,j})=\partial_V(\t)
			\end{multline*}
			so $\mathsf P(\pi_0\times\pi_1)\circ\p_{V\times V}(\p_V(\t))=\p_V(\t)$.
			\item Our computation of $\partial_{V\times V}(\partial_V(\t))$ clearly shows that this element is fixed under the action of $\mathsf P(c_V)$. So $\mathsf{P}(c_V) \circ \partial_{V \times V} \circ \partial_V = \partial_{V \times V} \circ \partial_V$.
\end{enumerate}
So we conclude that $\partial$ is a differential combinator transformation. \end{proof}

Thus $(\mathsf{P}, \mu, \eta, \partial)$ is a coCartesian differential monad and so the opposite of its Kleisli category is a Cartesian differential category (which we summarize in Corollary \ref{cor:POW} below) which, as we will explain below, captures power series differentiation. This coCartesian differential monad also comes equipped with a $\mathsf{D}$-linear counit. Define $\varepsilon_V: \mathsf{P}(V) \to V$ as simply the projection onto $V$: 
\[ \varepsilon_V\left( \mathfrak{t} \right) = \mathfrak{t}(1) \]
From a power series point of view, $\varepsilon$ projects out the degree 1 coefficients of a reduced power series. 

\begin{lemma}\label{lem:powepsilon} $\varepsilon$ is a $\mathsf{D}$-linear counit of $(\mathsf{P}, \mu, \eta, \partial)$.
\end{lemma}
\begin{proof} The proof, which is a straightforward verification, is left to the reader as an exercise.
\end{proof}

Therefore, the subcategory of $\mathsf{D}$-linear maps of the opposite category of the Kleisli category of $\mathsf{P}$ is isomorphic to the opposite category of $\mathbb{F}\text{-}\mathsf{VEC}$. We summarize these results as follows:  

\begin{corollary}\label{cor:POW} $(\mathsf{P}, \mu, \eta, \partial)$ is a Cartesian differential comonad on $\mathbb{F}\text{-}\mathsf{VEC}^{op}$ with $\mathsf{D}$-linear unit $\varepsilon$. Therefore $\mathbb{F}\text{-}\mathsf{VEC}^{op}_\mathsf{P}$ is a Cartesian differential category and $\mathsf{D}\text{-}\mathsf{lin}\left[ \mathbb{F}\text{-}\mathsf{VEC}^{op}_{\mathsf{P}} \right] \cong \mathbb{F}\text{-}\mathsf{VEC}^{op}$.
\end{corollary}

The Cartesian differential category $\mathbb{F}\text{-}\mathsf{VEC}^{op}_\mathsf{P}$ can be interpreted as the category whose objects are $\mathbb{F}$-vector spaces and whose maps are reduced power series between them. As a result, focusing on the finite-dimensional vector spaces, specifically $\mathbb{F}^n$, one obtains a Cartesian differential category of reduced power series over finite variables. We describe this category in detail. 

\begin{example} \normalfont \label{ex:CDCPOW} Let $\mathbb{F}$ be a field. Define the category $\mathbb{F}\text{-}\mathsf{POW}_{red}$ whose object are $n \in \mathbb{N}$, where a map ${\mathfrak{P}: n \to m}$ is a $m$-tuple of reduced power series (i.e. power series with no degree $0$ coefficients) in $n$ variables, that is, $\mathfrak{P} = \langle \mathfrak{p}_1(\vec x), \hdots, \mathfrak{p}_m(\vec x) \rangle$ with $\mathfrak{p}_i(\vec x) \in \mathbb{F}\llbracket x_1, \hdots, x_n\rrbracket_+$. The identity maps $1_n: n \to n$ are the tuples $1_n = \langle x_1, \hdots, x_n \rangle$ and where composition is given by multivariable power series substitution \cite[Section 4.1]{brewer2014algebraic}. $\mathbb{F}\text{-}\mathsf{POW}_{red}$ is a Cartesian left additive category where the finite product structure is given by $n \times m = n +m$ with projection maps ${\pi_0: n \times m \to n}$ and ${\pi_1: n \times m \to m}$ defined as the tuples $\pi_0 = \langle x_1, \hdots, x_n \rangle$ and $\pi_1 = \langle x_{n+1}, \hdots, x_{n+m} \rangle$, and where the additive structure is defined coordinate-wise via the standard sum of power series. $\mathbb{F}\text{-}\mathsf{POW}_{red}$ is also a Cartesian differential category where the differential combinator is given by the standard differentiation of power series, that is, for a map ${\mathfrak{P}: n \to m}$, with $\mathfrak{P} = \langle \mathfrak{p}_1(\vec x), \hdots, \mathfrak{p}_m(\vec x) \rangle$, its derivative $\mathsf{D}[\mathfrak{P}]: n \times n \to m$ is defined as the tuple of the sum of the partial derivatives of the power series $\mathfrak{p}_i(\vec x)$:
\begin{align*}
\mathsf{D}[\mathfrak{P}](\vec x, \vec y) := \left( \sum \limits^n_{i=1} \frac{\partial \mathfrak{p}_1(\vec x)}{\partial x_i} y_i, \hdots, \sum \limits^n_{i=1} \frac{\partial \mathfrak{p}_n(\vec x)}{\partial x_i} y_i \right) && \sum \limits^n_{i=1} \frac{\partial \mathfrak{p}_j (\vec x)}{\partial x_i} y_i \in \mathbb{F}\llbracket x_1, \hdots, x_n, y_1, \hdots, y_n \rrbracket_+
\end{align*} 
It is important to note that even if $\mathfrak{p}_j(\vec x)$ has terms of degree 1, every partial derivative $\frac{\partial \mathfrak p_j(\vec x)}{\partial x_i} y_i$ will still be reduced (even if $\frac{\partial \mathfrak p_j(\vec x)}{\partial x_i}$ has a degree 0 term), and thus the differential combinator $\mathsf{D}$ is indeed well-defined. A map ${\mathfrak{P}: n \to m}$ is $\mathsf{D}$-linear if it of the form: 
\begin{align*}
\mathfrak{P} = \left \langle \sum \limits^{n}_{i=0} r_{i,1}x_{i}, \hdots, \sum \limits^{n}_{i=0} r_{i,m}x_{i} \right \rangle && r_{i,j} \in \mathbb{F}
\end{align*}
Thus $\mathsf{D}\text{-}\mathsf{lin}[\mathbb{F}\text{-}\mathsf{POW}_{red}]$ is equivalent to $\mathbb{F}\text{-}\mathsf{LIN}$ (as defined in Example \ref{ex:CDCPOLY}). We note that this example generalize to the category of reduced formal power over an arbitrary commutative (semi)ring.
\end{example}

Observe that we also have the following chain of isomorphisms: 
\begin{align*}
\mathbb{F}\text{-}\mathsf{POW}_{red}(n,1) = \mathbb{F}\llbracket x_1, \hdots, x_n\rrbracket_+ \cong \mathsf{P}(\mathbb{F}^n) \cong \mathbb{F}\text{-}\mathsf{VEC} \left(\mathbb{F}, \mathsf{P}(\mathbb{F}^n) \right) = \mathbb{F}\text{-}\mathsf{VEC}_{\mathsf{P}}\left(\mathbb{F}, \mathbb{F}^n \right) = \mathbb{F}\text{-}\mathsf{VEC}^{op}_{\mathsf{P}}\left(\mathbb{F}^n, \mathbb{F} \right) 
\end{align*}
which then implies that $\mathbb{F}\text{-}\mathsf{POW}_{red}(n,m) \cong \mathbb{F}\text{-}\mathsf{VEC}^{op}_{\mathsf{P}}\left(\mathbb{F}^n, \mathbb{F}^m \right)$. Thus $\mathbb{F}\text{-}\mathsf{POW}_{red}$ is isomorphic to the full subcategory of $\mathbb{F}\text{-}\mathsf{VEC}^{op}_{\mathsf{P}}$ whose objects are the finite dimensional $\mathbb{F}$-vector spaces. In the finite dimensional case, the differential combinator transformation corresponds precisely to the differential combinator on $\mathbb{F}\text{-}\mathsf{POW}_{red}$:
\[\partial_{\mathbb{F}^n}(\mathfrak{p}(\vec x)) = \mathsf{D}[\mathfrak{p}](\vec x, \vec y)\] 
Therefore, $\mathbb{F}\text{-}\mathsf{POW}_{red}$ is a sub-Cartesian differential category of $\mathbb{F}\text{-}\mathsf{VEC}^{op}_{\mathsf{P}}$, where the latter allows for power series over infinite variables. 
	
\section{Example: Divided Power Algebras}\label{secpuisdiv}
In this section, we show that the free divided power algebra monad is a coCartesian differential monad, and therefore, we obtain a Cartesian differential category of divided power polynomials \cite[Section 12]{roby68}. Divided power algebras were introduced by Cartan \cite{cartan1954puissances} to study the homology of Eilenberg-MacLane spaces with coefficients in a prime field of positive characteristic. Such structures appear notably on the homotopy of simplicial algebras \cite{cartan1954puissances, fresse2000}, and in the study of $D$-modules and crystalline cohomology \cite{berthelot74}. The free divided power algebra monad $\mathsf{\Gamma}$ was first introduced by Roby in \cite{roby65} and generalized in the context of operads by Fresse in \cite{fresse2000}. Much as for reduced power series, the composition of divided power polynomials is only well-defined when they are reduced, that is, have no constant term. More generally, the study of divided power algebras has been widely developed in the non-unital setting \cite{fresse2000}. Since the monad we study encodes a structure of non-unital algebras, this provides another example of a Cartesian differential comonad which is not induced by a differential category. We begin by reviewing the definition of a divided power algebra.

\begin{definition}\label{defipuisdiv} Let $\mathbb{F}$ be a field. 
    \textbf{A divided power algebra} \cite[Section 2]{cartan1954puissances} over $\mathbb{F}$ is a commutative associative (non-unital) $\mathbb{F}$-algebra $(A,*)$, where $A$ is the underlying $\mathbb{F}$-vector space and $\ast$ is the $\mathbb{F}$-bilinear multiplication, which comes equipped with a divided power structure, that is, a family of functions $(-)^{[n]}: A \to A$, $a \mapsto a^{[n]}$, indexed by strictly positive integers $n$, such that the following identities hold:  
    \begin{enumerate}[{\bf [dp.1]}]
\item\label{relComlambda} $(\lambda a)^{[n]}=\lambda^na^{[n]}$ for all $a\in A$ and $\lambda\in\mathbb{F}$.
\item\label{relComrepet} $a^{[m]}*a^{[n]}=\binom{m+n}{m}a^{[m+n]}$ for all $a\in A$.
\item\label{relComsomme} $(a+b)^{[n]}=a^{[n]}+\big(\sum_{l=1}^{n-1}a^{[l]}*b^{[n-l]}\big)+b^{[n]}$ for all $a\in A$, $b\in A$.
\item\label{relComunit} $a^{[1]}=a$ for all $a\in A$.
\item\label{relComcomp1} $(a*b)^{[n]}=n!a^{[n]}*b^{[n]}=a^{*n}*b^{[n]}=a^{[n]}*b^{*n}$ for all $a\in A$, $b\in A$.
\item\label{relComcomp2} $(a^{[n]})^{[m]}=\frac{(mn)!}{m!(n!)^m}a^{[mn]}$ for all $a\in A$.
\end{enumerate}
The function $(-)^{[n]}$ is called the $n$-th divided power operation. 
\end{definition}

When the base field $\mathbb{F}$ is of characteristic $0$, the only divided power structure on a commutative associative algebra $(A,*)$ is given by $a^{[n]}=\frac{a^{*n}}{n!}$, which justifies the name ``divided powers''. Therefore, in the characteristic $0$ case, a divided power algebra is simply a commutative associative (non-unital) algebra. However, in general, for non-zero characteristics, the two notions diverge. Examples of divided power algebras include the homology of Eilenberg-MacLane spaces \cite[Section 5 and 8]{cartan1954puissances}, the homotopy of simplicial commutative algebras \cite[Th{\'e}or{\`e}me 1]{cartan1954puissances}, and all Zinbiel algebras (which we review in the next section) \cite[Theorem 3.4]{dokas09}. Furthermore, there exists a notion of free divided power algebras, which we review now.

Let $\mathbb{F}$ be a field. For an $\mathbb{F}$-vector space $V$, define $\mathsf{\Gamma}_n(V)=(V^{\otimes n})^{S(n)} \subseteq V^{\otimes n}$ as the subspace of tensors of length $n$ of $V$ which are fixed under the action of the symmetric group $S(n)$, that is, invariant under all $n$-permutations $\sigma \in S(n)$. Categorically speaking, $\mathsf{\Gamma}_n(V)$ is the joint equalizer of the $n$-permutations. %When compared with $\mathsf{Sym}_n(V)$, $\mathsf{\Gamma}_n(V)$ is the joint equalizer of the $n$-permutations, while $\mathsf{Sym}_n(V)$ is the joint coequalizer of the $n$-permutations. 
Define $\mathsf{\Gamma}(V)$ as follows: 
\[\mathsf{\Gamma}(V)=\bigoplus_{n=1}^{\infty}\mathsf{\Gamma}_n(V)=V\oplus \mathsf{\Gamma}_2(V)\oplus \mathsf{\Gamma}_3(V) \oplus \hdots
\]
The vector space $\mathsf{\Gamma}(V)$ is endowed with a divided power algebra structure, and is the free divided power algebra over $V$ \cite[Section 2]{cartan1954puissances}. Explicitly, the divided power operations and the product are defined on generators $v,w\in V$ by:
\[
v^{[n]}=v^{\otimes n}\qquad v*w=v\otimes w+w\otimes v.
\]
An arbitrary element of $\mathsf{\Gamma}(V)$ can then be expressed as a finite sum of divided power monomials \cite[Section 4]{cartan7relations}, which are elements of the form: 
\[ v_1^{[r_1]}*\hdots* v_n^{[r_n]} \]
for $v_1,\hdots,v_n\in V$, where $\ast$ is the multiplication of $\mathsf{\Gamma}(V)$, and $(-)^{[r_j]}$ are the divided power operations. Note that this decomposition in monomials is not unique in general. Later on, we will define the differential combinator on monomials. In order to check that this combinator is well defined, one can use the explicit form of such a monomial: 
\[ v_1^{[r_1]}*\hdots* v_n^{[r_n]} = \sum_{\sigma\in S(n)/S(r_1,\hdots,r_n)}\sigma (v_1^{\otimes r_1}\otimes\hdots\otimes v_n^{\otimes r_n}),\]
where $S(r_1,\hdots,r_n) = S(r_1)\times\hdots\times S(r_p)$ is the Young subgroup of the symmetric group $S(r_1+\hdots+r_p)$.

Free divided power algebras induce a monad $\mathsf{\Gamma}$ on $\mathbb{F}\text{-}\mathsf{VEC}$. Note that it is sufficient to define the monad structure maps on divided power monomials and then extend by linearity. Define the endofunctor $\mathsf{\Gamma}: \mathbb{F}\text{-}\mathsf{VEC} \to \mathbb{F}\text{-}\mathsf{VEC}$ which sends a $\mathbb{F}$-vector space $V$ to its free divided power algebra $\mathsf{\Gamma}(V)$, and which sends an $\mathsf{F}$-linear map $f:V\to W$ to the $\mathsf{F}$-linear map $\mathsf{\Gamma}(f):\mathsf{\Gamma}(V)\to \mathsf{\Gamma}(W)$ defined on divided powers monomials as follows: 
\[\mathsf{\Gamma}(f)(v_1^{[r_1]}*\hdots* v_p^{[r_n]})=(f(v_1))^{[r_1]}*\hdots* (f(v_n))^{[r_n]}\]
which we then extend by linearity. The monad unit $\eta_V:V\to \mathsf{\Gamma}(V)$ is the injection map of $V$ into $\mathsf{\Gamma}(V)$:
\[\eta_V(v)=v^{[1]}.\]
Note that, with this notation, the zero element of $\Gamma(V)$ will here be denoted by $0^{[1]}$. The monad multiplication ${\mu_V:\mathsf{\Gamma}(\mathsf{\Gamma}(V))\to\mathsf{\Gamma}(V)}$ is defined as follows on divided power monomials of divided power monomials, using \textbf{[dp.5]} and \textbf{[dp.6]}: 
\begin{align*}
   &\mu_V\left((v_{1,1}^{[q_{1,1}]}*\hdots*v_{1,k_1}^{[q_{1,k_1}]})^{[r_1]}*\hdots*(v_{p,1}^{[q_{p,1}]}*\hdots*v_{p,k_p}^{[q_{p,k_p}]})^{[r_p]}\right)\\
   &=~\left(\prod_{i=1}^p\frac{1}{r_i!}\prod_{j=1}^{k_i}\frac{(r_iq_{i,j})!}{q_{i,j}!^{r_i}}\right)v_{1,1}^{[r_1q_{1,1}]}*\hdots*v_{1,k_1}^{[r_1q_{1,k_1}]}*\hdots*v_{p,k_p}^{[r_pq_{p,k_p}]} 
\end{align*}
which we then extend by linearity. 
%As to not overload notation, we will write the monad multiplication as: 
%\[\mu_V\left( \prod_{i=1}^{p} \left( \prod_{j=1}^{k_i} v_{i,j}^{[q_{i,j}]} \right)^{[r_i]} \right)=\left(\prod_{i=1}^p\frac{1}{r_i!}\prod_{j=1}^{k_i}\frac{(r_iq_{i,j})!}{q_{i,j}!^{r_i}}\right)\prod_{i,j=1}^{p,k_i}v_{i,j}^{[r_iq_{i,j}]} \]
Note that the functor $\mathsf{\Gamma}$, and the monad structure we described, can be constructed from the operad of commutative (non-unital) algebras \cite[Proposition 1.2.3]{fresse2000}. Furthermore, note that the algebras of the monad $\mathsf{\Gamma}$ are precisely the dividied power algebras \cite[Section 10, Th{\'e}or{\`e}me 1 and 2]{roby65}. 

%From the point of view of operads, the functor $\mathsf{\Gamma}$ can be obtained from the operad of commutative associative algebras \textit{via} a construction similar to the usual free algebra functor \cite[Section 1.2.2]{fresse2000}, and as such, it is endowed with a monad structure. 

\begin{lemma}\cite[Proposition 1.2.3]{fresse2000}\label{defGamma}
$(\mathsf{\Gamma},\mu,\eta)$ is a monad. 
\end{lemma}

Observe that $\mathsf{\Gamma}$ will not be an algebra modality since $\mathsf{\Gamma}(V)$ is non-unital. Therefore, $\mathsf{\Gamma}$ will provide an example of Cartesian differential comonad which is not induced from a differential category structure. We now define the differential combinator transformation for $\mathsf{\Gamma}$. Define $\partial_V:\mathsf{\Gamma}(V)\to\mathsf{\Gamma}(V\times V)$ as follows on divided power monomials: 
\[ \partial_V(v_1^{[r_1]}*\hdots* v_n^{[r_n]})=\sum_{i=1}^n(v_1,0)^{[r_1]}*\hdots* (v_i,0)^{[r_i-1]} *\hdots* (v_n,0)^{[r_n]} * (0,v_i)^{[1]} \]
which we then extend by linearity. If $r_i=1$, we use the following convention: 
	\[(v_1,0)^{[r_1]}*\hdots* (v_i,0)^{[r_i-1]} *\hdots* (v_n,0)^{[r_n]} * (0,v_i)^{[1]}\!=\!(v_1,0)^{[r_1]}*\hdots* (v_{i-1},0)^{[r_{i-1}]}* (v_{i+1},0)^{[r_{i+1}]} *\hdots* (v_n,0)^{[r_n]} * (0,v_i)^{[1]}\] 
We will see below that $\partial$ corresponds to taking the sum of the partial derivatives of divided power polynomials. Note that a consequence of the lack of a unit in $\mathsf{\Gamma}(V)$ is that $\partial_V$ does not factor through a map $\mathsf{\Gamma}(V) \to \mathsf{\Gamma}(V) \otimes V$ since such a map would be undefined on the divided power monomials of degree 1, $v^{[1]}$. 

\begin{proposition}\label{Gammacomb} $\partial$ is a differential combinator transformation for $(\mathsf{\Gamma},\mu,\eta)$. 
\end{proposition}
\begin{proof} Throughout this proof, it is sufficient to do the calculations on divided power monomials and then extend by linearity. First, clearly $\partial$ is a natural transformation. We now need to show that $\partial$ satisfies the dual of the six axioms {\bf[dc.1]} to {\bf[dc.6]} from Definition \ref{def:cdcomonad}. 
%\begin{align*}
%&\mathsf{\Gamma}(f \times f) \left( \partial_V(v_1^{[r_1]}*\hdots* v_n^{[r_n]}) \right) =~  \mathsf{\Gamma}(f \times f) \left( \sum_{i=1}^n (v_1,0)^{[r_1]}*\hdots* (v_i,0)^{[r_i-1]} *\hdots* (v_n,0)^{[r_n]} * (0,v_i)^{[1]} \right) \\
%%&=~ \sum_{i=1}^n \mathsf{\Gamma}(f \times f) \left( (v_1,0)^{[r_1]}*\hdots* (v_i,0)^{[r_i-1]} *\hdots* (v_n,0)^{[r_n]} * (0,v_i)^{[1]} \right) \\
%%&=~ \sum_{i=1}^n \left((f \times f)(v_1,0) \right)^{[r_1]}*\hdots* \left((f \times f)(v_i,0)\right)^{[r_i-1]} *\hdots* \left((f \times f)(v_n,0)\right)^{[r_n]} * \left((f \times f)(0,v_i)\right)^{[1]} \\
%&=~ \sum_{i=1}^n (f(v_1),f(0))^{[r_1]}*\hdots* (f(v_i),f(0))^{[r_i-1]} *\hdots* (f(v_n),f(0))^{[r_n]} * (f(0),f(v_i))^{[1]} \\
%&=~ \sum_{i=1}^n (f(v_1),0)^{[r_1]}*\hdots* (f(v_i),0)^{[r_i-1]} *\hdots* (f(v_n),0)^{[r_n]} * (0,f(v_i))^{[1]} \\
%&=~ \partial_W \left( f(v_1)^{[r_1]}*\hdots* f(v_n)^{[r_n]}) \right) \\
%&=~ \partial_W \left( \mathsf{\Gamma}(f)\left( v_1^{[r_1]}*\hdots* v_n^{[r_n]} \right)  \right)
%\end{align*}
%So $\mathsf{\Gamma}(f \times f) \circ \partial_V = \partial_W \circ \mathsf{\Gamma}(f)$, therefore $\partial$ is a natural transformation. 

		\begin{enumerate}[{\bf [dc.1]}]
			\item Here we use the fact that multiplication by zero gives zero: 
			\begin{align*}
\mathsf{\Gamma}(\pi_0) \left( \partial_V(v_1^{[r_1]}*\hdots* v_n^{[r_n]}) \right) &=~  \mathsf{\Gamma}(\pi_0) \left( \sum_{i=1}^n (v_1,0)^{[r_1]}*\hdots* (v_i,0)^{[r_i-1]} *\hdots* (v_n,0)^{[r_n]} * (0,v_i)^{[1]} \right) \\
%&=~ \sum_{i=1}^n \mathsf{\Gamma}(\pi_0) \left( (v_1,0)^{[r_1]}*\hdots* (v_i,0)^{[r_i-1]} *\hdots* (v_n,0)^{[r_n]} * (0,v_i)^{[1]} \right) \\
&=~ \sum_{i=1}^n \left(\pi_0(v_1,0) \right)^{[r_1]}*\hdots* \left(\pi_0(v_i,0)\right)^{[r_i-1]} *\hdots* \left(\pi_0(v_n,0)\right)^{[r_n]} * \left(\pi_0(0,v_i)\right)^{[1]} \\
&=~ \sum_{i=1}^n v_1^{[r_1]}*\hdots* v_i^{[r_i-1]} *\hdots* v_n^{[r_n]} * 0^{[1]} \\
&=~ 0^{[1]}
\end{align*}
so $\mathsf{\Gamma}(\pi_0)\circ \partial_V=0$
\item Here we use the fact that, by \textbf{[dp.3]}, $(v+w)^{[1]} = v^{[1]} + w^{[1]}$, and the multilinearity of the multiplication: 
{\footnotesize\begin{align*}
 &\mathsf{\Gamma}(1_V \times \Delta_V) \left( \partial_V(v_1^{[r_1]}*\hdots* v_n^{[r_n]}) \right)=  \mathsf{\Gamma}(1_V \times \Delta_V) \left( \sum_{i=1}^n (v_1,0)^{[r_1]}*\hdots* (v_i,0)^{[r_i-1]} *\hdots* (v_n,0)^{[r_n]} * (0,v_i)^{[1]} \right) \\
%&=~ \sum_{i=1}^n \mathsf{\Gamma}(1_V \times \Delta_V)\left( (v_1,0)^{[r_1]}*\hdots* (v_i,0)^{[r_i-1]} *\hdots* (v_n,0)^{[r_n]} * (0,v_i)^{[1]} \right) \\
%&=~  \sum_{i=1}^n \left((1_V \times \Delta_V)(v_1,0) \right)^{[r_1]}\!*\hdots\!*\! \left((1_V \times \Delta_V)(v_i,0)\right)^{[r_i-1]} \!*\! \hdots \!*\! \left((1_V \times \Delta_V)(v_n,0)\right)^{[r_n]} * \left((1_V \times \Delta_V)(0,v_i)\right)^{[1]} \\
&=~ \sum_{i=1}^n (v_1,\Delta_V(0))^{[r_1]}*\hdots* (v_i,\Delta_V(0))^{[r_i-1]} *\hdots* (v_n,\Delta_V(0))^{[r_n]} * (0,\Delta_V(v_i))^{[1]} \\
&=~ \sum_{i=1}^n (v_1,0,0)^{[r_1]}*\hdots* (v_i,0,0)^{[r_i-1]} *\hdots* (v_n,0,0)^{[r_n]} * (0,v_i,v_i)^{[1]} \\
&=~ \sum_{i=1}^n (v_1,0,0)^{[r_1]}*\hdots* (v_i,0,0)^{[r_i-1]} *\hdots* (v_n,0,0)^{[r_n]} * \left( (0,v_i,0) + (0,0,v_i) \right)^{[1]} \\
%&=~ \sum_{i=1}^n (v_1,0,0)^{[r_1]}*\hdots* (v_i,0,0)^{[r_i-1]} *\hdots* (v_n,0,0)^{[r_n]} * \left( (0,v_i,0)^{[1]} + (0,0,v_i)^{[1]} \right) \\
&=~ \sum_{i=1}^n (v_1,0,0)^{[r_1]}*\hdots* (v_i,0,0)^{[r_i-1]} *\hdots* (v_n,0,0)^{[r_n]} * (0,v_i,0)^{[1]} \\
&+~ \sum_{i=1}^n (v_1,0,0)^{[r_1]}*\hdots* (v_i,0,0)^{[r_i-1]} *\hdots* (v_n,0,0)^{[r_n]} * (0,0,v_i)^{[1]} \\
&=~ \sum_{i=1}^n (v_1,\iota_0(0))^{[r_1]}*\hdots* (v_i,\iota_0(0))^{[r_i-1]} *\hdots* (v_n,\iota_0(0))^{[r_n]} * (0,\iota_0(v_i))^{[1]} \\
&+~ \sum_{i=1}^n (v_1,\iota_1(0))^{[r_1]}*\hdots* (v_i,\iota_1(0))^{[r_i-1]} *\hdots* (v_n,\iota_1(0))^{[r_n]} * (0,\iota_1(v_i))^{[1]} \\
%&=~ \sum_{i=1}^n \left((1_V \times \iota_0)(v_1,0) \right)^{[r_1]}*\hdots* \left((1_V \times \iota_0)(v_i,0)\right)^{[r_i-1]} *\hdots* \left((1_V \times \iota_0)(v_n,0)\right)^{[r_n]} * \left((1_V \times \iota_0)(0,v_i)\right)^{[1]} \\
%&+~ \sum_{i=1}^n \left((1_V \times \iota_1)(v_1,0) \right)^{[r_1]}*\hdots* \left((1_V \times \iota_1)(v_i,0)\right)^{[r_i-1]} *\hdots* \left((1_V \times \iota_1)(v_n,0)\right)^{[r_n]} * \left((1_V \times \iota_1)(0,v_i)\right)^{[1]} \\
&=~ \sum_{i=1}^n \mathsf{\Gamma}(1_V \times \iota_0)\left( (v_1,0)^{[r_1]}*\hdots* (v_i,0)^{[r_i-1]} *\hdots* (v_n,0)^{[r_n]} * (0,v_i)^{[1]} \right) \\
&+~ \sum_{i=1}^n \mathsf{\Gamma}(1_V \times \iota_1)\left( (v_1,0)^{[r_1]}*\hdots* (v_i,0)^{[r_i-1]} *\hdots* (v_n,0)^{[r_n]} * (0,v_i)^{[1]} \right) \\
%&=~ \mathsf{\Gamma}(1_V \times \iota_0) \left( \sum_{i=1}^n (v_1,0)^{[r_1]}*\hdots* (v_i,0)^{[r_i-1]} *\hdots* (v_n,0)^{[r_n]} * (0,v_i)^{[1]} \right) \\
%&+~ \mathsf{\Gamma}(1_V \times \iota_1) \left( \sum_{i=1}^n (v_1,0)^{[r_1]}*\hdots* (v_i,0)^{[r_i-1]} *\hdots* (v_n,0)^{[r_n]} * (0,v_i)^{[1]} \right) \\ 
&=~ \mathsf{\Gamma}(1_V\times\iota_0)\left( \partial_V(v_1^{[r_1]}*\hdots* v_n^{[r_n]}) \right) + \mathsf{\Gamma}(1_V\times\iota_1)\left( \partial_V(v_1^{[r_1]}*\hdots* v_n^{[r_n]}) \right) \\
&=~ \left(\mathsf{\Gamma}(1_V\times\iota_0)+\mathsf{\Gamma}(1_V\times\iota_1)\right) \left( \partial_V(v_1^{[r_1]}*\hdots* v_n^{[r_n]}) \right)
\end{align*}}%
So $\Gamma(V\times\Delta_V)\circ \partial_V =\left(\mathsf{\Gamma}(1_V \times\iota_0)+\mathsf{\Gamma}(1_V \times\iota_1)\right)\circ\partial_V$.
\item This is a straightforward. So $\partial_V\circ\eta_V=\eta_{V\times V}\circ\iota_1$.
%	\begin{align*}
%	\partial_V\left( \eta_V(v)  \right) &=~  \partial_V(v^{[1]}) \\
%	&=~ (0,v)^{[1]} \\
%	&=~ \eta_{V \times V}(0,v)\\
%	&=~  \eta_{V \times V}\left( \iota_1(v) \right)
%	\end{align*}		
%So $\partial_V\circ\eta_V=\eta_{V\times V}\circ\iota_1$.
			\item Carefully using the divided power structure axiom {\bf [dp.2]} when expanding out divided power monomials of divided power monomials, we compute:
{\scriptsize \begin{align*}
  & \mu_{V \times V}\left( \mathsf{\Gamma}\left( \left[\mathsf{\Gamma}(\iota_0), \partial_V\right] \right) \left(  \partial_{\mathsf{\Gamma}(V)}\left((v_{1,1}^{[q_{1,1}]}*\hdots*v_{1,k_1}^{[q_{1,k_1}]})^{[r_1]}*\hdots*(v_{p,1}^{[q_{p,1}]}*\hdots*v_{p,k_p}^{[q_{p,k_p}]})^{[r_p]}\right)  \right)  \right) \\
    &=~  \mu_{V \times V} \left( \mathsf{\Gamma} \left[\mathsf{\Gamma}(\iota_0), \partial_V\right] \left( \sum^{p}_{i=1}   \left( (v_{1,1}^{[q_{1,1}]}*\hdots*v_{1,k_1}^{[q_{1,k_1}]}), 0 \right)^{[r_1]}*\hdots*\left( (v_{i,1}^{[q_{i,1}]}*\hdots*v_{i,k_{i}}^{[q_{i,k_{i}}]}), 0 \right)^{[r_{i} - 1]}* \hdots \right. \right. \\
&~\left. \left. *\left((v_{p,1}^{[q_{p,1}]}*\hdots*v_{p,k_p}^{[q_{p,k_p}]}), 0 \right)^{[r_p]} * \left(0, (v_{i,1}^{[q_{i,1}]}*\hdots*v_{i,k_{i}}^{[q_{i,k_{i}}]}) \right)^{[1]} \right) \right) \\
    &=~  \mu_{V \times V} \left(  \sum^{p}_{i=1}   \left(  \iota_0(v_{1,1})^{[q_{1,1}]}*\hdots*\iota_0(v_{1,k_1})^{[q_{1,k_1}]}  \right)^{[r_1]}*\hdots* \left(  \iota_0(v_{i,1})^{[q_{i,1}]}*\hdots*\iota_0(v_{i,k_{i}})^{[q_{i,k_{i}}]} \right)^{[r_{i} - 1]}* \hdots  \right. \\
&~ \left. *\left( \iota_0(v_{p,1})^{[q_{p,1}]}*\hdots*\iota_0(v_{p,k_p})^{[q_{p,k_p}]} \right)^{[r_p]} * \left( \partial_V\left( v_{i,1}^{[q_{i,1}]}*\hdots*v_{i,k_{i}}^{[q_{i,k_{i}}]} \right)\right)^{[1]}  \right) \\
    %&=~  \mu_{V \times V} \left(  \sum^{p}_{i=1}   \left(  (v_{1,1},0)^{[q_{1,1}]}*\hdots*(v_{1,k_1},0)^{[q_{1,k_1}]}  \right)^{[r_1]}*\hdots* \left(  (v_{i,1},0)^{[q_{i,1}]}*\hdots*(v_{i,k_{i}},0)^{[q_{i,k_{i}}]} \right)^{[r_{i} - 1]}* \hdots  \right. \\
%&~ \left. *\left( (v_{p,1},0)^{[q_{p,1}]}*\hdots*(v_{p,k_p},0)^{[q_{p,k_p}]} \right)^{[r_p]} * \left( \partial_V\left( v_{i,1}^{[q_{i,1}]}*\hdots*v_{i,k_{i}}^{[q_{i,k_{i}}]} \right)\right)^{[1]}  \right) \\ 
    &=~  \mu_{V \times V} \left(  \sum^{p}_{i=1}   \left(  (v_{1,1},0)^{[q_{1,1}]}*\hdots*(v_{1,k_1},0)^{[q_{1,k_1}]}  \right)^{[r_1]}*\hdots* \left(  (v_{i,1},0)^{[q_{i,1}]}*\hdots*(v_{i,k_{i}},0)^{[q_{i,k_{i}}]} \right)^{[r_{i} - 1]}* \hdots  \right. \\
&~ \left. *\left( (v_{p,1},0)^{[q_{p,1}]}*\hdots*(v_{p,k_p},0)^{[q_{p,k_p}]} \right)^{[r_p]} * \right.\\
&~ \left. \left( \sum^{k_{i}}_{j=1}  (v_{i,1},0)^{[q_{i,1}]}*\hdots *(v_{i,j},0)^{[q_{i,j} -1]} * \hdots*(v_{i,k_{i}},0)^{[q_{i,k_{i}}]} * (0, v_{i,j})^{[1]} \right)^{[1]}  \right) \\
&=~ \mu_{V \times V} \left( \sum^{p}_{i=1} \sum^{k_{i}}_{j=1} \left(  (v_{1,1},0)^{[q_{1,1}]}*\hdots*(v_{1,k_1},0)^{[q_{1,k_1}]}  \right)^{[r_1]}*\hdots* \left(  (v_{i,1},0)^{[q_{i,1}]}*\hdots*(v_{i,k_{i}},0)^{[q_{i,k_{i}}]} \right)^{[r_{i} - 1]}* \hdots  \right.  \\
&~ \left. *\left( (v_{p,1},0)^{[q_{p,1}]}*\hdots*(v_{p,k_p},0)^{[q_{p,k_p}]} \right)^{[r_p]} *  \left(  (v_{i,1},0)^{[q_{i,1}]}*\hdots *(v_{i,j},0)^{[q_{i,j} -1]} * \hdots*(v_{i,k_{i}},0)^{[q_{i,k_{i}}]} * (0, v_{i,j})^{[1]} \right)^{[1]} \right) \\
%&=~  \sum^{p}_{i=1} \sum^{k_{i}}_{j=1} \mu_{V \times V} \left(\left(  (v_{1,1},0)^{[q_{1,1}]}*\hdots*(v_{1,k_1},0)^{[q_{1,k_1}]}  \right)^{[r_1]}*\hdots* \left(  (v_{i,1},0)^{[q_{i,1}]}*\hdots*(v_{i,k_{i}},0)^{[q_{i,k_{i}}]} \right)^{[r_{i} - 1]}* \hdots  \right.  \\
%&~ \left. *\left( (v_{p,1},0)^{[q_{p,1}]}*\hdots*(v_{p,k_p},0)^{[q_{p,k_p}]} \right)^{[r_p]} *  \left(  (v_{i,1},0)^{[q_{i,1}]}*\hdots *(v_{i,j},0)^{[q_{i,j} -1]} * \hdots*(v_{i,k_{i}},0)^{[q_{i,k_{i}}]} * (0, v_{i,j})^{[1]} \right)^{[1]} \right) \\
&=~ \sum^{p}_{i=1} \sum^{k_{i}}_{j=1} \left(\prod_{i_0=1}^p\frac{1}{r_{i_0}!}\prod_{j_0=1}^{k_{i_0}}\frac{(r_{i_0}q_{i_0,j_0})!}{q_{i_0,j_0}!^{r_{i_0}}}\right) (v_{1,1},0)^{[r_1q_{1,1}]}*\hdots*(v_{i,j},0)^{[r_{i}q_{i,j} - 1]}*\hdots*(v_{p,k_p},0)^{[r_pq_{p,k_p}]} * (0, v_{i,j})^{[1]} \\
%&=~ \left(\prod_{i_0=1}^p\frac{1}{r_{i_0}!}\prod_{j_0=1}^{k_{i_0}}\frac{(r_{i_0}q_{i_0,j_0})!}{q_{i_0,j_0}!^{r_{i_0}}}\right) \sum^{p}_{i=1} \sum^{k_{i}}_{j=1}  (v_{1,1},0)^{[r_1q_{1,1}]}*\hdots*(v_{i,j},0)^{[r_{i}q_{i,j} - 1]}*\hdots*(v_{p,k_p},0)^{[r_pq_{p,k_p}]} * (0, v_{i,j})^{[1]} \\
%&=~ \left(\prod_{i_0=1}^p\frac{1}{r_{i_0}!}\prod_{j_0=1}^{k_{i_0}}\frac{(r_{i_0}q_{i_0,j_0})!}{q_{i_0,j_0}!^{r_{i_0}}}\right) \partial_V\left(v_{1,1}^{[r_1q_{1,1}]}*\hdots*v_{1,k_1}^{[r_1q_{1,k_1}]}*\hdots*v_{p,k_p}^{[r_pq_{p,k_p}]} \right) \\
&=~ \partial_V\left( \left(\prod_{i_0=1}^p\frac{1}{r_{i_0}!}\prod_{j_0=1}^{k_{i_0}}\frac{(r_{i_0}q_{i_0,j_0})!}{q_{i_0,j_0}!^{r_{i_0}}}\right) v_{1,1}^{[r_1q_{1,1}]}*\hdots*v_{1,k_1}^{[r_1q_{1,k_1}]}*\hdots*v_{p,k_p}^{[r_pq_{p,k_p}]} \right) \\
&=~ \partial_V\left( \mu_V\left((v_{1,1}^{[q_{1,1}]}*\hdots*v_{1,k_1}^{[q_{1,k_1}]})^{[r_1]}*\hdots*(v_{p,1}^{[q_{p,1}]}*\hdots*v_{p,k_p}^{[q_{p,k_p}]})^{[r_p]}\right)  \right)
\end{align*}}%

So $\partial_V\circ\mu_V=\mu_{V \times V}\circ\mathsf{\Gamma}\left[\mathsf{\Gamma}(\iota_0), \partial_V\right]\circ \partial_{\mathsf{\Gamma}(V)}$
	     \end{enumerate}	
For the last two axioms, it will be useful to first compute $\partial_{V \times V} \circ \partial_V$ separately. So we compute: 
{\footnotesize \begin{align*}
   & \partial_{V \times V} \left( \partial_V(v_1^{[r_1]}*\hdots* v_n^{[r_n]}) \right) =~  \partial_{V \times V} \left( \sum_{i=1}^n (v_1,0)^{[r_1]}*\hdots* (v_i,0)^{[r_i-1]} *\hdots* (v_n,0)^{[r_n]} * (0,v_i)^{[1]} \right) \\
%&=~ \sum_{i=1}^n \partial_{V \times V}\left( (v_1,0)^{[r_1]}*\hdots* (v_i,0)^{[r_i-1]} *\hdots* (v_n,0)^{[r_n]} * (0,v_i)^{[1]} \right) \\
&=~ \sum_{i=1}^n \sum_{j=1}^n (v_1,0,0,0)^{[r_1]}*\hdots* (v_j,0,0,0)^{[r_j-1]} *\hdots* (v_i,0,0,0)^{[r_i-1]} *\hdots* (v_n,0,0,0)^{[r_n]} * (0,v_i,0,0)^{[1]} * (0,0,v_j,0)^{[1]}  \\
&+~ \sum_{i=1}^n (v_1,0,0,0)^{[r_1]}*\hdots* (v_i,0,0,0)^{[r_i-1]} *\hdots* (v_n,0,0,0)^{[r_n]} * (0,0,0,v_i)^{[1]} \\
%&=~ \sum_{i=1}^n \sum_{j=1}^n \left( \prod \limits^n_{k=1} (v_k, 0, 0, 0)^{[r_k - \delta_{k,i} - \delta_{k,j}]}\right) * (0,v_i,0,0)^{[1]} * (0,0,v_j,0)^{[1]} \\
%&+ \sum_{i=1}^n (v_1,0,0,0)^{[r_1]}*\hdots* (v_i,0,0,0)^{[r_i-1]} *\hdots* (v_n,0,0,0)^{[r_n]} * (0,0,0,v_i)^{[1]}
\end{align*}}% 
Thus we have that: 
\begin{align*}
    \partial_{V \times V} \left( \partial_V(v_1^{[r_1]}*\hdots* v_n^{[r_n]}) \right) &=~ \sum_{i=1}^n \sum_{j=1}^n \left( \prod \limits^n_{k=1} (v_k, 0, 0, 0)^{[r_k - \delta_{k,i} - \delta_{k,j}]}\right) * (0,v_i,0,0)^{[1]} * (0,0,v_j,0)^{[1]} \\&+ \sum_{i=1}^n (v_1,0,0,0)^{[r_1]}*\hdots* (v_i,0,0,0)^{[r_i-1]} *\hdots* (v_n,0,0,0)^{[r_n]} * (0,0,0,v_i)^{[1]}  
\end{align*}
where $\delta_{x,y} = 0$ if $x \neq y$ and $\delta_{x,y} = 1$ if $x = y$, and we used $\prod$ to denote a $*$-product of a family of elements. Note that there is a slight abuse of notation here in the case that $r_i=1$ and when $i=j=i$, potentially giving $r_i-\delta_{i,i}-\delta_{j,j}=-1$. However, when $r_i=1$, the term $(v_i, 0)^{[r_i-1]}$ vanishes, and so, this is not a problem. 

\begin{enumerate}[{\bf [dc.1]}]
			\setcounter{enumi}{4}
			\item Here we use the fact that the first part of $\partial_{V \times V} \circ \partial_V$ vanishes under $\mathsf{\Gamma}(\pi_0\times\pi_1)$ (since we are multiplying by zero), while the second part becomes $\partial_V$ under $\mathsf{\Gamma}(\pi_0\times\pi_1)$: 
\begin{align*}
&\mathsf{\Gamma}(\pi_0\times\pi_1) \left( \partial_{V \times V} \left( \partial_V(v_1^{[r_1]}*\hdots* v_n^{[r_n]}) \right)\right) \\
&=~ \mathsf{\Gamma}(\pi_0\times\pi_1) \left( \sum_{i=1}^n \sum_{j=1}^n \left( \prod \limits^n_{k=1} (v_k, 0, 0, 0)^{[r_k - \delta_{k,i} - \delta_{k,j}]}\right) * (0,v_i,0,0)^{[1]} * (0,0,v_j,0)^{[1]}\right) \\
&+~\mathsf{\Gamma}(\pi_0\times\pi_1) \left(  \sum_{i=1}^n (v_1,0,0,0)^{[r_1]}*\hdots* (v_i,0,0,0)^{[r_i-1]} *\hdots* (v_n,0,0,0)^{[r_n]} * (0,0,0,v_i)^{[1]}  \right) \\
%&=~  \sum_{i=1}^n \sum_{j=1}^n \mathsf{\Gamma}(\pi_0\times\pi_1) \left(\left( \prod \limits^n_{k=1} (v_k, 0, 0, 0)^{[r_k - \delta_{k,i} - \delta_{k,j}]}\right) * (0,v_i,0,0)^{[1]} * (0,0,v_j,0)^{[1]}\right) \\
%&+~  \sum_{i=1}^n \mathsf{\Gamma}(\pi_0\times\pi_1) \left((v_1,0,0,0)^{[r_1]}*\hdots* (v_i,0,0,0)^{[r_i-1]} *\hdots* (v_n,0,0,0)^{[r_n]} * (0,0,0,v_i)^{[1]}  \right) \\
%&=~  \sum_{i=1}^n \sum_{j=1}^n  \left( \prod \limits^n_{k=1} \left( (\pi_0\times\pi_1)(v_k, 0, 0, 0)\right)^{[r_k - \delta_{k,i} - \delta_{k,j}]}\right) * \left( (\pi_0\times\pi_1)(0,v_i,0,0)\right)^{[1]} * \left( (\pi_0\times\pi_1)(0,0,v_j,0)\right)^{[1]} \\
%&+~  \sum_{i=1}^n  \left( (\pi_0\times\pi_1)(v_1,0,0,0)\right)^{[r_1]}*\hdots* \left( (\pi_0\times\pi_1)(v_i,0,0,0)\right)^{[r_i-1]} *\\
%&~~~~~~\hdots* \left( (\pi_0\times\pi_1)(v_n,0,0,0)\right)^{[r_n]} * \left( (\pi_0\times\pi_1)(0,0,0,v_i)\right)^{[1]}  \\
&=~  \sum_{i=1}^n \sum_{j=1}^n  \left( \prod \limits^n_{k=1} \left(  \pi_0(v_k, 0), \pi_1(0, 0)\right)^{[r_k - \delta_{k,i} - \delta_{k,j}]}\right) * \left( \pi_0(0,v_i) \pi_1(0,0)\right)^{[1]} * \left( \pi_0(0,0), \pi_1(v_j,0)\right)^{[1]} \\
&+~  \sum_{i=1}^n  \left( \pi_0(v_1,0), \pi_1(0,0)\right)^{[r_1]}*\hdots* \left( \pi_0(v_i,0), \pi_1(0,0)\right)^{[r_i-1]} *\\
&~~~~~~\hdots* \left( \pi_0(v_n,0), \pi_1(0,0)\right)^{[r_n]} * \left( \pi_0(0,0), \pi_1(0,v_i)\right)^{[1]}  \\
&=~  \sum_{i=1}^n \sum_{j=1}^n  \left( \prod \limits^n_{k=1} \left(  v_k,0 \right)^{[r_k - \delta_{k,i} - \delta_{k,j}]}\right) * \left( 0,0\right)^{[1]} * \left( 0,0\right)^{[1]} \\
&+~  \sum_{i=1}^n  \left( v_1,0\right)^{[r_1]}*\hdots* \left( v_i,0\right)^{[r_i-1]} *\hdots* \left( v_n,0\right)^{[r_n]} * \left( 0,v_i\right)^{[1]}  \\
&=~ 0^{[1]} + \partial_V(v_1^{[r_1]}*\hdots* v_n^{[r_n]})\\
&=~ \partial_V(v_1^{[r_1]}*\hdots* v_n^{[r_n]})
\end{align*}			
So $\mathsf{\Gamma}(\pi_0\times\pi_1)\circ \partial_{V\times V}\circ \partial_V=\partial_V$.

\item Here we use commutativity of the multiplication, which will essentially swap the order of the $i$ and $j$ index in the first part of $\partial_{V \times V} \circ \partial_V$, while the second part is unchanged: 
\begin{align*}
&\mathsf{\Gamma}(c_V) \left( \partial_{V \times V} \left( \partial_V(v_1^{[r_1]}*\hdots* v_n^{[r_n]}) \right)\right) \\
&=~ \mathsf{\Gamma}(c_V) \left( \sum_{i=1}^n \sum_{j=1}^n \left( \prod \limits^n_{k=1} (v_k, 0, 0, 0)^{[r_k - \delta_{k,i} - \delta_{k,j}]}\right) * (0,v_i,0,0)^{[1]} * (0,0,v_j,0)^{[1]}\right) \\
&+~\mathsf{\Gamma}(c_V) \left(  \sum_{i=1}^n (v_1,0,0,0)^{[r_1]}*\hdots* (v_i,0,0,0)^{[r_i-1]} *\hdots* (v_n,0,0,0)^{[r_n]} * (0,0,0,v_i)^{[1]}  \right) \\
%&=~  \sum_{i=1}^n \sum_{j=1}^n \mathsf{\Gamma}(c_V) \left(\left( \prod \limits^n_{k=1} (v_k, 0, 0, 0)^{[r_k - \delta_{k,i} - \delta_{k,j}]}\right) * (0,v_i,0,0)^{[1]} * (0,0,v_j,0)^{[1]}\right) \\
%&+~  \sum_{i=1}^n \mathsf{\Gamma}(c_V) \left((v_1,0,0,0)^{[r_1]}*\hdots* (v_i,0,0,0)^{[r_i-1]} *\hdots* (v_n,0,0,0)^{[r_n]} * (0,0,0,v_i)^{[1]}  \right) \\
&=~  \sum_{i=1}^n \sum_{j=1}^n  \left( \prod \limits^n_{k=1} \left( c_V(v_k, 0, 0, 0)\right)^{[r_k - \delta_{k,i} - \delta_{k,j}]}\right) * \left( c_V(0,v_i,0,0)\right)^{[1]} * \left( (\pi_0\times\pi_1)(0,0,v_j,0)\right)^{[1]} \\
&+~  \sum_{i=1}^n  \left( c_V(v_1,0,0,0)\right)^{[r_1]}*\hdots* \left( c_V(v_i,0,0,0)\right)^{[r_i-1]} *\hdots* \left( c_V(v_n,0,0,0)\right)^{[r_n]} * \left( c_V(0,0,0,v_i)\right)^{[1]}  \\
 &=~ \sum_{i=1}^n \sum_{j=1}^n \left( \prod \limits^n_{k=1} (v_k, 0, 0, 0)^{[r_k - \delta_{k,i} - \delta_{k,j}]}\right) * (0,0,v_i,0)^{[1]} * (0,v_j,0,0)^{[1]} 
 \\&+ \sum_{i=1}^n (v_1,0,0,0)^{[r_1]}*\hdots* (v_i,0,0,0)^{[r_i-1]} *\hdots* (v_n,0,0,0)^{[r_n]} * (0,0,0,v_i)^{[1]}  \\
% &=~ \sum_{i=1}^n \sum_{j=1}^n \left( \prod \limits^n_{k=1} (v_k, 0, 0, 0)^{[r_k - \delta_{k,i} - \delta_{k,j}]}\right) * (0,v_i,0,0)^{[1]} * (0,0,v_j,0)^{[1]} \\&+ \sum_{i=1}^n (v_1,0,0,0)^{[r_1]}*\hdots* (v_i,0,0,0)^{[r_i-1]} *\hdots* (v_n,0,0,0)^{[r_n]} * (0,0,0,v_i)^{[1]}   
% \partial_{V \times V} \left( \partial_V(v_1^{[r_1]}*\hdots* v_n^{[r_n]}) \right) \\
 &=~  \partial_{V \times V} \left( \partial_V(v_1^{[r_1]}*\hdots* v_n^{[r_n]}) \right)
\end{align*}			    
So $\mathsf{\Gamma}(c_V)\circ \partial_{V\times V}\circ \partial_V=\partial_{V\times V}\circ \partial_V$.
        \end{enumerate}
So we conclude that $\partial$ is a differential combinator transformation. \end{proof}

Thus, $(\mathsf{\Gamma},\mu,\eta,\partial)$ is a coCartesian differential monad and so the opposite of its Kleisli category is a Cartesian differential category (which we summarize in Corollary \ref{cor:DIV} below) which, as we will explain below, captures differentiation of divided power polynomials. The coCartesian differential monad $(\mathsf{\Gamma},\mu,\eta,\partial)$ also comes equipped with a $\mathsf D$-linear counit. Define $\varepsilon_V:\mathsf{\Gamma}(V)\to V$ as follows on divided power monomials: 
	\[\varepsilon_V(v_1^{[r_1]}*\hdots* v_n^{[r_n]})=\begin{cases}
	v_1,&\mbox{if }n=1, r_n=1,\\
	0&\mbox{otherwise.}
	\end{cases}\]
which we extend by linearity. Thus $\varepsilon_V$ picks out the divided power monomials of degree 1, $v^{[1]}$ for all $v \in V$, while mapping the rest to zero. 

	\begin{lemma}
	    $\varepsilon$ is a $\mathsf D$-linear counit of $(\mathsf{\Gamma},\mu,\eta,\partial)$.
	\end{lemma}
	\begin{proof} The proof is straightforward, and so we leave this as an excercise for the reader. 
	\end{proof}
	
Thus the subcategory of $\mathsf{D}$-linear maps of the opposite category of the Kleisli category of $\mathsf{\Gamma}$ is isomorphic to the opposite category of $\mathbb{F}\text{-}\mathsf{VEC}$. Summarizing, we obtain the following statement: 

\begin{corollary}\label{cor:DIV}$(\mathsf{\Gamma}, \mu, \eta, \partial)$ is a Cartesian differential comonad on $\mathbb{F}\text{-}\mathsf{VEC}^{op}$ with $\mathsf{D}$-linear unit $\varepsilon$. Therefore $\mathbb{F}\text{-}\mathsf{VEC}^{op}_\mathsf{\Gamma}$ is a Cartesian differential category and $\mathsf{D}\text{-}\mathsf{lin}\left[ \mathbb{F}\text{-}\mathsf{VEC}^{op}_{\mathsf{\Gamma}} \right] \cong \mathbb{F}\text{-}\mathsf{VEC}^{op}$.
\end{corollary}	
	
	The Kleisli category $\mathbb{F}\text{-}\mathsf{VEC}_{\mathsf{\Gamma}}$ is closely related to the notion of (reduced) divided power polynomials. For a set $X$, we denote by $\mathbb F\lceil X\rceil$ the ring of reduced divided power polynomials over the set $X$, which is by definition the free divided power algebra over the $\mathbb{F}$-vector space with basis $X$ \cite[Section 12]{roby68}. In other words, a reduced divided polynomial with variables in $X$ is an $\mathbb{F}$-linear composition of commutative monomials of the type $x_1^{[k_1]}\hdots x_n^{[k_n]}$ where $x_1,\hdots,x_n$ is a tuple of $n$ different elements of $X$ and $k_1,\hdots,k_n$ are strictly positive integers. By reduced, we mean that these polynomials do not have degree 0 terms. Multiplication is given by concatenation, multilinearity and the relation {\bf[dp.2]} of Definition \ref{defipuisdiv}. Composition of divided polynomials can be deduced from the relations {\bf[dp.1]}, {\bf[dp.3]}, {\bf[dp.5]} and {\bf[dp.6]} of \ref{defipuisdiv}. For example, if $p(x)=x^{[n]}$, and $q(y,z)=y^{[m]}z$, then:
	\[p(q(y,z))=(y^{[m]}z)^{[n]}=n!(y^{[m]})^{[n]}z^{[n]}=\frac{(mn)!}{(m!)^n}y^{[mn]}z^{[n]}.\]
	We can define a notion of formal partial derivation for divided polynomials. For $x\in X$, define the linear map ${\frac{d}{dx}:\mathbb F\lceil X\rceil\!\to \mathbb F\lceil X\rceil\oplus \mathbb{F}}$ on monomials (which we then extend by linearity). 
	 For all monomial $m=x_1^{[k_1]}\hdots x_n^{[k_n]}$, $\frac{d}{dx}(m)=0$ if  $x\neq x_i$ for all $i\in\{1,\hdots,n\}$, $\frac{d}{dx}(m)=x_1^{[k_1]}\hdots x_{j-1}^{[k_{j-1}]}x_{j}^{[k_{j}-1]}x_{j+1}^{[k_{j+1}]}\hdots x_n^{[k_n]}$ if $x=x_j$ and $k_j>1$, $\frac{d}{dx}(m)=x_1^{[k_1]}\hdots x_{j-1}^{[k_{j-1}]}x_{j+1}^{[k_{j+1}]}\hdots x_n^{[k_n]}$ if $x=x_j$, $k_j=1$, and $n>1$, and finally, $\frac{d}{dx}(x)=1_{\mathbb{F}}$ where $1_{\mathbb{F}}\in\mathbb{F}$ is a generator of the second term of the direct sum $\mathbb F\lceil X\rceil\oplus \mathbb{F}$ given by the unit of $\mathbb{F}$. We note that, in the case where $X$ is a singleton, these definitions correspond to the notion of derivation for formal divided power series, also called Hurwitz series, as defined by Keigher and Pritchard in \cite{keigher2000}. We can restrict to the finite dimensional case and obtain a sub-Cartesian differential category of $\mathbb{F}\text{-}\mathsf{VEC}^{op}_{\mathsf{\Gamma}}$ which is isomorphic to the Lawvere theory of reduced divided power polynomials.

	\begin{example} \normalfont \label{ex:CDCdiv} Let $\mathbb{F}$ be a field. Define the category $\mathbb{F}\text{-}\mathsf{DPOLY}$ whose object are $n \in \mathbb{N}$, where a map ${P: n \to m}$ is a $m$-tuple of reduced divided polynomials in $n$ variables, that is, $P=\langle p_1(\vec x),\hdots,p_m(\vec x)\rangle$ with $p_i(\vec x)\in \mathbb F\lceil x_1,\hdots,x_n\rceil$.%, where we allow the notation $\mathbb F\lceil x_1,\hdots,x_n\rceil=\mathbb F\lceil \{x_1,\hdots,x_n\}\rceil$. 
	The identity maps ${1_n: n \to n}$ are the tuples $1_n = \langle x_1^{[1]}, \hdots, x_n^{[1]} \rangle$ and composition is given by divided power polynomial substitution as explained above. $\mathbb{F}\text{-}\mathsf{DPOLY}$ is a Cartesian left additive category where the finite product structure is given by $n \times m = n +m$ with projection maps ${\pi_0: n \times m \to n}$ and ${\pi_1: n \times m \to m}$ defined as the tuples $\pi_0 = \langle x_1^{[1]}, \hdots, x_n^{[1]} \rangle$ and $\pi_1 = \langle x_{n+1}^{[1]}, \hdots, x_{n+m}^{[1]} \rangle$, and where the additive structure is defined coordinate-wise via the standard sum of divided power polynomials. $\mathbb{F}\text{-}\mathsf{DPOLY}$ is also a Cartesian differential category where for a map ${P: n \to m}$, with $P = \langle p_1(\vec x), \hdots, p_m(\vec x) \rangle$, its derivative $\mathsf{D}[P]: n \times n \to m$ is defined as the tuple of the sum of the partial derivatives of the divided power polynomials $p_i(\vec x)$:
\begin{align*}
\mathsf{D}[P](\vec x, \vec y) := \left( \sum \limits^n_{i=1} \frac{dp_1(\vec x)}{d x_i} y_i^{[1]}, \hdots, \sum \limits^n_{i=1} \frac{dp_n(\vec x)}{d x_i} y_i^{[1]} \right) \\
\sum \limits^n_{i=1} \frac{dp_j(\vec x)}{d x_i} y_i^{[1]} \in \mathbb F\lceil x_1,\hdots,x_n,y_1,\hdots,y_n\rceil
\end{align*} 
It is important to note that even if $p_j(\vec x)$ has terms of degree 1, every partial derivative $\frac{dp_j(\vec x)}{d x_i} y_i^{[1]}$ will still be reduced (even if $\frac{dp_j(\vec x)}{d x_i}$ may have a degree 0 term), and thus, the differential combinator $\mathsf{D}$ is indeed well-defined. A map ${P: n \to m}$ is $\mathsf{D}$-linear if it of the form: 
\begin{align*}
P = \left \langle \sum \limits^{n}_{i=0} \lambda_{i,1}x_{i}^{[1]}, \hdots, \sum \limits^{n}_{i=0} \lambda_{i,m}x_{i}^{[1]} \right \rangle && \lambda_{i,j} \in \mathbb{F}
\end{align*}
Thus, $\mathsf{D}\text{-}\mathsf{lin}[\mathbb{F}\text{-}\mathsf{DPOLY}]$ is equivalent to $\mathbb{F}\text{-}\mathsf{LIN}$ (as defined in Example \ref{ex:CDCPOLY}).
\end{example} 

We have the following chain of isomorphisms: 
\begin{align*}
\mathbb{F}\text{-}\mathsf{DPOLY}(n,1) = \mathbb F\lceil x_1,\hdots,x_n\rceil \cong \mathsf{\Gamma}(\mathbb{F}^n) \cong \mathbb{F}\text{-}\mathsf{VEC} \left(\mathbb{F}, \mathsf{\Gamma}(\mathbb{F}^n) \right) = \mathbb{F}\text{-}\mathsf{VEC}_{\mathsf{\Gamma}}\left(\mathbb{F}, \mathbb{F}^n \right) = \mathbb{F}\text{-}\mathsf{VEC}^{op}_{\mathsf{\Gamma}}\left(\mathbb{F}^n, \mathbb{F} \right) 
\end{align*}
which then implies that $\mathbb{F}\text{-}\mathsf{DPOLY}(n,m) \cong \mathbb{F}\text{-}\mathsf{VEC}^{op}_{\mathsf{\Gamma}}\left(\mathbb{F}^n, \mathbb{F}^m \right)$. Therefore, $\mathbb{F}\text{-}\mathsf{DPOLY}$ is indeed isomorphic to the full subcategory of $\mathbb{F}\text{-}\mathsf{VEC}^{op}_{\mathsf{\Gamma}}$ whose objects are the finite dimensional $\mathbb{F}$-vector spaces. In the finite dimensional case, the differential combinator transformation corresponds precisely to the differential combinator on $\mathbb{F}\text{-}\mathsf{DPOLY}$:
\[\partial_{\mathbb{F}^n}(p(\vec x)) = \mathsf{D}[p](\vec x, \vec y)\] 
Thus, $\mathbb{F}\text{-}\mathsf{DPOLY}$ is a sub-Cartesian differential category of $\mathbb{F}\text{-}\mathsf{VEC}^{op}_{\mathsf{P}}$, where the latter allows for divided power polynomials over infinite variables (but will still only depend on a finite number of them). 

\section{Example: Zinbiel Algebras}\label{sec:ZAex}
In this section, we show that the free Zinbiel algebra monad is a coCartesian differential monad, and therefore we construct a Cartesian differential category based on non-commutative polynomials equipped with the half-shuffle product. Zinbiel algebras were introduced by Loday in \cite{loday95}, as Koszul dual to the classical notion of Leibniz algebra. Zinbiel algebras were further studied by Dokas \cite{dokas09}, who shows that they are closely related to divided power algebras. The free Zinbiel algebra is given by the non-unital shuffle algebra. Therefore, this example corresponds to differentiating non-commutative polynomials with a type of polynomial composition defined using the Zinbiel product. Due to the strangeness of the composition, the differential combinator transformation is very different from previous examples. Nevertheless, this is yet another interesting Cartesian differential comonad which does not arise as a differential category. Furthermore, it is worth mentioning that, while the (unital) shuffle algebra has been previously studied in a generalization of differential categories in \cite{bagnol2016shuffle}, the Zinbiel algebra perspective was not considered. In future work, it would be interesting to study the link between these two notions.  

\begin{definition} Let $\mathbb{F}$ be a field. A \textbf{Zinbiel algebra} \cite[Definition 1.2]{loday95} over $\mathbb{F}$, also called \textbf{dual Leibniz algebra}, is an $\mathbb{F}$-vector space $A$ equipped with a bilinear operation $<$ satisfying:
\[((a<b)<c)=(a<(b<c))+(a<(c<b)).\]
for all $a,b,c\in A$.   
\end{definition}

It is important to insist on the fact that the bilinear product $<$ is not assumed to be associative, commutative, or have a distinguished unit element. That said, it is interesting to point out that any Zinbiel algebra is equipped with an associative and commutative bilinear product $*$ defined as $a*b=a<b+b<a$. Thus, a Zinbiel algebra is also a non-unital commutative, associative algebra \cite[Proposition 1.5]{loday95}. The underlying vector space of free Zinbiel algebras is the same as for free non-unital tensor algebras. Readers familiar with the latter will note that the tensor algebra is non-commutative when the multiplication is given by concatenation. However, there is another possible multiplication which is commutative, called the shuffle product. The tensor algebra equipped with the shuffle product is called the shuffle algebra. Furthermore, it turns out that the shuffle product is the commutative associative multiplication $*$ one obtains from the free Zinbiel algebra. Thus, the free Zinbiel algebra and the shuffle algebra are the same object. For the purposes of this paper, we only need to work with the Zinbiel product $<$. 

Let $\mathbb F$ be a field. For an $\mathbb F$-vector space $V$, define $\mathsf{Zin}(V)$ as follows: 
\[\mathsf{Zin}(V)=\bigoplus_{n=1}^{\infty}V^{\otimes n}=V\oplus (V \otimes V) \oplus (V \otimes V \otimes V) \oplus \hdots
\]
$\mathsf{Zin}(V)$ is the free Zinbiel algebra over $V$ \cite[Proposition 1.8]{loday95} with Zinbiel product $<$ defined on pure tensors by:
\[(v_1\otimes\hdots\otimes v_n) < (w_1\otimes\hdots\otimes w_m)=\sum_{\sigma\in S(n+m)/S(n)\times S(m)}v_1\otimes \sigma\cdot(v_2\otimes \hdots\otimes v_n\otimes w_1\otimes \hdots\otimes w_m)\]
which we then extend by linearity. Thus, $\mathsf{Zin}(V)$ is spanned by words of elements of $V$. Free Zinbiel algebras induce a monad $\mathsf{Zin}$ on $\mathbb{F}\text{-}\mathsf{VEC}$. 

Similar to previous examples, note that it is sufficient to define the monad structure maps on pure tensors and then extend by linearity. Define the endofunctor $\mathsf{Zin}: \mathbb{F}\text{-}\mathsf{VEC} \to \mathbb{F}\text{-}\mathsf{VEC}$ which sends an $\mathbb{F}$-vector space $V$ to its free Zinbiel algebra $\mathsf{Zin}(V)$, and which sends an $\mathbb{F}$-linear map $f:V\to W$ to the $\mathbb{F}$-linear map ${\mathsf{Zin}(f):\mathsf{Zin}(V)\to \mathsf{Zin}(W)}$ defined on pure tensors as follows: 
\[\mathsf{Zin}(f)\left( v_0 \otimes \hdots \otimes v_n \right) = f(v_0) \otimes \hdots \otimes f(v_n) \]
which we then extend by linearity. Note that $\mathsf{Zin}(f)$ is a Zinbiel algebra morphism, so we have that:
\[\mathsf{Zin}(f)(\mathfrak{v} < \mathfrak{w}) = \mathsf{Zin}(f)(\mathfrak{v}) < \mathsf{Zin}(f)(\mathfrak{w})\]
for all $\mathfrak{v}, \mathfrak{w} \in \mathsf{Zin}(V)$. The monad unit $\eta_V:V\to \mathsf{Zin}(V)$ is the injection of $V$ into $\mathsf{Zin}(V)$,
\[\eta_V(v)=v,\]
and the monad multiplication $\mu_V:\mathsf{Zin}\mathsf{Zin}(V) \to\mathsf {Zin}(V)$ is defined on pure tensors by taking their Zinbiel product starting from the right. That is, $\mu_V$ is defined on a pure tensor $\mathfrak{v_1} \otimes \hdots\otimes \mathfrak{v_n} \in \mathsf{Zin}\mathsf{Zin}(V)$, where ${\mathfrak{v_1}, \hdots, \mathfrak{v_n} \in \mathsf{Zin}(V)}$, by:
\begin{align*}
   \mu_V\left( \mathfrak{v_1} \otimes \hdots\otimes \mathfrak{v_n} \right)= \mathfrak{v_1} < \left( \hdots \left( \mathfrak{v_{n-1}} < \mathfrak{v_n} \right) \hdots \right)
\end{align*}
which we then extend by linearity. Unsurprisingly, the algebras of the monad $\mathsf{Zin}$ are precisely the Zinbiel algebras. 

\begin{lemma}\cite[Proposition 1.8]{loday95}\label{proploday}
    $(\mathsf{Zin},\mu,\eta)$ is a monad. 
\end{lemma}

It is worth noting the link between divided power algebras and Zinbiel algebra. Any Zinbiel algebra is endowed with a divided power algebra structure \cite[Theorem 3.4]{dokas09}, and this results in an inclusion of the free divided power algebra into the free Zinbiel algebra, $\mathsf{\Gamma}(V)\to \mathsf{Zin}(V)$ \cite[Section 5]{dokas09}. As such, this inclusion can be extended to a monic monad morphism $\mathsf{\Gamma} \Rightarrow \mathsf{Zin}$. Similar to the other examples, due to a lack of unit, $\mathsf{Zin}$ will not be an algebra modality and therefore this will result in another example of a Cartesian differential comonad which does not come from a differential category. 

We can now define the differential combinator transformation for $\mathsf{Zin}$. Define $\partial_V: \mathsf{Zin}(V) \to \mathsf{Zin}(V \times V)$ on pure tensors as follows: 
\[ \partial_V(v_1\otimes v_2 \hdots\otimes v_n)=(0,v_1)\otimes(v_2,0)\otimes\hdots\otimes(v_n,0) \]
which we then extend by linearity. Note that this differential combinator transformation is quite different from the other examples in appearance. Below, we will explain how this differential combinator transformation corresponds to a sum of partial derivative for non-commutative polynomials with the multiplication given by the Zinbiel product. Before proving that $\partial$ is indeed a differential combinator transformation, we prove the following useful identity: 

\begin{lemma}\label{lemmaZin} For all $\mathfrak{v}_1,\mathfrak{v}_2\in\mathsf{Zin}(V)$, the following equality holds: 
	\[
		\partial_V(\mathfrak{v}_1<\mathfrak{v}_2)=\partial_V(\mathfrak{v}_1)<\mathsf{Zin}(\iota_0)(\mathfrak{v}_2),
	\]
\end{lemma}
\begin{proof} It suffices to prove this identity on pure tensors. Assume that $\mathfrak{v}_1=v_0\otimes\hdots\otimes v_n$ and $\mathfrak{v}_2=w_1\otimes\hdots\otimes w_m$. Then:
	\begin{align*}
&\partial_V \left( (v_0\otimes\hdots\otimes v_n) < (w_1\otimes\hdots\otimes w_m) \right)=~\partial_V\left(v_0\otimes\left(\sum_{\sigma\in S(n+m)/S(n)\times S(m)}\sigma(v_1\otimes\hdots\otimes v_n\otimes w_1\otimes\hdots\otimes w_m)\right)\right)\\
		&=~\sum_{\sigma\in S(n+m)/S(n)\times S(m)}(0,v_0)\otimes\sigma\left((v_1,0)\otimes\hdots\otimes(v_p,0)\otimes(w_1,0)\otimes\hdots\otimes(w_r,0)\right)\\
&=~ \left((0,v_0)\otimes (v_2, 0) \hdots\otimes (v_n,0) \right) < \left( (w_1,0) \otimes\hdots\otimes (w_m,0) \right) \\
&=~ \partial_V(v_0\otimes\hdots\otimes v_n)  < \left( \iota_0(w_1) \otimes\hdots\otimes \iota_0(w_m) \right) \\
&=~ \partial_V(v_0\otimes\hdots\otimes v_n)  < \mathsf{Zin}(\iota_0)(w_1\otimes\hdots\otimes w_m)
	\end{align*}
Thus the desired identity holds. 	
\end{proof}

\begin{proposition}\label{difcombzin}
	$\partial$ is a differential combinator transformation for $(\mathsf{Zin},\mu,\eta)$. 
	\end{proposition}
	\begin{proof} It suffices to prove the necessary identities on pure tensors. First, $\partial$ is clearly a natural transformation. We will now show that $\partial$ satisfies the dual of the six axioms {\bf[dc.1]} to {\bf[dc.6]} from Definition \ref{def:cdcomonad}. 
	%We begin by proving the naturality of $\partial$. So given a map $f: V \to W$, we compute: 
%\begin{align*}
%\mathsf{Zin}(f \times f)\left( \partial_V(v_1 \otimes \hdots \otimes v_n) \right) &=~ %\mathsf{Zin}(f \times f)((0,v_1) \otimes (v_2,0) \otimes \hdots \otimes (v_n,0)) \\
%%&=~ (f \times f)(0,v_1) \otimes (f \times f)(v_2,0) \otimes \hdots \otimes (f \times f)(v_n,0) \\
%&=~ (f(0),f(v_1)) \otimes (f(v_2),f(0)) \otimes \hdots \otimes (f(v_n),f(0)) \\
%&=~ (0,f(v_1)) \otimes (f(v_2),0) \otimes \hdots \otimes (f(v_n),0) \\
%&=~ \partial_W\left( f(v_1) \otimes \hdots \otimes f(v_n) \right) \\
%&=~ \partial_W \left( \mathsf{Zin}(f)(v_1 \otimes \hdots \otimes v_n) \right)     
%\end{align*}	
%So $\mathsf{Zin}(f \times f) \circ \partial_V = \partial_W \circ \mathsf{Zin}(f)$, thus $\partial$ is a natural transformation. 

		\begin{enumerate}[{\bf [dc.1]}]
			\item Here we use the fact that tensoring with zero gives zero: 
\begin{align*}
\mathsf{Zin}(\pi_0)\left( \partial_V(v_1 \otimes \hdots \otimes v_n) \right) &=~ \mathsf{Zin}(\pi_0)((0,v_1) \otimes (v_2,0) \otimes \hdots \otimes (v_n,0)) \\
&=~ \pi_0(0,v_1) \otimes \pi_0(v_2,0) \otimes \hdots \otimes \pi_0(v_n,0) \\
&=~ 0 \otimes v_2 \otimes \hdots \otimes v_n \\
&=~ 0
\end{align*}			
so $\mathsf{Zin}(\pi_0)\circ\partial_V=0$.
			\item Here we use the multilinerity of the tensor product: 
\begin{align*}
\mathsf{Zin}(1_V \times \Delta_V)\left( \partial_V(v_1 \otimes \hdots \otimes v_n) \right) &=~ \mathsf{Zin}(1_V \times \Delta_V)((0,v_1) \otimes (v_2,0) \otimes \hdots \otimes (v_n,0)) \\
%&=~ (1_V \times \Delta_V)(0,v_1) \otimes (1_V \times \Delta_V)(v_2,0) \otimes \hdots \otimes (1_V \times \Delta_V)(v_n,0) \\
&=~ (0,\Delta_V(v_1)) \otimes (v_2,\Delta_V(0)) \otimes \hdots \otimes (v_n,\Delta_V(0)) \\
&=~ (0,v_1,v_1) \otimes (v_2,0,0) \otimes \hdots \otimes (v_n,0,0) \\
&=~ \left( (0,v_1,0) + (0,0,v_1) \right) \otimes (v_2,0,0) \otimes \hdots \otimes (v_n,0,0) \\
&=~ (0,v_1,0) \otimes (v_2,0,0) \otimes \hdots \otimes (v_n,0,0) \\
&+~ (0,0,v_1) \otimes (v_2,0,0) \otimes \hdots \otimes (v_n,0,0) \\
&=~ (0,\iota_0(v_1)) \otimes (v_2,\iota_0(0)) \otimes \hdots \otimes (v_n,\iota_0(0)) \\
&+~  (0,\iota_1(v_1)) \otimes (v_2,\iota_1(0)) \otimes \hdots \otimes (v_n,\iota_1(0)) \\
%&=~ (1_V \times \iota_0)(0,v_1) \otimes (1_V \times \iota_0)(v_2,0) \otimes \hdots \otimes (1_V \times \iota_0)(v_n,0) \\
%&+~ (1_V \times \iota_1)(0,v_1) \otimes (1_V \times \iota_1)(v_2,0) \otimes \hdots \otimes (1_V \times \iota_1)(v_n,0) \\
&=~ \mathsf{Zin}(1_V \times \iota_0)((0,v_1) \otimes (v_2,0) \otimes \hdots \otimes (v_n,0)) \\
&+~ \mathsf{Zin}(1_V \times \iota_1)((0,v_1) \otimes (v_2,0) \otimes \hdots \otimes (v_n,0)) \\
&=~ \mathsf{Zin}(1_V \times \iota_0)\left( \partial_V(v_1 \otimes \hdots \otimes v_n) \right) \\
&+~ \mathsf{Zin}(1_V \times \iota_1)\left( \partial_V(v_1 \otimes \hdots \otimes v_n) \right) \\
&=~ \left( \mathsf{Zin}(1_V \times \iota_0) + \mathsf{Zin}(1_V \times \iota_1) \right)\left( \partial_V(v_1 \otimes \hdots \otimes v_n) \right)
\end{align*}			
So $\mathsf{Zin}(1_V\times\Delta_V)\circ\partial_V=(\mathsf{Zin}(1_V\times\iota_0)+\mathsf{Zin}(1_V,\iota_1))\circ\partial_V$.

\item This is a straightforward verification. So $\partial_v\circ\eta_V=\eta_{V\times V}\circ\iota_1$.
%\begin{align*}
%   \partial_V\left( \eta_V(v)  \right) &=~ \partial_V(v) \\
%   &=~ (0,v) \\
%   &=~ \eta_{V \times V}\left( 0,v \right) \\
%   &=~ \eta_{V \times V}\left( \iota_1(v) \right)
%\end{align*}
%So $\partial_v\circ\eta_V=\eta_{V\times V}\circ\iota_1$.
\item Let $\mathfrak{v}_1\otimes\hdots\otimes \mathfrak{v}_k$ be a pure tensor in $\mathsf{Zin}(\mathsf{Zin}(V))$, with $\mathfrak{v}_1,\hdots,\mathfrak{v}_k\in\mathsf{Zin}(V)$. Using Lemma \ref{lemmaZin}, and the fact that $\mathsf{Zin}(\iota_0)$ is a Zinbiel algebra morphism, we compute: 
	\begin{align*}
	\partial_V\left( \mu_V(\mathfrak{v}_1\otimes\hdots\otimes\mathfrak{v}_k) \right)&=~\partial_V(\mathfrak{v}_1<(\hdots<(\mathfrak{v}_k)\hdots))\\
				&=~\partial_V(\mathfrak{v}_1)<\mathsf{Zin}(\iota_0)\left(\mathfrak{v}_2<(\hdots<\mathfrak{v}_n)\hdots)\right)\\	
				%&=~ \partial_V(\mathfrak{v}_1)<(\mathsf{Zin}(\iota_0)(\mathfrak{v}_2)<(\hdots<\mathsf{Zin}(\iota_0)(\mathfrak{v}_n))\hdots) \\
&=~ \mu_{V \times V}\left( \partial_V(\mathfrak{v}_1) \otimes \mathsf{Zin}(\iota_0)(\mathfrak{v}_2) \otimes \hdots \otimes \mathsf{Zin}(\iota_0)(\mathfrak{v}_n) \right) \\
&=~ \mu_{V \times V}\left( \left[\mathsf{Zin}(\iota_0), \partial_V\right](0,\mathfrak{v}_1) \otimes \left[\mathsf{Zin}(\iota_0), \partial_V\right](\mathfrak{v}_2,0) \otimes \hdots \otimes \left[\mathsf{Zin}(\iota_0), \partial_V\right](\mathfrak{v}_n,0) \right) \\
&=~ \mu_{V \times V}\left( \mathsf{Zin}\left( \left[\mathsf{Zin}(\iota_0), \partial_V\right] \right) \left( (0,\mathfrak{v}_1) \otimes (\mathfrak{v}_2,0) \otimes \hdots \otimes (\mathfrak{v}_n,0)
 \right) \right) \\
 &=~ \mu_{V \times V}\left( \mathsf{Zin}\left( \left[\mathsf{Zin}(\iota_0), \partial_V\right] \right) \left( \partial_{\mathsf{Zin}(V)} (\mathfrak{v}_1 \otimes \mathfrak{v}_2 \otimes \hdots \otimes \mathfrak{v}_n)
 \right) \right) 
			\end{align*}
So $\partial_V \circ \mu_V = \mu_{V \times V} \circ  \mathsf{Zin}\left( \left[\mathsf{Zin}(\iota_0), \partial_V\right] \right)\circ  \partial_{\mathsf{Zin}(V)}$.
\item This is straightforward: 
\begin{align*}
&\mathsf{Zin}(\pi_0 \times \pi_1)\left( \partial_{V \times V} \left( \partial_V(v_1 \otimes \hdots \otimes v_n) \right) \right)=~ \mathsf{Zin}(\pi_0 \times \pi_1)\left( \partial_{V \times V}((0,v_1) \otimes (v_2,0) \otimes \hdots \otimes (v_n,0)) \right) \\
&=~ \mathsf{Zin}(\pi_0 \times \pi_1)\left( (0,0,0,v_1) \otimes (v_2,0,0,0) \otimes \hdots \otimes (v_n,0,0,0) \right) \\
%&=~ (\pi_0 \times \pi_1)(0,0,0,v_1) \otimes (\pi_0 \times \pi_1)(v_2,0,0,0) \otimes \hdots \otimes (\pi_0 \times \pi_1)(v_n,0,0,0) \\
&=~ (\pi_0(0,0),\pi_1(0,v_1)) \otimes (\pi_0(v_2,0),\pi_1(0,0)) \otimes \hdots \otimes (\pi_0(v_n,0),\pi_1(0,0)) \\
&=~ (0,v_1) \otimes (v_2,0) \otimes \hdots \otimes (v_n,0) \\
&=~ \partial_V(v_1 \otimes \hdots \otimes v_n)
\end{align*}			
So $\mathsf{Zin}(\pi_0\times\pi_1)\circ\partial_{V\times V}\circ\partial_V=\partial_V$.
\item This is again straightforward: 
\begin{align*}
\mathsf{Zin}(c_V)\left( \partial_{V \times V} \left( \partial_V(v_1 \otimes \hdots \otimes v_n) \right) \right) &=~ \mathsf{Zin}(c_V)\left( \partial_{V \times V}((0,v_1) \otimes (v_2,0) \otimes \hdots \otimes (v_n,0)) \right) \\
&=~ \mathsf{Zin}(c_V)\left( (0,0,0,v_1) \otimes (v_2,0,0,0) \otimes \hdots \otimes (v_n,0,0,0) \right) \\
&=~ c_V(0,0,0,v_1) \otimes c_V(v_2,0,0,0) \otimes \hdots \otimes c_V(v_n,0,0,0) \\
&=~ (0,0,0,v_1) \otimes (v_2,0,0,0) \otimes \hdots \otimes (v_n,0,0,0) \\ 
&=~ \partial_{V \times V}((0,v_1) \otimes (v_2,0) \otimes \hdots \otimes (v_n,0)) \\ 
&=~ \partial_{V \times V} \left( \partial_V(v_1 \otimes \hdots \otimes v_n) \right)
\end{align*}
So $\mathsf{Zin}(c_V)\circ\partial_{V\times V}\circ\partial_V=\partial_{V\times V}\circ\partial_V$.
		\end{enumerate}
So we conclude that $\partial$ is a differential combinator transformation. 		
	\end{proof}
	
Thus, $(\mathsf{Zin},\mu,\eta,\partial)$ is a coCartesian differential monad, and so the opposite of its Kleisli category is a Cartesian differential category (which we summarize in Corollary \ref{cor:ZIN} below). As we will explain below, this Cartesian differential category corresponds to differentiating non-commutative polynomials. The coCartesian differential monad $(\mathsf{Zin},\mu,\eta,\partial)$ also comes equipped with a $\mathsf D$-linear counit. Define $\varepsilon_V:\mathsf{Zin}(V)\to V$ on pure tensors as follows: 
	\[\varepsilon_V(v_1\otimes\hdots\otimes v_n)=\begin{cases}
	v_1&\mbox{if }n=1,\\
	0&\mbox{otherwise.}
	\end{cases}\]
which we extend by linearity. In other words, $\varepsilon_V$ projects out the $V$ component of $\mathsf{Zin}(V)$. 

	\begin{lemma}
	    $\varepsilon$ is a $\mathsf D$-linear counit of $(\mathsf{Zin},\mu,\eta,\partial)$.
	\end{lemma}
	\begin{proof}The proof is straightforward, and so we leave this as an excercise for the reader. 
%We must first show that $\varepsilon$ is a natural transformation. So for a map $f: V \to W$, we compute: 
%\begin{align*}
%    \varepsilon_W\left( \mathsf{Zin}(f)(v_1\otimes\hdots\otimes v_n) \right)&=~ \varepsilon_W(f(v_1)\otimes\hdots\otimes f(v_n))\\
%    &=~\begin{cases}
%	f(v_1)&\mbox{if }n=1,\\
%	0&\mbox{otherwise.}
%	\end{cases}\\
%	&=~f(\varepsilon_V(v_1\otimes\hdots\otimes v_n))
%\end{align*}
%So $\varepsilon$ is a natural transformation. 
%
%    Next, we must show that $\varepsilon$ satisfies the two axioms \textbf{[du.1]} and \textbf{[du.2]} from Definition \ref{Dlincounit}. 
%    \begin{enumerate}[{\bf [du.1]}] 
%\item This is automatic by the definition of $\eta$: 
%\begin{align*}
%    \varepsilon_V\left( \eta_V(v) \right) &=~  \varepsilon_V(v) \\
%    &=~ v
%\end{align*}
%So $\varepsilon_V \circ \eta_V = 1_V$. 
%\item This is mostly straightforward: 
%\begin{align*}
%\eta_V\left( \varepsilon_V(v_1\otimes\hdots\otimes v_n) \right) &=~   \begin{cases}
%	\eta_V(v_1)&\mbox{if }n=1,\\
%	\eta_V(0)&\mbox{otherwise.}
%	\end{cases} \\
%	&=~ 
%	 \begin{cases}
%v_1 &\mbox{if }n=1,\\
%0 &\mbox{otherwise.}
%	\end{cases}  \\
%&=~ 	 v_1\otimes 0^{\otimes n-1}\\
%&=~ \pi_1(0,v_1) \otimes \pi_1(v_2,0) \otimes \hdots \otimes \pi_1(v_n,0) \\
%&=~ \mathsf{Zin}(\pi_1)\left((0,v_1) \otimes (v_2,0) \otimes \hdots \otimes (v_n,0) \right)\\
%&=~ \mathsf{Zin}(\pi_1)\left( \partial_V(v_1\otimes\hdots\otimes v_n) \right)
%\end{align*}
% So $\eta_V \circ \varepsilon_V = \mathsf{Zin}(\pi_1) \circ \partial_V$.
%\end{enumerate}
%So we conclude that $\varepsilon$ is a $\mathsf{D}$-linear counit.	
\end{proof}
	
Thus the subcategory of $\mathsf{D}$-linear maps of the opposite category of the Kleisli category of $\mathsf{Zin}$ is isomorphic to the opposite category of $\mathbb{F}\text{-}\mathsf{VEC}$. Summarizing, we obtain the following statement: 

\begin{corollary}\label{cor:ZIN} $(\mathsf{Zin}, \mu, \eta, \partial)$ is a Cartesian differential comonad on $\mathbb{F}\text{-}\mathsf{VEC}^{op}$ with $\mathsf{D}$-linear unit $\varepsilon$. Therefore, $\mathbb{F}\text{-}\mathsf{VEC}^{op}_\mathsf{Zin}$ is a Cartesian differential category and $\mathsf{D}\text{-}\mathsf{lin}\left[ \mathbb{F}\text{-}\mathsf{VEC}^{op}_{\mathsf{Zin}} \right] \cong \mathbb{F}\text{-}\mathsf{VEC}^{op}$.
\end{corollary}		

The Kleisli category $\mathbb{F}\text{-}\mathsf{VEC}_{\mathsf{Zin}}$ is closely related to non-commutative polynomials. For a set $X$, let $\mathbb F\langle X\rangle$ be the set of non-commutative polynomials and $\mathbb F\langle X\rangle_+$ be the set of reduced non-commutative polynomials, that is, those without any constant terms. As a vector space, $\mathbb F\langle X\rangle_+$ over a set $X$ is isomorphic to the underlying vector space of the free Zinbiel algebra over the free vector space generated by $X$. Thus, to distinguish between polynomials and non-commutative polynomials, we will use the tensor product $\otimes$. For example, $xy=yx$ is the commutative polynomial, while $x \otimes y$ and $y\otimes x$ are two different non-commutative polynomials. Composition in the Kleisli category corresponds to using the Zinbiel product $<$ to define a new kind of substitution of non-commutative polynomials. For example, if $p(x,y)=x\otimes y$ and $q(u,v)=u\otimes v\otimes u$, then
\[q(p(x,y),v)=(x\otimes y)<(v<(x\otimes y))=(x\otimes y)<(v\otimes x\otimes y)=(x\otimes y\otimes v\otimes x\otimes y)+(x\otimes v\otimes y\otimes x\otimes y)+2(x\otimes v\otimes x\otimes y\otimes y).\]
We will use the term Zinbiel polynomials to refer to reduced non-commutative polynomials with the Zinbiel product and the Zinbiel substitution. We are now in a position to define partial derivatives on non-commutative polynomials. For $x \in X$, define ${\frac{d}{dx}:\mathbb F\langle X\rangle\to \mathbb F\langle X\rangle}$ as follows on Zinbiel monomials (which we then extend by linearity): 
\[\frac{d(x_1\otimes x_2\otimes\hdots\otimes x_n)}{dx}=\begin{cases}x_2\otimes\hdots\otimes x_n,&\mbox{if }x_1=x,\\
0,&\mbox{otherwise}.
\end{cases}
\]
and with the convention that $\frac{d(x)}{dx} =1$. 
We can restrict to the finite-dimensional case and obtain a sub-Cartesian differential category of $\mathbb{F}\text{-}\mathsf{VEC}^{op}_{\mathsf{Zin}}$ which is isomorphic to the Lawvere theory of Zinbiel polynomials, and where the differential combinator is defined using their partial derivatives.

\begin{example} \normalfont \label{ex:CDCzin} Let $\mathbb{F}$ be a field. Define the category $\mathbb{F}\text{-}\mathsf{ZIN}$ whose object are natural numbers $n \in \mathbb{N}$, where a map ${P: n \to m}$ is an $m$-tuple of reduced non-commutative polynomials in $n$ variables, that is, $P = \langle {p}_1(\vec x), \hdots, {p}_m(\vec x) \rangle$ with ${p}_i(\vec x) \in \mathbb{F}\langle x_1, \hdots, x_n\rangle_+$. The identity maps $1_n: n \to n$ are the tuples $1_n = \langle x_1, \hdots, x_n \rangle$ and where composition is given by Zinbiel substitution, as defined above. $\mathbb{F}\text{-}\mathsf{ZIN}$ is a Cartesian left additive category where the finite product structure is given by $n \times m = n +m$ with projection maps ${\pi_0: n \times m \to n}$ and ${\pi_1: n \times m \to m}$ defined as the tuples $\pi_0 = \langle x_1, \hdots, x_n \rangle$ and $\pi_1 = \langle x_{n+1}, \hdots, x_{n+m} \rangle$, and where the additive structure is defined coordinate wise via the standard sum of non-commutative polynomials. $\mathbb{F}\text{-}\mathsf{ZIN}$ is also a Cartesian differential category where the differential combinator is given by the differentiation of Zinbiel polynomial given above, that is, for a map ${P: n \to m}$, with $P = \langle p_1(\vec x), \hdots, p_m(\vec x) \rangle$, its derivative $\mathsf{D}[P]: n \times n \to m$ is defined as the tuple of the sum of the partial derivatives of the Zinbiel polynomials $p_i(\vec x)$:
\begin{align*}
\mathsf{D}[P](\vec x, \vec y) := \left( \sum \limits^n_{i=1} y_i\otimes\frac{d{p}_1(\vec x)}{d x_i}, \hdots,\sum \limits^n_{i=1} y_i\otimes\frac{d {p}_n(\vec x)}{d x_i} \right) && \sum \limits^n_{i=1} y_i\otimes\frac{d {p}_j (\vec x)}{d x_i} \in \mathbb{F}\langle x_1, \hdots, x_n, y_1, \hdots, y_n \rangle_+
\end{align*} 
It is important to note that even if $p_i(\vec x)$ has terms of degree 1, every partial derivative $y_i\otimes\frac{d p_j(\vec x)}{d x_i} $ will still be reduced. Indeed, the polynomial $y_i\otimes 1\in \mathbb{F}\langle x_1, \hdots, x_n, y_1, \hdots, y_n \rangle$ is identified with the reduced polynomial $y_i\in \mathbb{F}\langle x_1, \hdots, x_n, y_1, \hdots, y_n \rangle_+$, and so, for example, $y_i\otimes\frac{d(x)}{x}=y_i$.
Thus, the differential combinator $\mathsf{D}$ is indeed well-defined. A map ${P: n \to m}$ is $\mathsf{D}$-linear if it of the form: 
\begin{align*}
P = \left \langle \sum \limits^{n}_{i=0} r_{i,1}x_{i}, \hdots, \sum \limits^{n}_{i=0} r_{i,m}x_{i} \right \rangle && r_{i,j} \in \mathbb{F}
\end{align*}
Thus $\mathsf{D}\text{-}\mathsf{lin}[\mathbb{F}\text{-}\mathsf{ZIN}]$ is equivalent to $\mathbb{F}\text{-}\mathsf{LIN}$ (as defined in Example \ref{ex:CDCPOLY}). We note that this example generalize to the category of Zinbiel polynomials over an arbitrary commutative (semi)ring.
\end{example}

We also have the following chain of isomorphisms: 
\begin{align*}
\mathbb{F}\text{-}\mathsf{ZIN}(n,1) = \mathbb \mathbb{F}\langle x_1, \hdots, x_n, y_1, \hdots, y_n \rangle_+ \cong \mathsf{\Gamma}(\mathbb{F}^n) \cong \mathbb{F}\text{-}\mathsf{VEC} \left(\mathbb{F}, \mathsf{Zin}(\mathbb{F}^n) \right) = \mathbb{F}\text{-}\mathsf{VEC}_{\mathsf{Zin}}\left(\mathbb{F}, \mathbb{F}^n \right) = \mathbb{F}\text{-}\mathsf{VEC}^{op}_{\mathsf{Zin}}\left(\mathbb{F}^n, \mathbb{F} \right) 
\end{align*}
which then implies that $\mathbb{F}\text{-}\mathsf{Zin}(n,m) \cong \mathbb{F}\text{-}\mathsf{VEC}^{op}_{\mathsf{Zin}}\left(\mathbb{F}^n, \mathbb{F}^m \right)$. Therefore, $\mathbb{F}\text{-}\mathsf{Zin}$ is isomorphic to the full subcategory of $\mathbb{F}\text{-}\mathsf{VEC}^{op}_{\mathsf{\Gamma}}$ whose objects are the finite dimensional $\mathbb{F}$-vector spaces. In the finite dimensional case, the differential combinator transformation corresponds precisely to the differential combinator on $\mathbb{F}\text{-}\mathsf{ZIN}$:
\[\partial_{\mathbb{F}^n}(p(\vec x)) = \mathsf{D}[p](\vec x, \vec y)\] 
Thus, $\mathbb{F}\text{-}\mathsf{ZIN}$ is a sub-Cartesian differential category of $\mathbb{F}\text{-}\mathsf{VEC}^{op}_{\mathsf{P}}$, where the latter allows for Zinbiel polynomials over infinite variables (but will still only depend on a finite number of them). 

We conclude this section by noting that the inclusion $\mathsf{\Gamma} \Rightarrow \mathsf{Zin}$ is not compatible with the differential combinators. For instance, let $V$ be the vector space spanned by $x$ and $y$, and let $\p^{\mathsf{\Gamma}}$ and $\p^{\mathsf{Zin}}$ denote the differential combinator transformation for the respective monad. Denote by $p(x,y)=x^{[1]}*y^{[1]}\in \mathsf{\Gamma}(V)$. On one hand, the injection ${\mathsf{\Gamma}(V)\to \mathsf{Zin}(V)}$ identifies $p(x,y)$ to the non-commutative polynomial $p(x,y)=x\otimes y+y\otimes x$, and so, $\p^{\mathsf{Zin}}_V(p)(x,y,x^*,y^*)=x^*\otimes y+y^*\otimes x$. On the other hand, $\p^{\mathsf{\Gamma}}_V(p)(x,y,x^*,y^*)=(x^*)^{[1]}*y^{[1]}+(y^*)^{[1]}*x^{[1]}$, which the injection $\mathsf{\Gamma}(V\times V)\to \mathsf{Zin}(V\times V)$ identifies to the non-commutative polynomial $\p^{\mathsf{\Gamma}}_V(p)(x,y,x^*,y^*)=x^*\otimes y+y\otimes x^*+y^*\otimes x+x\otimes y^*$. 

\section{Future Work}\label{conclusion}

Beyond finding and constructing new interesting examples of Cartesian differential comonads, and therefore also new examples of Cartesian differential categories, there are many other interesting possibilities for future work with Cartesian differential comonads. We conclude this paper by listing some of these potential ideas: 
\begin{enumerate}[{\em (i)}]
    \item In \cite{garner2020cartesian}, it was shown that every Cartesian differential category embeds into the coKleisli category of a differential (storage) category \cite[Theorem 8.7]{garner2020cartesian}. In principle, this already implies that every Cartesian differential category embeds into the coKleisli category of a Cartesian differential comonad. However, Cartesian differential comonads can be defined without the need for a symmetric monoidal structure. Thus, it is reasonable to expect that there is a finer (and possibly simpler) embedding of a Cartesian differential category into the coKleisli category of a Cartesian differential comonad. 
    \item In this paper, we studied the (co)Kleisli categories of (co)Cartesian differential (co)monads. A natural follow-up question to ask is: what can we say about the (co)Eilenberg-Moore categories of (co)Cartesian differential (co)monads? As discussed in \cite{cockett_et_al:LIPIcs:2020:11660}, for differential categories the answer is tangent categories \cite{cockett2014differential}. Indeed, the Eilenberg-Moore category of any codifferential category is always a tangent category \cite[Theorem 22]{cockett_et_al:LIPIcs:2020:11660}, while the coEilenberg-Moore category of a differential (storage) category with sufficient limits is a (representable) tangent category \cite[Theorem 27]{cockett_et_al:LIPIcs:2020:11660}. As such, it is reasonable to expect the same to be true for (co)Cartesian differential (co)monads, that is, that the (co)Eilenberg-Moore category of (co)Cartesian differential (co)monad is a tangent category by generalizing the constructions found in \cite{cockett_et_al:LIPIcs:2020:11660}. 
    \item An important part of the theory of calculus is integration, specifically its relationship to differentiation given by antiderivatives and the Fundamental Theorems of Calculus. Integration and antiderivatives have found their way into the theory of differential categories \cite{cockett-lemay2019, ehrhard2017introduction} and Cartesian differential categories \cite{COCKETT201845}. In future work, it would therefore be of interest to define integration and antiderivatives for (co)Cartesian differential (co)monads. We conjecture that integration in this setting would be captured by an \emph{integral combinator transformation}, which should be a natural transformation of the opposite type of the differential combinator transformation, that is, of type $\int_A: \oc(A) \to \oc(A \times A)$. The axioms of an integral combinator transformation should be analogue to the axioms of an integral combinator \cite[Section 5]{COCKETT201845} in the coKleisli category.  Some of the examples presented in this paper seem to come equipped with an integral combinator transformation. For example, there is a well-established notion of integration for power series which should induce integral combinator transformations in an obvious way. In the case of divided power polynomial, there is a notion of integration in the one-variable case (see \cite{keigher2000} for the integration of formal divided power series in one variable). However, it is unclear to us how integration for multivariable divided power polynomials would be defined, which is necessary if we wish to construct an integral combinator transformation. In the case of Zinbiel algebras, we conjecture that the natural transformation $\int_V: \mathsf{Zin}(V\times V)\to \mathsf{Zin}(V)$ defined as: 
	\[\int (a_{1,0},a_{1,1}) \otimes \hdots \otimes (a_{n,0},a_{n,1})=\sum_{f:\lbrace 1, \hdots, n \rbrace \to\{0,1\}}a_{1,f(1)} \otimes \hdots \otimes a_{n,f(n)}\]
is a candidate for an integral combinator transformation (in the dual sense). In a differential category, one way to build an integration operator is via the notion of antiderivatives \cite[Definition 6.1]{cockett-lemay2019}, which is the assumption that a canonical natural transformation $\mathsf{K}_A: \oc(A) \to \oc(A)$ be a natural isomorphism. Another goal for future work would be to generalize antiderivatives (in the differential category sense) for Cartesian differential comonads. 
\end{enumerate}
In conclusion, there are many potential interesting paths to take for future work with Cartesian differential comonads.

\bibliographystyle{plain}      % mathematics and physical sciences
\bibliography{diffcomonadbib}   % name your BibTeX data base
\vfill
Sacha {\scshape Ikonicoff}, {\scshape PIMS--CNRS} Postdoctoral Associate, University of Calgary\\
e-mail: \url{sacha.ikonicoff@ucalgary.ca} 

~

\noindent Jean-Simon Pacaud {\scshape Lemay}, {\scshape NSERC}  Postdoctoral Fellow, Mount Allison University\\
e-mail: \url{jsplemay@gmail.com}
\end{document}